\documentclass[11pt]{amsart}
\usepackage{amssymb, times}
\usepackage{graphicx}
\usepackage{amsfonts}
\usepackage{amsmath}
\usepackage{amsthm}
\usepackage{fancyhdr}
\usepackage{indentfirst}
\usepackage{hyperref}
\usepackage{comment}
\usepackage{color}
\usepackage{subcaption}
\usepackage{epsfig}
\usepackage{pst-grad} % For gradients
\usepackage{pst-plot} % For axes
\usepackage[space]{grffile} % For spaces in paths
\usepackage{etoolbox} % For spaces in paths
\usepackage{mathrsfs} % For more math fonts
\usepackage{enumitem} % For labels in \enumerate
\usepackage{comment}

\newcommand{\A}{\mathcal{A}}

\newcommand{\Ah}{\mathcal{A}^h}
\newcommand{\al}{\alpha}

\newcommand{\V}{\mathcal{V}}
\newcommand{\R}{\mathbb{R}}
\newcommand{\N}{\mathbb{N}}
\newcommand{\mH}{\mathcal{H}}
\newcommand{\F}{\mathcal{F}}

\newcommand{\bB}{\mathbf{B}}

\newcommand{\C}{\mathcal{C}}
\newcommand{\bC}{\mathbf{C}}
\newcommand{\bd}{\mathbf{d}}

\newcommand{\la}{\lambda}

\newcommand{\La}{\Lambda}

\newcommand{\si}{\sigma}
\newcommand{\Si}{\Sigma}
\newcommand{\de}{\delta}

\newcommand{\ep}{\epsilon}

\newcommand{\Om}{\Omega}
\newcommand{\om}{\omega}

\newcommand{\ga}{\gamma}

\newcommand{\lap}{\triangle}

\newcommand{\mZ}{\mathbb{Z}}
\newcommand{\Z}{\mathcal{Z}}
\newcommand{\mS}{\mathcal{S}}
\newcommand{\f}{\mathbf{f}}

\newcommand{\M}{\mathbf{M}}
\newcommand{\bL}{\mathbf{L}}
\newcommand{\bI}{\mathbf{I}}
\newcommand{\mf}{\mathbf{f}}

\newcommand{\mF}{\mathbf{F}}

\newcommand{\X}{\mathfrak{X}}
\newcommand{\sA}{\mathscr{A}}
\newcommand{\bmu}{\boldsymbol\mu}

\newcommand{\n}{\mathbf{n}}

\newcommand{\bK}{\mathbf{K}}

\newcommand{\rom}[1]{\expandafter\romannumeral #1}
\newcommand{\Rom}[1]{\uppercase\expandafter{\romannumeral #1}}

\newcommand{\dist}{\operatorname{dist}}

\newcommand{\vol}{\operatorname{Vol}}

\newcommand{\Area}{\operatorname{Area}}

\newcommand{\area}{\operatorname{area}}

\newcommand{\Diff}{\operatorname{Diff}}
\newcommand{\Id}{\operatorname{Id}}
\newcommand{\Gr}{\operatorname{Graph}}

\newcommand{\PMC}{\mathcal{P}}
\newcommand{\ind}{\operatorname{index}}
\newcommand{\dmn}{\operatorname{dmn}}

\begin{document}

%\swapnumbers
%\theoremstyle{plain}
\newtheorem{theorem}{Theorem}[section]
\newtheorem*{conjecture}{Conjecture}
\newtheorem*{theoremA}{Theorem A}
\newtheorem*{theoremB}{Theorem B}
\newtheorem*{theoremC}{Theorem C}
\newtheorem{proposition}[theorem]{Proposition}
\newtheorem{corollary}[theorem]{Corollary}

\newtheorem{claim}{Claim}

\theoremstyle{remark}
\newtheorem{remark}[theorem]{Remark}

\theoremstyle{definition}
\newtheorem{definition}[theorem]{Definition}

\theoremstyle{plain}
\newtheorem{lemma}[theorem]{Lemma}

\numberwithin{equation}{section}

%Tile and Author
\title[Multiplicity One Conjecture]{On the Multiplicity One Conjecture in Min-max theory}
%\date{\today}

\author[Xin Zhou]{Xin Zhou}
\address{Department of Mathematics, University of California Santa Barbara, Santa Barbara, CA 93106, USA; and School of Mathematics, Institute for Advanced Study, Princeton, NJ 08540, USA}
\email{zhou@math.ucsb.edu}

\maketitle

\pdfbookmark[0]{}{beg}

\begin{abstract}
We prove that in a closed manifold of dimension between 3 and 7 with a bumpy metric, the min-max minimal hypersurfaces associated with the volume spectrum introduced by Gromov, Guth, Marques-Neves, are two-sided and have multiplicity one. This confirms a conjecture by Marques-Neves. 

%The proof follows an approximation strategy using the min-max theory for hypersurfaces with prescribed mean curvature established by the author with Zhu. A scheme is developed to open the parameter space in the definition of volume spectrum so as to produce sweepouts by boundaries of Caccioppoli sets.
We prove that in a bumpy metric each volume spectrum is realized by the min-max value of certain relative homotopy class of sweepouts of boundaries of Caccioppoli sets. The main result follows by approximating such min-max value using the min-max theory for hypersurfaces with prescribed mean curvature established by the author with Zhu.
\end{abstract}

\setcounter{section}{-1}

%%%%%%%%%%%%%%%%%%%%%%%%%%%%%%%%%%
% Section 0   	 		  Introduction	                      		     %
%%%%%%%%%%%%%%%%%%%%%%%%%%%%%%%%%%
\section{Introduction}
\label{S:intro}

Let $(M^{n+1}, g)$ be a closed orientable Riemannian manifold of dimension $3\leq (n+1)\leq 7$. In \cite{Almgren62}, Almgren proved that the space of mod-2 cycles $\Z_n(M, \mZ_2)$ is weakly homotopic the Eilenberg-MacLane space $K(\mZ_2, 1)=\R\mathbb P^\infty$; (see also \cite{Marques-Neves18} for a simpler proof).  Later, Gromov \cite{Gromov88, Gromov03}, Guth \cite{Guth09}, Marque-Neves \cite{Marques-Neves17} introduced the notion of volume spectrum as a nonlinear version of spectrum for the area functional in $\Z_n(M, \mZ_2)$. In particular, the volume spectrum is a non-decreasing sequence of positive numbers 
\[ 0<\om_1(M, g)\leq \cdots \leq \om_k(M, g)\leq \cdots \to +\infty, \]
which is uniquely determined by the metric $g$ in a given closed manifold $M$.
%In particular, given the generator $\bar\la \in H_1(\Z_n(M, \mZ_2), \mZ_2)$, and $k\in\N$, a $k$-sweepout is a continuous mapping $\Phi: X\to \Z_n(M, \mZ_2)$ defined on a cubical complex $X$ such that $\Phi*(\bar{\la}^k)\neq 0\in H_k(X, \mZ_2)$. The $k$-width is defined as: \[ \om_k= \inf\{ \max_{x\in X}\M(\Phi(x)):\, \Phi: X\to \Z_n(M, \mZ_2) \text{ is a $k$-sweepout}\}. \]

By adapting the celebrated min-max theory developed by Almgren \cite{Almgren65}, Pitts \cite{Pitts81} (for $3\leq (n+1)\leq 6$), and Schoen-Simon \cite{Schoen-Simon81} (for $n+1=7$), Marques-Neves \cite{Marques-Neves17, Marques-Neves16} proved that each $\om_k(M, g)$ is associated with an integral varifold $V_k$ whose support is a disjoint collection of smooth, connected, closed, embedded, minimal hypersurfaces $\{\Si^k_1, \cdots, \Si^k_{l_k}\}$, such that
\begin{equation}
\label{E:min-max equality}
\om_k(M, g)=\sum_{i=1}^{l_k} m^k_i\cdot \Area(\Si^k_i),
\end{equation}
where $\{m^k_1, \cdots, m^k_{l_k}\}\subset \N$ is a set of positive integers, usually called {\em multiplicities}. We refer to \cite{Smith82, Colding-DeLellis03, DeLellis-Tasnady13, Guaraco18, DeLellis-Ramic16, Li-Zhou16, Colding-Minicozzi08b, Zhou10, Zhou17b, Riviere17} for other variants of this theory.

Our main theorem states that all these integer multiplicities are identically equal to one for a bumpy metric. A metric $g$ is called bumpy if every closed immersed minimal hypersurface is non-degenerate. White proved that the set of bumpy metrics is generic in Baire sense \cite{White91, White17}.
\begin{theoremA}
\label{T:thmA}
Given a closed manifold $M^{n+1}$ of dimension $3\leq (n+1)\leq 7$ with a bumpy metric $g$, the min-max minimal hypersurfaces $\{\Si^k_i: k\in \N, i=1, \cdots, l_k\}$ associated with volume spectrum are all two-sided and have multiplicity one and index bounded by $k$. That is $m^k_i=1$ for all $k\in\N$, $1\leq i\leq l_k$,
%\[ m^k_i\equiv 1, \, \text{ for all $k\in\N$, $1\leq i\leq l_k$, and } \sum_{i=1}^{l_k}\ind(\Si^k_i)\leq k. \]
\[ \om_k(M, g)=\sum_{i=1}^{l_k} \Area(\Si^k_i), \, \text{ and } \sum_{i=1}^{l_k}\ind(\Si^k_i)\leq k. \]
\end{theoremA}
\begin{remark}
This solves the {\em Multiplicity One Conjecture} of Marques-Neves \cite[1.2]{Marques-Neves18}; (see also \cite{Marques-Neves16} for an earlier weaker version of this conjecture). We refer to Theorem \ref{T:theorem A} for a more detailed statement of this result. Note that by standard compactness analysis (see \cite{Sharp17}), the same conclusion concerning two-sidedness and multiplicity one also holds true for a metric with positive Ricci curvature. 
\end{remark}
\begin{remark}
This conjecture was proved earlier for 1-parameter min-max constructions under positive Ricci curvature assumption by Marques-Neves \cite{Marques-Neves12}, the author \cite{Zhou15, Zhou17}, and Ketover-Marques-Neves \cite{Ketover-Marques-Neves16}. Later it was fully proved for 1-parameter case by Marques-Neves \cite{Marques-Neves16}. Recently, Chodosh-Mantoulidis \cite{Chodosh-Mantoulidis18} proved this conjecture in dimension three $(n+1)=3$ for the Allen-Cahn setting; (see \cite{Gaspar-Guaraco18} for earlier works along this direction); they also proved that the total index is exactly $k$ for their $k$-min-max solutions when $(n+1)=3$. After our results were poseted, Marques-Neves finished their program and also proved the same optimal index estimates for $3\leq (n+1)\leq 7$ \cite[Addendum]{Marques-Neves18}.
\end{remark}

One motivation of this conjecture is to prove the Yau's conjecture \cite{Yau82} on existence of infinitely many closed minimal surfaces in three manifolds. Combining with the growth estimates of $\{\om_k(M, g)\}$ by Marques-Neves \cite[Theorem 5.1 and 8.1]{Marques-Neves17} and the Frankel Theorem \cite{Frankel66}, we have
%standard compactness theory (see \cite{Sharp17}), we have
\begin{theoremB}
Let $M^{n+1}$ be a closed manifold of dimension $3\leq (n+1)\leq 7$.
\begin{enumerate}[label=(\alph*)]
\item For each bumpy metric $g$, there exists infinitely many smooth, connected, closed, embedded, minimal hypersurfaces.
\item If a metric $g$ has positive Ricci curvature, then there exists a sequence of smooth, connected, closed, embedded, minimal hypersurfaces $\{\Si_k\}_{k\in\N}$, such that 
\[ \Area(\Si_k)\sim k^{\frac{1}{n+1}},\, \text{ as } k\to \infty. \]
\end{enumerate}
\end{theoremB}
\begin{remark}
Result (a) was already known even without the bumpy assumption by combining Marques-Neves \cite{Marques-Neves17} and Song \cite{Song18}. For a set of generic metrics, Irie-Marques-Neves \cite{Irie-Marques-Neves18} and Marques-Neves-Song \cite{Marques-Neves-Song17} proved denseness and equi-distribution for the space of closed embedded minimal hypersurfaces, using the Weyl Law for volume spectrum by Liokumovich-Marques-Neves \cite{Liokumovich-Marques-Neves18}. Their generic set in principle could be much smaller than the set of bumpy metrics. 

Result (b) was also obtained by Chodosh-Mantoulidis \cite{Chodosh-Mantoulidis18} in dimension three $(n+1)=3$. 
\end{remark}

As a direct corollary of the compactness theory (see \cite{Sharp17}), our multiplicity one result also gives a solution to the {\em Weighted Morse Index Bound Conjecture} by Marques-Neves.

\begin{theoremC}
Let $M^{n+1}$ be a closed manifold of dimension $3\leq (n+1)\leq 7$ with an arbitrary metric $g$. In (\ref{E:min-max equality}), we have
\[ \sum_{\Si_i^k: \text{ orientable}} m_i^k \cdot \ind(\Si_i^k) + \sum_{\Si_i^k: \text{ nonorientable}} \frac{m_i^k}{2} \cdot \ind(\Si_i^k) \leq k. \]
\end{theoremC}

\subsection{Sketch of the proof}
The key idea of our proof is to approximate the $\Area$-functional by the weighted $\Ah$-functional used in the prescribing mean curvature (PMC) min-max theory developed by the author with Zhu \cite{Zhou-Zhu18}. Note that the $\Ah$-functional is only defined for boundaries of Caccioppoli sets; see (\ref{E: Ah}). A smooth critical point of $\Ah$ is a hypersurface whose mean curvature is prescribed by the restriction of $h$ to itself. There are two crucial parts in the proof. In the first part, we consider min-max construction of minimal hypersurfaces using sweepouts of boundaries of Caccioppoli sets. We observe that in a bumpy metric if one approximates $\Area$ by a sequence $\{\A^{\ep_k h}\}_{k\in \N}$ where $\{\ep_k\}_{k\in\N}\to 0$, and if $h: M\to \R$ is carefully chosen, then the limit min-max minimal hypersurfaces (of min-max PMC hypersurfaces associated with $\A^{\ep_k h}$) are all two-sided and have multiplicity one; see Theorem \ref{T:multiplicity 1 for sweepouts of boundaries}. In the second part, we show that in a bumpy metric the volume spectrum $\om_k(M, g)$ can be realized by the area of some minimal hypersurfaces coming from min-max constructions using sweepouts of boundaries. We now elaborate the detailed ideas.

To implement the idea in the first part, we generalize the PMC min-max theory in \cite{Zhou-Zhu18} to multi-parameter families using continuous sweepouts. Since the space of Caccioppoli sets $\C(M)$ is contractible, there is no nontrivial free homotopy class to do min-max, so we have to consider relative homotopy class. Heuristically, given a $k$-dimensional parameter space $X$, a subset $Z\subset X$, and a continuous map $\Phi_0: X \to \C(M)$, we can consider its relative $(X, Z)$-homotopy class $\Pi=\Pi(\Phi_0)$ consisting of all maps $\Phi: X\to \C(M)$ that are homotopic to $\Phi_0$ and such that $\Phi\vert_Z \equiv \Phi_0\vert_Z$. If the min-max value $\bL^h=\inf\{\max_{x\in X} \Ah(\Phi(x)): \Phi\in\Pi\}$ satisfies the nontriviality condition $\bL^h>\max_{x\in Z}\Ah(\Phi_0(x))$ with respect to the $\Ah$-functional, and if $h$  is chosen in a dense subset $\mathcal S(g)\subset C^\infty(M)$ (depending on the metric $g$, see \cite[Proposition 0.2]{Zhou-Zhu18}), we prove the existence of a smooth closed hypersurface $\Si^h$ of prescribed mean curvature $h$; moreover, it is represented as the boundary $\Si^h=\partial\Om^h$ for some Caccioppoli set $\Om^h$ and $\Ah(\Om^h)=\bL^h$; hence $\Si^h$ is two-sided and have multiplicity one. $\Si^h$ is usually called a min-max PMC hypersurface. We also established Morse index upper bounds following Marques-Neves \cite{Marques-Neves16}. That is, we prove that the Morse index of $\Si^h$ is bounded from above by $k$ (the dimension of parameter space).  %Note that by the regularity theory in \cite{Zhou-Zhu18}, if $h$ is chosen in a dense subset $\mathcal S(g)\subset C^\infty(M)$ (depending on the metric $g$), the min-max PMC hypersurfaces associated with $\Ah$ are represented as the boundaries of some Caccioppoli sets; hence they are all two-sided and have multiplicity one. 

Given a relative homotopy class $\Pi$ as above, consider the min-max construction for the $\Area$-functional and let $\bL=\inf\{ \max_{x\in X} \Area(\partial\Phi(x)) : \Phi\in \Pi\}$. If the nontriviality condition $\bL>\max_{x\in Z}\Area(\partial\Phi_0(x))$ is satisfied, we can approximate $\bL$ by $\bL^{\ep h}$ for a fixed $h\in \mathcal S(g)$ (to be chosen later) and small enough $\ep>0$. We know that $\ep\cdot h$ also belongs to the dense subset $\mathcal S(g)$. Denote $\Si_\ep$ as the min-max PMC hypersurface associated with $\bL^{\ep h}$. As the family $\{\Si_\ep: \ep>0\}$ have uniformly bounded area and Morse index, we can pick a subsequence $\{\Si_k=\Si_{\ep_k}: \ep_k\to0\}$ that converges as varifolds and also locally smooth and graphically away from finitely many points to some limit minimal hypersurface $\Si_\infty$ with integer multiplicity such that $\Area(\Si_\infty)=\bL$. The limit can be extended to a closed embedded minimal hypersurface $\Si_\infty$ across the bad points, and $\Si_\infty$ also has the same Morse index upper bound. Hence $\Si_\infty$ is a min-max minimal hypersurface associated with $\bL$.  As a standard process, if the multiplicity is greater than one, or if a component is one-sided, one can obtain solutions of the Jacobi operator $L_{\Si_\infty}$ of $\Si_\infty$ by taking the limit of the renormalizations of the heights between the top and bottom sheets of $\Si_k$. In particular, there are two possibilities for the limit depending on the orientations of the top and bottom sheets. For simplicity, let us assume that $\Si_\infty$ is connected and two-sided. An easier case happens when the top and bottom sheets have the same orientation, and hence the limit is a nontrivial nonnegative solution $\varphi$ of the Jacobi equation $L_{\Si_\infty} \varphi =0$ which cannot happen in a bumpy metric. When the top and bottom sheets have opposite orientations, the limit is either a nontrivial nonnegative solution to the Jacobi equation, or is a solution $\varphi$ of the following equation
\[ L_{\Si_\infty} \varphi = 2 h\vert_{\Si_\infty}, \quad \text{such that } \varphi \text{ does not change sign}. \]
The key observation is that one can find a $h\in \mathcal S(g)$ so that the unique solution (as $\Si_\infty$ is non-degenerate) of $L_{\Si_\infty} \varphi = 2 h\vert_{\Si_\infty}$ must change sign, and hence $\Si_\infty$ must have multiplicity one; (see Lemma \ref{L:key lemma}).  Indeed, the set of minimal hypersurfaces with bounded area and Morse index in a bumpy metric is finite by the standard compactness results \cite{Sharp17}. On each such $\Si$, we can construct a $h_\Si \in C^\infty(\Si)$ such that the unique solution $f_\Si$ of $L_\Si f_\Si = 2 h_\Si$ must change sign, and we can further make the support of all such $h_\Si$ pairwise disjoint. Since $\mathcal S(g)$ is open and dense, we can pick a $h\in \mathcal S(g)$ that approximates all $h_\Si$ on $\Si$ as close as we want. Then the solution of $L_\Si \varphi = 2 h\vert_{\Si}$ must also change sign. Up to here, we have elucidated how to construct two-sided min-max minimal hypersurfaces with multiplicity one for sweepouts of boundaries of Caccioppoli sets.

Lastly we apply the above multiplicity one result to the volume spectrum. Though the volume spectrum $\om_k(M, g)$ is defined using cohomological relations, Marques-Neves proved in \cite{Marques-Neves16}, using their Morse index estimates, that in a bumpy metric $\om_k(M, g)$ is realized by the min-max value $\bL(\Pi)$ for certain free homotopy class $\Pi$ of maps $\Phi: X\to \Z_n(M^{n+1}, \mZ_2)$, where $X$ is some fixed $k$-dimensional parameter space and $\Z_n(M^{n+1}, \mZ_2)$ is the space of mod-2 cycles. It was observed by Marques-Neves \cite{Marques-Neves18} that the space of Caccioppoli sets $\C(M)$ forms a double cover of $\Z_n(M^{n+1}, \mZ_2)$ via the boundary map $\partial: \C(M)\to \Z_n(M^{n+1}, \mZ_2)$. Therefore, by lifting to the double cover, for each $\Phi\in \Pi$, we can produce a map $\tilde \Phi: \tilde X\to \C(M)$, where $\pi: \tilde X \to X$ is a double cover, such that $\partial \tilde \Phi (x) = \Phi (\pi(x))$. To produce a nontrivial relative homotopy class, we pick a map $\Phi_0\in \Pi$ such that $\max_{x\in X} \Area(\Phi_0(x))$ is very close to $\bL(\Pi)=\om_k(M, g)$.  Let $Z\subset X$ to be the subset where each $\Phi_0(x)$, $x\in Z$, is $\ep$-distance away from the set of smooth closed embedded minimal hypersurface $\Si$ with $\Area(\Si)\leq \bL$ and $\ind(\Si)\leq k$. Note that this set of minimal hypersurfaces is finite in a bumpy metric, hence for $\ep$ small enough the complement $Y=\overline{X\setminus Z} \subset X$ is topologically trivial in the sense that $Y$ does not detect the generator of the cohomological ring of $\Z_n(M^{n+1}, \mZ_2)$.  Therefore the pre-image $\tilde Y=\pi^{-1}(Y)\subset \tilde X$ is homeomorphic to two disjoint identical copies of $Y$, denoted as $Y^+$ and $Y^-$. On the other hand, since no element in $\Phi_0(Z)$ is regular, by Pitts's combinatorial argument, one can homotopically deform $\Phi_0\vert_Z$ so that $\max_{x\in Z}\Area(\Phi_0(x))<\bL$. Now consider the relative $(\tilde X, \tilde Z)$-homotopy class $\tilde \Pi$ generated by the map $\tilde \Phi_0: \tilde X\to \C(M)$. One key observation is that the min-max value $\bL(\tilde\Pi)\geq \bL(\Pi)>\max_{x\in Z}\Area(\Phi_0(x))$.  To see this, given any homotopic deformation $\tilde \Psi: \tilde X\to \C(M)$ of $\tilde\Phi_0$ relative to $(\tilde \Phi_0)\vert_{\tilde Z}$, if $\max_{x\in Y^+}\Area(\partial \tilde\Psi(x)) <\bL(\Pi)$, then we can pass it to quotient and obtain a continuous map $\Psi: X\to \Z_(M, \mZ_2)$ as $Y^+$ and $Y^-$ are disjoint and $\tilde\Psi\vert_{\tilde Z}\equiv (\tilde \Phi_0)\vert_{\tilde Z}$, so that $\max_{x\in X} \Area(\Psi(x))<\bL(\Pi)$, but this is a contradiction as $\Psi$ is homotopic to $\Phi_0$.  Therefore, $\tilde\Pi$ is a nontrivial relative homotopy class in $\C(M)$, and its associated min-max minimal hypersurfaces are two-sided and have multiplicity one. Finally, as the metric is bumpy, the min-max value $\bL(\tilde\Pi)$ of $\tilde\Pi$ is equal to $\bL(\Pi)$ when $\max_{x\in X} \Area(\Phi_0(x))$ is close enough to $\bL(\Pi)=\om_k(M, g)$. Hence we have explained how to construct two-sided min-max minimal hypersurfaces of multiplicity one whose areas realize the volume spectrum.

\subsection{Outline of the paper} 
In Section \ref{S:Multi-parameter min-max for PMC}, we establish the multi-parameter version of min-max theory for prescribing mean curvature hypersurfaces using continuous sweepouts. In Section \ref{S:compactness}, we prove several compactness results for prescribing mean curvature hypersurfaces with uniform area and Morse index upper bounds. In Section \ref{S:Morse index upper bound}, we prove the Morse index upper bound for prescribing mean curvature hypersurfaces produced by our min-max theory. In Section \ref{S:first multiplicity one result}, we prove that min-max minimal hypersurfaces associated with families of boundaries have multiplicity one in a bumpy metric.  Finally, in Section \ref{S:Application to volume spectrum}, we prove the multiplicity one conjecture for volume spectrum.

\subsection*{Acknowledgements} I would like to thank my Ph. D. advisor Richard Schoen for innumerable advices and long-term encouragement and support. I want to thank Brian White for showing me an unpublished notes and for an enlightening conversation which inspired the key idea of this paper. I also want to thank Fernando Marques and Andre Neves, from whom I learned so many things that are used in this work, and also for their comments. Finally, thanks to Jonathan Zhu for the collaboration on the prescribing mean curvature min-max theory which is essentially used here, and to Zhichao Wang for carefully reading the draft and useful comments to improve the presentation. This work was done when I visited the Institute for Advanced Study, and I would like to thank IAS for their support and hospitality. The work was partially supported by NSF grant DMS-1811293.  

%%%%%%%%%%%%%%%%%%%%%%%%%%%%%%%%%%
% Section 1   	 		  Muti-parameter	                      	   %
%%%%%%%%%%%%%%%%%%%%%%%%%%%%%%%%%%

\section{Multi-parameter min-max theory for prescribing mean curvature hypersurfaces}
\label{S:Multi-parameter min-max for PMC}

Here we present an adaption to multi-parameter families of the min-max theory for hypersurfaces with {\em prescribed mean curvature} (abbreviated as PMC) established by the author with Zhu \cite{Zhou-Zhu17, Zhou-Zhu18}. Let $\mathcal S=\mathcal S(g)$ (depending on the metric $g$) be the open and dense subset of $C^\infty(M)$ chosen as in \cite[Proposition 0.2]{Zhou-Zhu18}. More precisely, $\mathcal S (g)$ consists of all Morse functions $h$ such that the zero set $\Si_0=\{ h=0 \}$ is a smooth closed embedded hypersurface, and the mean curvature of $\Si_0$ vanishes to at most finite order.  A hypersurface is {\em almost embedded} (sometime also called {\em strongly Alexandrov embedded}) if it locally decomposes
into smooth embedded sheets that touch but do not cross. By \cite[Theorem 3.11]{Zhou-Zhu18}, any almost embedded hypersurface of prescribed mean curvature $h\in\mathcal{S}$ has touching set $(n-1)$-rectifiable, and no component is minimal.  

\subsection*{Notations}
\label{SS:Notation and background}

We collect some notions. We refer to \cite{Simon83} and \cite[\S 2.1]{Pitts81} for further materials in geometric measure theory.

Let $(M^{n+1}, g)$ denote a closed, oriented, smooth Riemannian manifold of dimension $3\leq (n+1)\leq 7$. Assume that $(M, g)$ is embedded in some $\R^L$, $L\in\N$. $B_r(p)$ denotes the geodesic ball of $(M, g)$. %$B_r(p), \sB_r(p)$ denote respectively the Euclidean ball of $\R^L$ or the geodesic ball of $(M, g)$. 
We denote by $\mH^k$ the $k$-dimensional Hausdorff measure; $\bI_{k}(M)$ (or $\bI_{k}(M, \mZ_2)$) the space of $k$-dimensional integral (or mod 2) currents in $\R^L$ with support in $M$; $\Z_{k}(M)$ (or $\Z_k(M, \mZ_2)$) the space of integral (or mod 2) currents $T\in\bI_{k}(M)$ with $\partial T=0$; $\V_{k}(M)$ the closure, in the weak topology, of the space of $k$-dimensional rectifiable varifolds in $\R^L$ with support in $M$; $G_k(M)$ the Grassmannian bundle of un-oriented $k$-planes over $M$; $\F$ and $\M$ respectively the flat norm \cite[\S 31]{Simon83} and mass norm \cite[26.4]{Simon83} on $\bI_k(M)$; $\mF$ the varifold $\mF$-metric on $\V_k(M)$ and currents $\mF$-metric on $\bI_k(M)$ or $\bI_{k}(M, \mZ_2)$, \cite[2.1(19)(20)]{Pitts81}; $\C(M)$ or $\C(U)$ the space of sets $\Om\subset M$ or $\Om\subset U\subset M$ with finite perimeter (Caccioppoli sets), \cite[\S 14]{Simon83}\cite[\S 1.6]{Giusti84}; and $\X(M)$ or $\X(U)$ the space of smooth vector fields in $M$ or supported in $U$. $\partial\Om$ denotes the (reduced)-boundary of $[[\Om]]$ as an integral current, and $\nu_{\partial\Om}$ denotes the outward pointing unit normal of $\partial \Om$, \cite[14.2]{Simon83}.

We also utilize the following definitions:
\begin{enumerate}[label=(\alph*), leftmargin=1cm]
\label{En: notations}
\item Given $T\in\bI_{k}(M)$, $|T|$ and $\|T\|$ denote respectively the integral varifold and Radon measure in $M$ associated with $T$;
\item Given $c>0$, a varifold $V\in \V_k(M)$ is said to have {\em $c$-bounded first variation in an open subset $U\subset M$}, if
\[ |\de V(X)|\leq c \int_M|X|d\mu_V, \quad \text{for any } X\in\X(U); \]
here the first variation of $V$ along $X$ is $\de V(X)=\int_{G_k(M)} div_S X(x)d V(x, S)$, \cite[\S 39]{Simon83};
%\item $U_r(V)$ denotes the ball in $\V_k(M)$ under $\mF$-metric with center $V\in\V_k(M)$ and radius $r>0$;
%\item Given $p\in\spt\|V\|$, $\VarTan(V,p)$ denotes the space of tangent varifolds of $V$ at $p$, \cite[42.3]{Simon83};
\item Given a smooth immersed, closed, orientable hypersurface $\Si$ in $M$, or a set $\Om\in\C(M)$ with finite perimeter, $[[\Si]]$, $[[\Om]]$ denote the corresponding integral currents with the natural orientation, and $[\Si]$ denotes the corresponding integer-multiplicity varifold.
%\item $\partial\Om$ denotes the (reduced)-boundary of $[[\Om]]$ as an integral current, and $\nu_{\partial\Om}$ denotes the outward pointing unit normal of $\partial \Om$, \cite[14.2]{Simon83}.
\end{enumerate}

As noted by Marques-Neves \cite[Section 5]{Marques-Neves18}, $\C(M)$ is identified with $\bI_{n+1}(M, \mZ_2)$. In particular, the flat $\F$-norm and the mass $\M$-norm are the same on $\C(M)$. Given $\Om_1, \Om_2\in \C(M)$, the $\mF$-distance between them is:
\[ \mF(\Om_1, \Om_2)=\F(\Om_1-\Om_2)+\mF(|\partial\Om_1|, |\partial\Om_2|). \]
%Note that by \cite[Lemma 7.2]{Zhou17}, if $\M(\Om_2-\Om_1)<\vol(M)/2$, then $\F(\partial\Om_1-\partial\Om_2)=\M(\Om_2-\Om_1)=\M(\Om_1\Delta\Om_2)$.
%Note that the term $\F(\Om_1-\Om_2)=\M(\Om_1-\Om_2)$ is a little stronger than the flat norm distance $\F(\partial\Om_1-\partial\Om_2)$. In fact, by \cite[Lemma 7.2]{Zhou17}, if $\M(\Om_2-\Om_1)<\vol(M)/2$, then $\F(\partial\Om_1-\partial\Om_2)=\M(\Om_2-\Om_1)$. 
%We choose this norm in order to distinguish $\emptyset$ and $M$.
Given $\Om\in\C(M)$, we will denote $\overline{\bB}^\mF_\ep(\Om)=\{\Om'\in \C(M): \mF(\Om', \Om)\leq \ep\}$.

We are interested in the following weighted area functional defined on $\C(M)$. Given $h:M\rightarrow \mathbb{R}$, define the {\em $\Ah$-functional} on $\C(M)$ as
\begin{equation}
\label{E: Ah}
\Ah(\Om)=\mH^n(\partial\Om)-\int_\Omega h\, d\mH^{n+1}. 
\end{equation}
The {\em first variation formula} for $\Ah$ along $X\in \X(M)$ is (see \cite[16.2]{Simon83}) 
\begin{equation}
\label{E: 1st variation for Ah}
\de\Ah\vert_{\Om}(X)=\int_{\partial\Om}div_{\partial \Om}X d\mu_{\partial\Om}-\int_{\partial\Om}h\langle X,\nu\rangle \, d\mu_{\partial\Om},
\end{equation}
where $\nu= \nu_{\partial \Om}$ is the outward unit normal on $\partial \Om$. 

When the boundary $\partial\Om=\Si$ is a smooth immersed hypersurface, we have \[div_{\Si}X=H\langle   X,\nu\rangle,\] where $H$ is the mean curvature of $\Si$ with respect to $\nu$; if $\Om$ is a critical point of $\Ah$, then (\ref{E: 1st variation for Ah}) directly implies that $\Si=\partial \Om$ must have mean curvature $H=h|_\Sigma$. In this case, we can calculate the {\em second variation formula} for $\Ah$ along normal vector fields $X\in \X(M)$ such that $X=\varphi\nu $ along $\partial\Om=\Si$ where $\varphi\in C^\infty(\Si)$, \cite[Proposition 2.5]{BCE88}, 
\begin{equation}
\label{E: 2nd variation for Ah}
\de^2\Ah|_{\Om}(X,X) = \Rom{2}_\Si (\varphi,\varphi) =\int_{\Si}\left( |\nabla\varphi|^2-\left(Ric^M(\nu, \nu)+|A^\Si|^2 + \partial_\nu h\right)\varphi^2\right)d\mu_{\Si}. 
\end{equation}
In the above formula, $\nabla\varphi$ is the gradient of $\varphi$ on $\Si$; $Ric^M$ is the Ricci curvature of $M$; $A^\Si$ is the second fundamental form of $\Si$.

\subsection{Min-max construction for $(X, Z)$-homotopy class}
\label{SS:min-max construction in continuous setting}

In this part, we describe the setup for min-max theory for PMC hypersurfaces associated with multiple parameter families in $\C(M)$.

Let $X^k$ be a cubical complex of dimension $k\in\N$ in some $I^m=[0, 1]^m$ and $Z\subset X$ be a cubical subcomplex. 

Let $\Phi_0: X\to (\C(M), \mF)$ be a continuous map (with respect to the $\mF$-topology on $\C(M)$). We let $\Pi$ be the set of all sequences of continuous (in $\mF$-topology) maps $\{\Phi_i: X\to \C(M)\}_{i\in\N}$ such that :
\begin{enumerate}
\item each $\Phi_i$ is homotopic to $\Phi_0$ in the flat topology on $\C(M)$, and
\item there exist homotopy maps $\{\Psi_i: [0, 1]\times X \to \C(M)\}_{i\in\N}$ which are continuous in the flat topology, $\Psi_i(0, \cdot)=\Phi_i$, $\Psi_i(1, \cdot)=\Phi_0$, and satisfy
\begin{equation}
\label{E:boundary requirement}
\limsup_{i\to\infty} \sup\{\mF(\Psi_i(t, x), \Phi_0(x)): t\in [0, 1], x\in Z\}=0.
\end{equation}
\end{enumerate}
Note that a sequence $\{\Phi_i\}_{i\in\N}$ with $\Phi_i= \Phi_0$ for all $i\in\N$ belongs to $\Pi$. 

\begin{definition}
Given a pair $(X, Z)$ and $\Phi_0$ as above, $\{\Phi_i\}_{i\in\N}$ is called a {\em $(X, Z)$-homotopy sequence of mappings into $\C(M)$}, and $\Pi$ is called the {\em $(X, Z)$-homotopy class of $\Phi_0$}.
\end{definition}

\begin{remark}
$\Pi$ can be viewed as the relative homotopy class for $\Phi_0$ in $(\C(M), \Phi_0 \vert_Z)$. However, we cannot fix the values $\Phi_i\vert_Z$ to be exactly $\Phi_0\vert_Z$. In fact, in the later discretization/interpolation process, we will allow $\Phi_i\vert_Z$ to deviate slightly from $\Phi_0 \vert_Z$; but the deviations will converge to zero as $i\to\infty$.
\end{remark}

\begin{definition}
\label{D:h-width}
The {\em $h$-width} of $\Pi$ is defined by:
\[ \bL^h = \bL^h(\Pi) = \inf_{\{\Phi_i\}\in \Pi}\limsup_{i\to\infty} \sup_{x\in X}\{ \Ah(\Phi_i(x))\}. \]
\end{definition}

\begin{definition}
\label{D:min-max sequence}
A sequence $\{\Phi_i\}_{i\in\N}\in \Pi$ is called a {\em min-max sequence} if
\[ \bL^h(\Phi_i):=\sup_{x\in X} \Ah(\Phi_i(x)) \]
satisfies $\bL^h(\{\Phi_i\}):=\limsup_{i\to\infty} \bL^h(\Phi_i)=\bL^h(\Pi)$.
\end{definition}

\begin{lemma}
Given $\Phi_0$ and $\Pi$, there exists a min-max sequence.
\end{lemma}
\begin{proof}
Take a sequence $\{ \{\Phi^\al_i\}_{i\in\N} \}_{\al\in\N}$ in $\Pi$, such that 
\[ \lim_{\al\to\infty} \bL^h(\{\Phi^\al_i\}_{i\in\N})=\bL^h(\Pi). \]
Now we pick up a new sequence by a diagonalization process. Take a sequence $\ep_\al\to 0$. For each $\al$, we pick $i_\al\in\N$, such that
\[ \sup_{t\in[0, 1], x\in Z}\mF(\Psi^\al_{i_\al}(t, x), \Phi_0(x))<\ep_\al, \text{ and}\]
\[ \bL^h(\{\Phi^\al_i\})-\ep_\al \leq \sup_{x\in X}\Ah(\Phi^\al_{i_\al}(x))\leq \bL^h(\{\Phi^\al_i\})+\ep_\al, \]  
where $\Psi^\al_{i_\al}$ is the homotopy between $\Phi^\al_{i_\al}$ and $\Phi_0$ in the flat topology. 
Hence the sequence $\{\Phi^\al_{i_\al}\}_{\al\in\N}$ belongs to $\Pi$ and is a min-max sequence.
\end{proof}

\begin{definition}
The {\em image set} of $\{\Phi_i\}_{i\in\N}$ is defined by
\[ \bK(\{\Phi_i\})=\{ V=\lim_{j\to\infty}|\partial \Phi_{i_j}(x_j)| \text{ as varifolds}: x_j\in X\}. \]
%\[ \begin{array}{r} \bK(\{\Phi_i\})= \{ V\in \V_n(M): 
%\exists \text{ sequences } \{i_j\}\to\infty, x_{i_j}\in X,\\
%\text{such that } \lim_{j\to\infty}\mF(|\partial \Phi_{i_j}(x_{i_j})|, V)=0\}.
%\end{array}\]

If $\{\Phi_i\}_{i\in\N}$ is a min-max sequence in $\Pi$, the {\em critical set} of $\{\Phi_i\}$ is defined by
\[ \bC(\{\Phi_i\})=\{V=\lim_{j\to\infty}|\partial \Phi_{i_j}(x_j)| \text{ as varifolds}:\, \text{with }\lim_{j\to\infty}\Ah(\Phi_{i_j}(x_j))=\bL^h(\Pi)\}. \]
\end{definition}

\vspace{1em}
Now we are ready to state the continuous version of min-max theory for PMC hypersurfaces associated with a $(X, Z)$-homotopy class. It is a generalization of \cite[Theorem 4.8 and Proposition 7.3]{Zhou-Zhu18}, and the proof is given in Section \ref{SS:proof of min-max theorem}.

\begin{theorem}[Min-max theorem]
\label{T:main min-max theorem}
Let $(M^{n+1}, g)$ be a closed Riemannian manifold of dimension $3\leq (n+1)\leq 7$, and $h\in\mathcal{S}(g)$ which satisfies $\int_M h\geq 0$.  Given a map $\Phi_0: X\to (\C(M), \mF)$ continuous in the $\mF$-topology and the associated $(X, Z)$-homotopy class $\Pi$, suppose
\begin{equation}
\label{E:nontrivial assumption}
\bL^h(\Pi)>\max_{x\in Z}\Ah(\Phi_0(x)).
\end{equation}
Let $\{\Phi_i\}_{i\in\N}\in\Pi$ be a min-max sequence for $\Pi$. Then there exists $V\in \bC(\{\Phi_i\})$ induced by a nontrivial, smooth, closed, almost embedded hypersurface $\Si^n\subset M$ of prescribed mean curvature $h$ with multiplicity one. 

Moreover, $V=\lim_{j\to\infty}|\partial \Phi_{i_j}(x_j)|$ for some $\{i_j\}\subset \{i\}$, $\{x_j\}\subset X\backslash Z$, with $\lim_{j\to\infty} \Ah(\Phi_{i_j}(x_j))=\bL^h(\Pi)$, and $\Phi_{i_j}(x_j)$ converges in the $\mF$-topology to some $\Om\in\C(M)$ such that $\Si=\partial\Om$ where its mean curvature with respect to the unit outer normal is $h$, and  
\[ \Ah(\Om)=\bL^h(\Pi). \]
\end{theorem}

\subsection{Pull-tight}

Now we describe the {\bf pull-tight} process in \cite[Section 5]{Zhou-Zhu18}. Let $c=\sup_M|h|$, and $L^c=2\bL^h+c\vol(M)$. Denote 
\[ A^c_\infty=\{V\in\V_n(M): \|V\|(M)\leq L^c, V \text{ has $c$-bounded first variation, or } V\in |\partial\Phi_0|(Z) \}. \]
We can follow \cite[Section 4]{Zhou-Zhu17} or \cite[Section 5]{Zhou-Zhu18} to construct a continuous map: 
\[ H: [0, 1]\times (\C(M), \mF)\cap\{\M(\partial \Om)\leq L^c\}\to (\C(M), \mF)\cap \{\M(\partial \Om)\leq L^c\} \]
such that:
\begin{itemize}
\item[(\rom{1})] $H(0, \Om)=\Om$ for all $\Om$;
\item[(\rom{2})] $H(t, \Om)=\Om$ if $|\partial\Om|\in A^c_\infty$;
\item[(\rom{3})] if $|\partial\Om|\notin A^c_\infty$,
\[ \Ah(H(1, \Om))-\Ah(\Om)\leq -L(\mF(|\partial\Om|, A^c_\infty))<0; \]
here $L: [0, \infty) \to [0, \infty)$ is a continuous function with $L(0)=0$, $L(t)>0$ when $t>0$;
\item[(\rom{4})] for every $\ep>0$, there exists $\de>0$ such that
\[ x\in Z,\, \mF(\Om, \Phi_0(x))<\de \Longrightarrow \mF(H(t, \Om), \Phi_0(x))<\ep,\, \text{ for all } t\in[0, 1]; \]
this is a direct consequence of (\rom{2}) since $|\partial\Phi_0|(Z)\subset A^c_\infty$.
%Assuming not true, then $\exists \ep>0$ and $\{\de_i\}\to 0$, and $\{\Om_i\}$, $\{x_i\in Z\}$, such that  \[ \mF(\Om_i, \Phi_0(x_i))<\de_i, \quad \mF(H(t_i, \Om_i), \Phi_0(x_i))\geq \ep. \]  Letting $i\to\infty$, $x_i\to x_\infty\in Z$, $\Om_i\to \Phi_0(x_\infty)$, and $H(t_i, \Om_i)\to H(t_\infty, \Phi_0(x_\infty))=\Phi_0(x_\infty)$; but this is a contradiction.
\end{itemize}
Note that to construct $H$, the only modification of \cite[\S 5.1]{Zhou-Zhu18} is to add $|\partial\Phi_0|(Z)$ into the definition of $A^c_\infty$ as we want to fix the values assumed on $Z$ in the tightening process; all other steps in \cite[\S 5.1]{Zhou-Zhu18} carry out the same way. In particular, (using notions in \cite[\S 5.1]{Zhou-Zhu18}), $H(t, \Om):= \big(\Psi_{|\partial \Om|}(t)\big)(\Om)$.

\begin{lemma}
Given a min-max sequence $\{\Phi_i^*\}_{i\in\N}\in\Pi$, we define $\Phi_i(x)=H(1, \Phi^*_i(x))$ for every $x\in X$. Then $\{\Phi_i\}_{i\in\N}$ is also a min-max sequence in $\Pi$. Moreover, $\bC(\{\Phi_i\})\subset \bC(\{\Phi^*_i\})$ and every element of $\bC(\{\Phi_i\})$ either has $c$-bounded first variation, or belongs to $|\partial \Phi_0|(Z)$.
\end{lemma}
%Note that (\ref{E:nontrivial assumption}) is not enough to rule out the case where $V\in \bC(\{\Phi_i\})$ belongs to $|\partial \Phi_0|(\partial X)$. In fact, if $V=|\partial \Phi_0(x)|$ for some $x\in \partial X$, $V$ maybe represented as $M-\Phi_0(x)$. We will rule out this bad case after we establish the regularity of a good min-max varifold. In fact, being $c$-almost minimizing in small annuli implies that $V$ has $c$-bounded first variation away from finitely many points. If $V=|\partial \Phi_0(x)|$, we can simply extend the $c$-bounded 1st variation property across these bad points and obtain almost embededness. Then by \cite[Proposition 7.3]{Zhou-Zhu18}, we proved that $V=\partial \Om$ and $\Ah(\Om)=\bL^h(\Pi)$. 
\begin{proof}
By continuity of $H$, we know that $\Phi_i$ is homotopic to $\Phi^*_i$ in the flat topology. By (\rom{4}), $\{\Psi_i(t, x)=H(t, \Phi_i^*(x))\}$ satisfies (\ref{E:boundary requirement}), and hence $\{\Phi_i\}\in\Pi$. By (\rom{2})(\rom{3}), $\Ah(\Phi_i(x))\leq \Ah(\Phi^*_i(x))$ for every $x\in X$, so $\{\Phi_i\}$ is also a min-max sequence. Finally, given any $V\in \bC(\{\Phi_i\})$, then $V=\lim_{j\to\infty}|\partial \Phi_{i_j}(x_j)|$ where $\lim_{j\to\infty}\Ah(\Phi_{i_j}(x_j))=\bL^h$. Denote $V^*=\lim_{j\to\infty}|\partial \Phi^*_{i_j}(x_j)|$. By (\rom{3}), $\lim_{j\to\infty}\mF(|\partial \Phi^*_{i_j}(x_j)|, A^c_\infty)=0$ (as $\lim_{j\to\infty}\Ah(\Phi_{i_j}(x_j))=\lim_{j\to\infty}\Ah(\Phi^*_{i_j}(x_j))=\bL^h$), so $V^*\in A^c_\infty$. On the other hand,
\[ V= \lim_{j\to\infty}|\partial H(1, \Phi^*_{i_j}(x_j))|=H(1, \lim_{j\to\infty} |\partial\Phi^*_{i_j}(x_j)|)=H(1, V^*)=V^*.\]
(Note that $H$ is also well defined as a continuous map $H: [0, 1]\times \{V\in\V_n(M), \|V\|(M)\leq L^c\} \to \{V\in\V_n(M), \|V\|(M)\leq L^c\}$.) Hence $\bC(\{\Phi_i\})\subset \bC(\{\Phi^*_i\})$ and the proof is finished.
\end{proof} 

\begin{definition}
\label{D:pulled-tight}
Let $c=\sup_M|h|$. Any min-max sequence $\{\Phi_i\}_{i\in\N}\in\Pi$ such that every element of $\bC(\{\Phi_i\})$ has $c$-bounded first variation or belongs to $|\partial\Phi_0|(Z)$ is called {\em pulled-tight}.
\end{definition}

\subsection{Discretization and interpolation results}

We record several discretization and interpolation results developed by Marques-Neves \cite{Marques-Neves14, Marques-Neves17}. Though these results were proven for sweepouts in $\Z_n(M, \mZ)$ or $\Z_n(M, \mZ_2)$, they work well for sweepouts in $\C(M)$. We will point out necessary modifications.

We refer to Appendix \ref{A:cubical complex structures} for the notion of cubic complex structure on $X$. We refer to \cite[Section 4]{Zhou-Zhu18} for the notion of discrete sweepouts. Though all definitions therein were made when $X=[0, 1]$, there is no change for discrete sweepouts on $X$.  
%maybe add more details in the future. %appendix added

Recall that given a map $\phi: X(k)_0\to \C(M)$, the {\em fineness of $\phi$} is defined as
\[ \mf(\phi)=\sup\{ \F(\phi(x)-\phi(y))+\M(\partial\phi(x)-\partial\phi(y)):\, x, y \text{ are adjacent vertices in } X(k)_0 \}. \]

\begin{definition}[c.f. \S 3.7 in \cite{Marques-Neves17}]
Given a continuous (in the flat topology) map $\Phi: X\to\C(M)$, we say that $\Phi$ has {\em no concentration of mass} if
\[ \lim_{r\to 0} \sup\{\|\partial \Phi(x)\|(B_r(p)), p\in M, x\in X\}=0. \]
%Here $B_r(p)$ is the geodesic ball of $M$ centered at $p$ of radius $r>0$.
\end{definition}

The purpose of the next theorem is to construct discrete maps out of a continuous map in flat topology. 

\begin{theorem}%[{\cite[Theorem 13.1]{Marques-Neves14} and \cite[Theorem 3.9]{Marques-Neves17}}]
\label{T:continuous to discrete}
Let $\Phi: X\to \C(M)$ be a continuous map in the flat topology that has no concentration of mass, and $\sup_{x\in X}\M(\partial\Phi(x))<+\infty$. Assume that $\Phi\vert_Z$ is continuous under the $\mF$-topology. Then there exist a sequence of maps
\[ \phi_i: X(k_i)_0 \to \C(M), \]
and a sequence of homotopy maps:
\[ \psi_i: I(k_i)_0 \times X(k_i)_0 \to \C(M), \]
with $k_i<k_{i+1}$, $\psi_i(0, \cdot)=\phi_{i-1}\circ \n(k_i, k_{i-1})$, $\psi_i(1, \cdot)=\phi_i$, and a sequence of numbers $\{\de_i\}_{i\in \N}\to 0$ such that
\begin{itemize}
\item[(\rom{1})] the fineness $\mf(\psi_i)<\de_i$;
\item[(\rom{2})] \[\sup\{\F(\psi_i(t, x)-\Phi(x)): t\in I(k_i)_0, x\in X(k_i)_0\}\leq \de_i; \]
\item[(\rom{3})] for some sequence $l_i\to\infty$, with $l_i<k_i$ %and every $\al\in X(l_i)$,
\[ \M(\partial\psi_i(t, x))\leq \sup\{ \M(\partial\Phi(y)): x, y\in \al, \, \text{ for some } \al\in X(l_i) \}+\de_i; \]
and this directly implies that
\[\sup\{ \M(\partial\phi_i(x)): x\in X(k_0)_0\}\leq \sup\{\M(\partial\Phi(x)): x\in X\}+\de_i.\]
\end{itemize}

As $\Phi\vert_Z$ is continuous in $\mF$-topology, we have from (\rom{3}) that for all $t\in I(k_i)_0$ and $x\in Z(k_i)_0$
\[ \M(\partial \psi_i(t, x))\leq \M(\partial\Phi(x))+\eta_i \]
with $\eta_i\to 0$ as $i\to\infty$. Applying \cite[Lemma 4.1]{Marques-Neves14} with $\mathcal S=\Phi(Z)$, we get by (\rom{2}) that

\begin{itemize}
\item[(\rom{4})] \[ \sup\{ \mF(\psi_i(t, x), \Phi(x)):\, t\in I(k_i)_0, x\in Z(k_i)_0\}\to 0, \text{ as } i\to\infty. \]
\end{itemize}

Now given $h\in C^\infty(M)$, denoting $c=\sup_M|h|$, then we have from (\rom{2})(\rom{3}) that
\begin{itemize}
\item[(\rom{5})] \[ \Ah(\phi_i(x))\leq \sup\{ \Ah(\Phi(y)): \al\in X(l_i), x, y\in \al \}+(1+c)\de_i; \]
and hence
\[\sup\{ \Ah(\phi_i(x)): x\in X(k_i)_0\}\leq \sup\{\Ah(\Phi(x)): x\in X\}+(1+c)\de_i.\]
\end{itemize}
\end{theorem}
\begin{proof}
\cite[Theorem 13.1]{Marques-Neves14} and \cite[Theorem 3.9]{Marques-Neves17} proved this result when $\C(M)$ is replaced by $\Z_n(M)$ and $\Z_n(M, \mZ_2)$ respectively. The adaption to $\C(M)$ was done in \cite[Theorem 5.1]{Zhou17} when $X=[0, 1]$ %(without requiring the no mass concentration assumptions), 
and it is the same for general $X$. 
\end{proof}

The purpose of the next theorem is to construct a continuous map in the $\mF$-topology out of a discrete map with small fineness.

\begin{theorem}
\label{T:discrete to continuous}
There exist some positive constants $C_0=C_0(M, m)$ and $\de_0=\de_0(M, m)$ so that if $Y$ is a cubical subcomplex of $I(m, k)$ and 
\[ \phi: Y_0\to \C(M) \]
has $\mf(\phi)<\de_0$, then there exists a map
\[ \Phi: Y\to \C(M) \]
continuous in the $\mF$-topology and satisfying
\begin{itemize}
\item[(\rom{1})] $\Phi(x)=\phi(x)$ for all $x\in Y_0$;
\item[(\rom{2})] if $\al$ is some $j$-cell in $Y$, then $\Phi$ restricted to $\al$ depends only on the values of $\phi$ restricted on the vertices of $\al$;
\item[(\rom{3})] \[ \sup\{\mF(\Phi(x), \Phi(y)):\, x, y \text{ lie in a common cell of } Y\}\leq C_0\mf(\phi). \]
\end{itemize}
\end{theorem}
\begin{proof}
\cite[Theorem 3.10]{Marques-Neves17} proved this result when $\C(M)$ is replaced by $\Z_n(M, \mZ_2)$. We can use the double cover $\partial: \C(M)\to \Z_n(M, \mZ_2)$ (see \cite[Section 5]{Marques-Neves18}) to lift the extension from $\Z_n(M, \mZ_2)$ to $\C(M)$.

Let $C_0=C_0(M, m)$ and $\de_0=\de_0(M)$ be given in \cite[Theorem 3.10]{Marques-Neves17}. Denote $\tilde{\phi}=\partial\circ \phi: Y_0\to \Z_n(M, \mZ_2)$ as the projection of $\phi$ into $\Z_n(M, \mZ_2)$. Then $\f(\tilde{\phi})<\de_0$, so by \cite[Theorem 3.10]{Marques-Neves17}, there exists a map:
\[ \tilde{\Phi}: Y\to \Z_n(M, \M, \mZ_2) \]
continuous in the $\M$-topology and satisfying
\begin{itemize}
\item[(a)] $\tilde{\Phi}(x)=\tilde{\phi}(x)$ for all $x\in Y_0$;
\item[b)] if $\al$ is some $j$-cell in $Y$, then $\tilde{\Phi}$ restricted to $\al$ depends only on the values of $\tilde{\phi}$ restricted on the vertices of $\al$;
\item[(c)] \[ \sup\{\M(\tilde{\Phi}(x), \tilde{\Phi}(y)):\, x, y \text{ lie in a common cell of } Y\}\leq C_0\mf(\phi). \]
\end{itemize}

By \cite[Claim 5.2]{Marques-Neves18}, $\tilde{\Phi}$ can be uniquely lifted to a continuous map $\Phi: Y \to \C(M)$ such that $\partial\circ \Phi=\tilde{\Phi}$ and $\Phi(x)=\phi(x)$ for all $x\in Y_0$.  In fact, given a $j$-cell $\al$ and a fixed vertex $x_0\in \al_0$, there is a unique lift $\Phi: \al \to \C(M)$ such that $\Phi(x_0)=\phi(x_0)$. By the construction in \cite[Claim 5.2]{Marques-Neves18}, $\F(\Phi(x), \Phi(x_0))=\F(\tilde{\Phi}(x), \tilde{\Phi}(x_0))\leq C_0\f(\phi)$ for every $x\in\al$, so we know by the Constancy Theorem that $\Phi(x)=\phi(x)$ for each vertex $x\in \al_0$ when $\de_0$ is small enough. Thus $\Phi$ can be obtained by lifting $\tilde{\Phi}$ in each cell of $Y$. %more words?

Since $\partial\Phi(x)$ and $\tilde{\Phi}(x)$ represent the same varifold, $\Phi$ is continuous in the $\mF$-topology. So we have proved (\rom{1})(\rom{2}).

For (\rom{3}), we have
\[ \mF(\Phi(x), \Phi(y))=\F(\Phi(x), \Phi(y))+\mF(|\partial \Phi(x)|, |\partial \Phi(y)|) \leq 2 C_0 \mf(\phi). \]
\end{proof}
\begin{remark}
Note that in general the mass of $\partial\Phi(x)-\partial\Phi(y)$ as element in $\Z_n(M)$ may not be equal to that of $\tilde{\Phi}(x)-\tilde{\Phi}(y)$, 
%given $\Om_1, \Om_2\in \C(M)$, $\partial \Om_2-\partial \Om_1$ may not be the same as $[\partial\Om_2]-[\partial \Om_1]$
so we may not be able to prove the $\M$-continuity for $\Phi$.  
\end{remark}

Following \cite[3.10]{Marques-Neves17}, we call the map $\Phi$ given in Theorem \ref{T:discrete to continuous} the {\em Almgren extension} of $\phi$. We will record a few properties concerning the homotopy equivalence of Almgren's extensions.

Before stating the next result, we first recall the notion of homotopic equivalence between discrete sweepouts. Let $Y$ be a cubical subcomplex of $I(m, k)$. Given two discrete maps $\phi_i: Y(l_i)_0 \to \C(M)$, we say {\em $\phi_1$ is homotopic to $\phi_2$ with fineness less than $\eta$}, if there exist $l\in \N$, $l>l_1, l_2$ and a map
\[ \psi: I(1, k+l)_0\times Y(l)_0 \to \C(M) \]
with fineness $\mf(\psi)<\eta$ and such that 
\[ \psi([i-1], y)=\phi_i(\n(k+l, k+l_i)(y)),\, i=1, 2, y\in Y(l)_0. \]

%The following result follows in the same way as \cite[Proposition 3.11]{Marques-Neves17} if we use Theorem \ref{T:discrete to continuous} and Proposition \ref{P:Almgren homotopy} in place of \cite[Theorem 3.10 and Proposition 3.5]{Marques-Neves17}.
The following result is analogous to \cite[Proposition 3.11]{Marques-Neves17}. We provide a lightly different proof.

\begin{proposition}
\label{P:Almgren extensions are homotopic}
With $\phi_1, \phi_2$ as above, if $\eta<\de_0(M, m)$ in Theorem \ref{T:discrete to continuous}, then the Almgren extensions 
\[ \Phi_1, \Phi_2: Y \to \C(M)\]
of $\phi_1, \phi_2$, respectively, are homotopic to each other in the $\mF$-topology.
\end{proposition}
%The proof is as follows: set $\eta=\min\{\de_0, \de/(2C_0)\}$ where $\de, \de_0, C_0$ are given in Proposition \ref{P:Almgren homotopy} and Theorem \ref{T:discrete to continuous}.  The Almgren extension $\Psi: I\times Y\to \C(M)$ is continuous in $\mF$-topology and is a homotopy between the Almgren extensions $\Phi_1', \Phi_2'$ of $\phi_i': Y(l)_0\to \C(M)$ (given by $\phi_i'(y)=\psi([i-1], y)$). By Theorem \ref{T:discrete to continuous}(\rom{3}), $\mF(\Phi_i'(x), \phi'_i(y))\leq C_0\eta<\frac{\de}{2}$ for any cell $\al\in Y(l)$ with $x\in \al$ and $y\in \al_0$. Also the Almgren extensions $\Phi_i$ of $\phi_i$ satisfy $\mF(\Phi_i(x), \phi_i(y))\leq C_0\eta<\frac{\de}{2}$ for any cell $\beta\in Y$ with $x\in \beta$ and $y\in \beta_0$. Therefore, $\mF(\Phi_i'(y), \Phi_i(y))<\de$ for all $y\in Y$. Hence $\Phi_i$ is homotopic to $\Phi_i'$ in the flat topology, and hence are $\Phi_1$ with $\Phi_2$.
\begin{proof}
By Theorem \ref{T:discrete to continuous}, the Almgren extension $\Psi: I\times Y\to \C(M)$ of $\psi$ is continuous in $\mF$-topology and is a homotopy between the Almgren extensions $\Phi_1', \Phi_2'$ of $\phi_1', \phi_2': Y(l)_0\to \C(M)$ (given by $\phi_i'(y)=\psi([i-1], y)$). Note that $\Phi_i'$ is just a reparametrization of the Almgren extension $\Phi_i$ of $\phi_i$ for $i=1, 2$ respectively, so $\Phi_i$ is homotopic to $\Phi_i'$ in the $\mF$-topology. Now let us describe the reparametrization map. Given an arbitrary cell $\al$ and $k\in \N$, we take $\al_c$ to be the center cell of $\al(k)$. We can define a map $\n_{\al, k}: \al \to \al$ such that it maps $\al_c$ to $\al$ linearly, and for each $x\in \al\backslash \al_c$, if we denote by $x_c$ the nearest point projection of $x$ to $\partial\al_c$ then $\n_{\al, k}$ maps $x$ to $\n_{\al, k}(x_c)$. This map dilates $\al_c$ to $\al$ and compresses $\al\backslash\al_c$ to the boundary $\partial\al$, and it is homotopic to the identity map. With this notion $\Phi_i' \vert_\al = \Phi_i\vert_\al \circ \n_{\al, l-l_i}$ on each cell $\al\in Y(l_i)$.  
Hence we finish the proof.
\end{proof}

The following result is the counterpart of \cite[Corollary 3.12]{Marques-Neves17}. %using Proposition \ref{P:Almgren extensions are homotopic} in place of \cite[Proposition 3.11]{Marques-Neves17}.
\begin{proposition}
\label{P:discretized sweepouts are homotopic}
Let $\{\phi_i\}_{i\in\N}$ and $\{\psi_i\}_{i\in\N}$ be given by Theorem \ref{T:continuous to discrete} applied to some $\Phi$ therein. Assume that $\Phi$ is continuous in the $\mF$-topology on $X$. Then the Almgren extension $\Phi_i$ is homotopic to $\Phi$ in the $\mF$-topology for sufficiently large $i$. 

In particular, for $i$ large enough, there exist homotopy maps $\Psi_i: [0, 1]\times X\to \C(M)$ continuous in the $\mF$-topology, $\Psi_i(0, \cdot)=\Phi_i$, $\Psi_i(1, \cdot)=\Phi$, and
\[ \limsup_{i\to\infty}  \sup_{t\in[0, 1], x\in X}\mF(\Psi_i(t, x), \Phi(x)) \to 0.  \]
%Furthermore,
%\[ \limsup_{i\to\infty} \sup_{x\in X}\M(\partial \Phi_i(x))\leq \sup_{x\in X}\M(\partial \Phi(x)). \]
Therefore for given $h\in C^\infty(M)$, we have
\[ \limsup_{i\to\infty} \sup_{x\in X}\Ah(\Phi_i(x))\leq \sup_{x\in X}\Ah(\Phi(x)). \]
\end{proposition}

\begin{comment}%earlier proof
\begin{proof}
%By Theorem \ref{T:continuous to discrete}(\rom{1})(\rom{2}) and Theorem \ref{T:discrete to continuous}(\rom{1})(\rom{3}),
%\[ \lim_{i\to\infty} \sup\{\F(\Phi_i(x), \Phi(x)): x\in X= 0\}. \]
The part of homotopy equivalence and the statement of supremum of masses follow the same way as \cite[Corollary 3.12]{Marques-Neves17} using Proposition \ref{P:Almgren extensions are homotopic} in place of \cite[Proposition 3.11]{Marques-Neves17}.

By Theorem \ref{T:discrete to continuous}(\rom{1})(\rom{3}), for any cell $\al\in X(k_i)$, $x\in \al$, $y\in \al_0$
\[ \Ah(\Phi_i(x))\leq \Ah(\phi_i(y))+(1+c)C_0\f(\phi_i), \]
for $i$ sufficiently large. The statement of supremum of $\Ah$-values then follows from Theorem \ref{T:continuous to discrete}(\rom{1})(\rom{5}).
\end{proof}
\end{comment}

\begin{proof}
For $i$ large enough such that $\de_i<\de_0$ in Theorem \ref{T:discrete to continuous}, we let $\bar\Psi_i : I\times X \to \C(M)$ be the Almgren extensions of $\psi_i$. By Theorem \ref{T:continuous to discrete}(\rom{4}) (with $Z=X$) and Theorem \ref{T:discrete to continuous}(\rom{3}), we know that 
\begin{equation}
\label{E:mF closeness after discretization and interpolation}
\limsup_{i\to\infty}  \sup_{t\in[0, 1], x\in X}\mF(\bar\Psi_i(t, x), \Phi(x)) \to 0. 
\end{equation}
As in the proof of the above Proposition, we can amend $\bar\Psi_i$ with the reparametrization maps associated with the two pairs $(\Phi_{i-1}', \Phi_{i-1})$ and $(\Phi_i', \Phi_i)$, and abuse the notation and still denote them by $\bar\Psi_i$. Then $\bar\Psi_i$ is a continuous (in the $\mF$-topology) homotopy between $\Phi_{i-1}$ and $\Phi_i$. Note that the reparametrizations are done is small cells with sizes converging to zero, so (\ref{E:mF closeness after discretization and interpolation}) still holds true for the amended maps by Theorem \ref{T:discrete to continuous}(\rom{3}) again.  For given $i$ large enough, to construct the homotopy from $\Phi_i$ to $\Phi$, we can just let $\Psi_i : [0, \infty] \times X\to \C(M)$ be the gluing of all $\{\bar\Psi_j\}_{j\geq i}$. Note that by (\ref{E:mF closeness after discretization and interpolation}),  $\Psi_i(\infty, \cdot)=\Phi$ (we can identify $[0, \infty]$ with $[0 ,1]$ in the definition of $\Psi_i$), and (\ref{E:mF closeness after discretization and interpolation}) holds true with $\bar\Psi_i$ replaced by $\Psi_i$. Hence we finish the proof.
\end{proof}

\subsection{Proof of the min-max Theorem}
\label{SS:proof of min-max theorem}

One key ingredient in the Almgren-Pitts theory to prove regularity of min-max varifold is to introduce the ``almost minimizing" concept. Given $h\in \mathcal S(g)$, we refer to \cite[Section 6]{Zhou-Zhu18} for the detailed notion of $h$-almost minimizing varifold and related properties.  The existence of almost minimizing varifolds follows from a combinatorial argument of Pitts \cite[page 165-page 174]{Pitts81} inspired by early work of Almgren \cite{Almgren65}. Pitts's argument works well in the construction of min-max PMC hypersurfaces; see \cite[Theorem 6.4]{Zhou-Zhu18}. Marques-Neves has generalized Pitts's combinatorial argument to a more general form in \cite[2.12]{Marques-Neves17}, and we can adapt their result to the PMC setting with no change.  We now describe the adaption.

Consider a sequence of cubical subcomplexes $Y_i$ of $I(m, k_i)$ with $k_i\to\infty$, and a sequence $S=\{\varphi_i\}_{i\in\N}$ of maps
\[ \varphi_i: (Y_i)_0\to \C(M) \]
with fineness $\mf(\varphi_i)=\de_i$ converging to zero. Define
\[ \bL^h(S)=\limsup_{i\to\infty} \sup\{ \Ah(\varphi_i(y)): y\in (Y_i)_0 \}, \]
\[ \bK(S)=\{ V=\lim_{j\to\infty}|\partial \varphi_{i_j}(y_j)| \text{ as varifolds}: y_j\in (Y_{i_j})_0 \}, \]
and
\[ \bC(S)=\{V=\lim_{j\to\infty}|\partial \varphi_{i_j}(y_j)| \text{ as varifolds}:\, \text{with }\lim_{j\to\infty}\Ah(\varphi_{i_j}(y_j))=\bL^h(S)\}. \]

We say that an element $V\in \bC(S)$ is {\em $h$-almost minimizing in small annuli with respect to $S$} (c.f. \cite[Definition 6.3]{Zhou-Zhu18}), if for any $p\in M$ and any small enough annulus $A=A_{r_1, r_2}(p)$ centered at $p$ with radii $0<r_1<r_2$, there exist sequences $\{i_j\}_{j\in\N}\subset \{i\}_{i\in\N}$ and $\{y_j: y_j\in (Y_{i_j})_0\}_{j\in\N}$, such that $V=\lim_{j\to\infty}|\partial \varphi_{i_j}(y_j)|$, $\lim_{j\to\infty}\Ah(\varphi_{i_j}(y_j))=\bL^h(S)$, and $\varphi_{i_j}(y_j)\in \sA^h(A; \ep_i, \de_i; \M)$ (see \cite[Definition 6.1]{Zhou-Zhu18}) for some $\ep_i, \de_i\to 0$. The last condition is usually called {\em $(\ep_i, \de_i, h)$-almost minimizing}. Note that by \cite[Proposition 6.5]{Zhou-Zhu18}, $V$ is also $h$-almost minimizing in small annuli in the sense of \cite[Definition 6.3]{Zhou-Zhu18}. 
%Indeed, by shrinking $A$ a little bit to any smaller $A'\subset\subset A$ and changing $\{\de_i\}$ to $\{\de_i'\}\to 0$. we have $\varphi_{i_j}(y_j)\in \sA^h(A'; \ep_i, \de_i'; \F)$.
 
The following is a variant of \cite[Theorem 2.13]{Marques-Neves17} and \cite[Theorem 4.10]{Pitts81}.
\begin{theorem}
\label{T:combinatorial deformation}
If no element $V\in \bC(S)$ is $h$-almost minimizing in small annuli with respect to $S$, then there exists a sequence $\tilde{S}=\{\tilde{\varphi}_i\}$ of maps
\[ \tilde{\varphi}_i: Y_i(l_i)_0\to \C(M), \]
for some $l_i\in \N$, such that:
\begin{itemize}
\item $\tilde{\varphi}_i$ is homotopic to $\varphi_i$  with fineness converging to zero as $i\to\infty$;
\item $\bL^h(\tilde{S})<\bL^h(S)$.
\end{itemize}
\end{theorem}
\begin{proof}
By the assumption of the theorem, for each $V\in \bC(S)$, there exists a $p\in M$,  such that for any $\tilde r>0$, there exist $r, s>0$,  with $\tilde r> r+2s >r-2s >0$ and $\ep>0$, such that, %and $I\in\N$, if $i>I$ and
 if $\Ah(\varphi_i(y))>\bL^h(S)-\ep$ and $\mF(|\partial \varphi_i(y)|, V)<\ep$, then $\varphi_i(y)\notin \sA^h(A_{r-2s, r+2s}(p); \ep, \de; \M)$ for any $\de>0$. As in the proof of \cite[Theorem 4.10]{Pitts81}, we denote $c=(3^m)^{3^m}$. By the compactness of $\bC(S)$, we can find a uniform $\ep>0$ and $I\in\N$, and finitely many points $p_1, \cdots, p_\nu \in M$, and for each $p_j$, we can find $c$ concentric annuli $A_{j, 1}\supset\supset \cdots \supset\supset A_{j, c}$ (centered at $p_j$), such that, if $\Ah(\varphi_i(y))>\bL^h(S)-\ep$ and $i>I$,  then there exists some $j\in \{1, \cdots, \nu\}$, so that $\varphi_i(y)\notin \sA^h(A_{j, a}; \ep, \de; \M)$ for all $a\in \{1, \cdots, c\}$ and for any $\de>0$. From here the construction in \cite[Page 165-174]{Pitts81} can be applied to $S$ so as to produce the desired $\tilde S$. 
\end{proof}

\vspace{1em}
Now we are ready to prove Theorem \ref{T:main min-max theorem} following closely that of \cite[Theorem 3.8]{Marques-Neves16}. The only additional thing is to keep track of the volume term $\int_\Om h d\mH^{n+1}$ in $\Ah(\Om)$ and the values of maps assumed on $Z$.

\begin{proof}[Proof of Theorem \ref{T:main min-max theorem}]
Let $\{\Phi_i\}_{i\in\N}$ be a pulled-tight min-max sequence for $\Pi$.  Given $\Phi_i : X\to\C(M)$, it has no concentration of mass as it is continuous in the $\mF$-topology, so applying Theorem \ref{T:continuous to discrete} gives a sequence of maps:
\[ \phi_i^j : X(k_i^j)_0\to \C(M), \]
with $k_i^j<k_i^{j+1}$ and a sequence of positive $\{\de_i^j\}_{j\in\N}\to 0$, satisfying (\rom{1})$\cdots$(\rom{5}) in Theorem \ref{T:continuous to discrete}.

As $\Phi_i$ is continuous in the $\mF$-topology, by the same reasoning as Theorem \ref{T:continuous to discrete}(\rom{3})(\rom{4}), we further have that for every $x\in X(k_i^j)_0$,
\[ \M(\partial\phi_i^j(x))\leq \M(\partial \Phi_i(x))+\eta_i^j \]
with $\eta_i^j\to 0$ as $j\to\infty$, and 
\[ \sup\{\mF(\phi_i^j(x), \Phi_i(x)): x\in X(k_i^j)_0\}\to 0, \text{ as } j\to\infty. \]

Now choose $j(i)\to\infty$ as $i\to\infty$, such that $\varphi_i=\phi_i^{j(i)}: X(k_i^{j(i)})_0\to \C(M)$ satisfies:
\begin{itemize}
\item $\sup\{\mF(\varphi_i(x), \Phi_i(x)): x\in X(k_i^{j(i)})_0\}\leq a_i$ with $a_i\to 0$ as $i\to\infty$;
\item $\sup\{\mF(\Phi_i(x), \Phi_i(y)): x, y \in \al, \al\in X(k_i^{j(i)})\}\leq a_i$;
\item the fineness $\mf(\varphi_i)\to 0$ as $i\to\infty$;
\item the Almgren extensions $\Phi_i^{j(i)}: X\to \C(M)$ is homotopic to $\Phi_i$ in the $\mF$-topology with homotopy maps $\Psi_i^{j(i)}$, and
\[ \limsup_{i\to\infty} \sup\{\mF(\Psi_i^{j(i)}(t, x), \Phi_i(x)): t\in [0, 1], x\in X\}=0, \]
and
\[ \limsup_{i\to\infty} \sup_{x\in X}\Ah(\Phi_i^{j(i)}(x))\leq \limsup_{i\to\infty}\sup_{x\in X}\Ah(\Phi_i(x))=\bL^h(\Pi), \]
by Proposition \ref{P:discretized sweepouts are homotopic}.
%\[ \lim_{i\to\infty} \sup\{\mF(\Phi_i^{j(i)}(x), \Phi_i(x)): x\in X\}=0, \] %write a few more reasoning words, done
%by Theorem \ref{T:discrete to continuous}(\rom{1})(\rom{3}).
\end{itemize}

Therefore, if $S=\{\varphi_i\}$, then $\bL^h(S)=\bL^h(\{\Phi_i\})$ and $\bC(S)=\bC(\{\Phi_i\})$.  
%Now let us justify $\bL^h(S)=\bL^h(\{\Phi_i\})$: on one hand, by Theorem \ref{T:discrete to continuous}(\rom{1}), \[\limsup_{i\to\infty} \sup_{x\in X(k_i^{j(i)})_0}\Ah(\varphi_i(x)) \leq \limsup_{i\to\infty} \sup_{x\in X}\Ah(\Phi_i^{j(i)}(x))\leq \bL^h(\Pi);\] on the other hand, choose a subsequence $\{\Phi_i(x_i)\}$, such that $\Ah(\Phi_i(x_i))\to \bL^h(\Pi)$, then we can pick $x_i'$ as the nearest projection of $x_i$ onto $X(k_i^{j(i)})_0$, such that \[ \mF(\varphi_i(x_i'), \Phi_i(x_i))\leq \mF(\varphi_i(x_i'), \Phi_i(x_i'))+\mF(\Phi_i(x_i'), \Phi_i(x_i))\leq 2a_i\to 0; \] this means that \[ \M(\partial \varphi_i(x_i'))-\M(\partial\Phi_i(x_i))\to 0, \text{  and } \] \[\M(\varphi_i(x_i')-\Phi_i(x_i)) \to 0;\] so this implies that \[ \Ah(\varphi_i(x_i'))-\Ah(\Phi_i(x_i))\to 0, \] and hence proved \[ \bL^h(\Pi)\leq \limsup_{i\to\infty} \sup_{x\in X(k_i^{j(i)})_0}\Ah(\varphi_i(x)). \]
%We next justify $\bC(S)=\bC(\{\Phi_i\})$: if $\lim \Ah(\varphi_i(x_i))=\bL^h$, then $\lim \Ah(\Phi_i(x_i))=\bL^h$, and then $V=\lim |\partial \varphi_i(x_i)|=\lim |\partial \Phi_i(x_i)|\in \bC(\{\Phi_i\})$.
By Theorem \ref{T:combinatorial deformation}, if no element $V\in \bC(S)$ is $h$-almost minimizing in small annuli with respect to $S$, we can find a sequence $\tilde{S}=\{\tilde{\varphi}_i\}$ of maps:
\[ \tilde{\varphi}_i: X(k_i^{j(i)}+l_i)_0\to \C(M) \]
such that
\begin{itemize}
\item $\tilde{\varphi}_i$ is homotopic to $\varphi_i$  with fineness converging to zero as $i\to\infty$;
\item $\bL^h(\tilde{S})<\bL^h(S)$.
\end{itemize}

By Proposition \ref{P:Almgren extensions are homotopic}, the Almgren extensions of $\varphi_i, \tilde{\varphi}_i$:
\[ \Phi_i^{j(i)}, \tilde{\Phi}_i: X\to \C(M), \]
respectively, are homotopic to each other in the $\mF$-topology for $i$ large enough, so $\tilde{\Phi}_i$ is homotopic to $\Phi_i$ in the $\mF$-topology. 

By assumption (\ref{E:nontrivial assumption}) and (\ref{E:boundary requirement}), for $i$ large enough, $\tilde{\varphi}_i$ is the identical to $\varphi_i \circ \n(k_i^{j(i)}+l_i, k_i^{j(i)})$ near $Z(k_i^{j(i)}+l_i)_0$; indeed, the deformation process in Theorem \ref{T:combinatorial deformation} was only made to those $\varphi_i(x)$ with $\Ah(\varphi_i(x))$ close to $\bL^h(S)$. %by Theorem \ref{T:discrete to continous}(\rom{1})(\rom{3})
Therefore the homotopy maps $\tilde\Psi_i$ between $\Phi_i^{j(i)}$ and $\tilde\Phi_i$ produced by Proposition \ref{P:Almgren extensions are homotopic} when restricted to $Z$ are just the reparametrization maps described therein. Hence
\[ \limsup_{i\to\infty}\sup\{\mF(\tilde\Psi_i(t, x), \Phi_i^{j(i)}(x)): t\in [0, 1], x\in Z\}=0. \] %more reasoning words, by by Theorem \ref{T:discrete to continous}(\rom{1})(\rom{3})
Therefore $\{\tilde{\Phi}_i\}_{i\in\N} \in \Pi$.  However, by Theorem \ref{T:discrete to continuous}
\[ \limsup_{i\to\infty} \sup\{\Ah(\tilde{\Phi}_i(x)): x\in X\}\leq \bL^h(\tilde{S})<\bL^h(S)=\bL^h(\Pi). \]
This is a contradiction. So some $V\in \bC(S)=\bC(\{\Phi_i\})$ is $h$-almost minimizing in small annuli with respect to $S$, and hence is $h$-almost minimizing in small annuli in the sense of \cite[Definition 6.3]{Zhou-Zhu18}. 

To finish the proof, we need to show that $V$ has $c$-bounded first variation, and then \cite[Theorem 7.1 and Proposition 7.3]{Zhou-Zhu18} give the regularity of $V$ and the existence of $\Om$. Indeed, by Definition \ref{D:pulled-tight}, $V$ either has $c$-bounded first variation or belongs to $|\partial \Phi_0|(Z)$. Being $h$-almost minimizing in small annuli implies that $V$ has $c$-bounded first variation away from finitely many points by \cite[Lemma 6.2]{Zhou-Zhu18}. If $V\in |\partial \Phi_0|(Z)$, then the proof of \cite[Theorem 4.1]{Harvey-Lawson75} implies that $\|V\|$ has at most $r^{n-\frac{1}{2}}$-volume growth near these bad points, so the first variation extends across these points, and hence $V$ has $c$-bounded first variation in $M$. 
%we know that $\|V\|$ has polynomial volume growth near these bad points (since $\Phi_0: Z\to \C(M)$ is continuous in $\mF$-topology), so the first variation extends across these points (c.f. \cite{Harvey-Lawson75}), and $V$ has $c$-bounded first variation in $M$. 
(Note that even if $V\in |\partial\Phi_0|(Z)$, the associated $\Om\notin \Phi_0(Z)$, as $\Om$ may be equal to $M\backslash\Phi_0(z)$ for some $z\in Z$.) So we finish the proof. %more reasoning words.
\end{proof}

%%%%%%%%%%%%%%%%%%%%%%%%%%%%%%%%%%
% Section 2   	 		  Compactness	                      	   %
%%%%%%%%%%%%%%%%%%%%%%%%%%%%%%%%%%

\section{Compactness of PMC hypersurfaces with bounded Morse index}
\label{S:compactness}

Now we present an adaption of Sharp's compactness theorem \cite{Sharp17} (for minimal hypersurfaces) to the PMC setting and necessary modifications of the proof. Given a closed Riemannian manifold $(M^{n+1}, g)$ and $h\in \mathcal S(g)$, denote by $\PMC^h$ the class of smooth, closed, almost embedded hypersurfaces $\Si\subset M$, such that $\Si$ is represented as the boundary of some open subset $\Om\subset M$ (in the sense of current), and the mean curvature of $\Si$ with respect to the outer normal of $\Om$ is prescribed by $h$, i.e.
\[ H_\Si=h|_\Si. \]
In the following we will sometime abuse the notation and identify $\Si$ with $\Om$. 

Note that when $h\in \mathcal S (g)$, the min-max PMC hypersurfaces produced in Theorem \ref{T:main min-max theorem} satisfy the above requirements.  Indeed, such $\Si=\partial\Om$ is a critical point of the weighted $\Ah$ functional (\ref{E: Ah}):
\[\Ah(\Om)=\Area(\Si)-\int_\Om h\, d\mH^{n+1}.\]
The second variation formula for $\Ah$ along normal vector field $X=\varphi\nu\in\X(M)$ is given by
\[ \de^2\Ah|_\Om (X, X)=\int_\Si(|\nabla \varphi|^2-(Ric^M(\nu, \nu)+|A^\Si|^2+\partial_\nu h)\varphi^2)d\mu_\Si. \]
The classical Morse index for $\Si$ is defined as the number of negative eigenvalues of the the above quadratic form. However, since we will deal with hypersurfaces with self-touching, a weaker version of index is needed. We adopt a concept used by Marques-Neves \cite[Definition 4.1]{Marques-Neves16}. As we will see, this weaker index works well for proving both compactness theory and Morse index upper bound. 
 
\begin{definition}
\label{D:k-unstable}
Given $\Si\in \PMC^h$ with $\Si=\partial\Om$, $k\in\N$ and $\ep\geq 0$, we say that $\Si$ is {\em $k$-unstable in an $\ep$-neighborhood} if there exists $0<c_0<1$ an a smooth family $\{F_v\}_{v\in \overline{B}^k}\subset \Diff(M)$ with $F_0=\Id$, $F_{-v}=F_v^{-1}$ for all $v\in\overline{B}^k$ (the standard $k$-dimensional ball in $\R^k$) such that, for any $\Om'\in \overline{\bB}^{\mF}_{2\ep}(\Om)$, the smooth function:
\[ \Ah_{\Om'}: \overline{B}^k \to [0, \infty),\quad \Ah_{\Om'}(v)=\Ah(F_v(\Om'))\]
satisfies:
\begin{itemize}
\item $\Ah_{\Om'}$ has a unique maximum at $m(\Om')\in B^k_{c_0/\sqrt{10}}(0)$;
\item $-\frac{1}{c_0}\Id\leq D^2\Ah_{\Om'}(u)\leq -c_0\Id$ for all $u\in \overline{B}^k$.
\end{itemize}
Since $\Si$ is a critical point of $\Ah$, necessarily $m(\Om)=0$.
\end{definition}
\begin{remark}
If a sequence $\Om_i$ converges to $\Om$ in the $\mF$-topology, then $\Ah_{\Om_i}$ tends to $\Ah_\Om$ in the smooth topology. Thus if a $\Si\in\PMC^h$ is $k$-unstable in a $0$-neighborhood, then it is $k$-unstable in an $\ep$-neighborhood for some $\ep>0$.
\end{remark}

\begin{definition}
\label{D:weak index}
Given a $\Si\in \PMC^h$ and $k\in\N$, we say that {\em its Morse index is bounded (from above) by $k$}, denoted as
\[  \ind(\Si)\leq k,\] 
if it is not $j$-unstable in $0$-neighborhood for any $j\geq k+1$.%We say that $\Si$ is stable if $\ind(\Si)=0$.

All the above concepts can be localized to an open subset $U\subset M$ by using $\Diff(U)$ in place of $\Diff(M)$. If $\Si$ has index equal to 0 in $U$, we say $\Si$ is {\em weakly stable} in $U$. 
\end{definition}

\begin{proposition}
\label{P:equivalence of k-unstable}
If $\Si\in\PMC^h$ is smoothly embedded with no self-touching, then $\Si$ is $k$-unstable (in $0$-neighborhood) if and only if its classical Morse index is $\geq k$.
\end{proposition}
\begin{proof}
The proof is the same as \cite[Proposition 4.3]{Marques-Neves16}.
\end{proof}

We have the following curvature estimates as a variant of \cite[Theorem 3.6]{Zhou-Zhu18} (with relatively weaker stability assumptions).
\begin{theorem}[Curvature estimates for weakly stable PMC]
\label{T:Curvature estimates for weakly stable PMC}
Let $3\leq (n+1)\leq 7$, and $U\subset M$ be an open subset. Let $\Si\in\PMC^h$ be weakly stable in $U$ with $\Area(\Si)\leq C$, then there exists $C_1$ depending only on $n, M, \|h\|_{C^3}, C$, such that
\[ |A^\Si|^2(x)\leq \frac{C_1}{\dist^2_M(x, \partial U)} \quad \text{ for all } x\in\Si. \] 
\end{theorem}
\begin{proof}
The curvature estimates follow from standard blowup arguments together with the Bernstein Theorem \cite[Theorem 2]{Schoen-Simon-Yau75} and \cite[Theorem 3]{Schoen-Simon81}. In particular, being weakly stable in $U$ means that for any ambient vector field $X\in\X(U)$ which generates the flow $\phi^X_t$, we have
\begin{equation}
\label{E:proof of curv est1}
\frac{d^2}{dt^2}\Big\vert_{t=0}\Ah(\phi^X_t(\Om))\geq 0. 
\end{equation}
Assume the conclusion were false, then there exists a sequence of weakly stable hypesurfaces $\{\Si_i\}_{i\in\N}$ with prescribing functions $\{h_i\}_{i\in\N}$ satisfying uniform bounds, but $\sup_U \dist^2_M(\cdot, \partial U) |A^{\Si_i}|^2(\cdot)\to \infty$. By the standard blowup process (c.f. \cite{White87}), one can take a sequence of rescalings of $\Si_i$ which converges locally in $C^{3, \al}$ and graphically to a non-flat minimal hypersurface $\Si_\infty$ in $\R^{n+1}$. Note that the rescalings of $\{h_i\}$ converges to 0 locally uniformly in $C^3$. By the almost embedded assumption and the maximum principle for minimal hypersurfaces (\cite{Colding-Minicozzi11}), $\Si_\infty$ is embedded and hence is 2-sided. By the classical monotonicity formula and area upper bound assumption on $\{\Si_i\}$, $\Si_\infty$ has polynomial volume growth. The key observation is that (\ref{E:proof of curv est1}) is preserved under locally $C^{3, \al}$ convergence, and hence $\Si_\infty$ is a stable minimal hypersurface. Therefore it has to be flat by the Bernstein Theorem, but this is a contradiction.
\end{proof}

Given $h\in \mS(g)$, $0<\La\in\R$ and $I\in\N$, let
\[ \PMC^h(\La, I):=\{\Si\in\PMC^h: \Area(\Si)\leq \La, \ind(\Si)\leq I\}. \]

\begin{theorem}[Compactness for PMC's with bounded index]
\label{T:compactness with bounded index}
Let $(M^{n+1}, g)$ be a closed Riemannian manifold of dimension $3\leq (n+1)\leq 7$. Assume that $\{h_k\}_{k\in\N}$ is a sequence of smooth functions in $\mathcal{S}(g)$ such that $\lim_{k\to\infty}h_k=h_\infty$ in smooth topology. Let $\{\Si_k\}_{k\in\N}$ be a sequence of hypersurfaces such that $\Si_k\in \PMC^{h_k}(\La, I)$ for some fixed $\La>0$ and $I\in\N$. Then,
\begin{enumerate}[label=(\roman*)]
\item Up to a subsequence, there exists a smooth, closed, almost embedded hypesurface $\Si_\infty$ with prescribed mean curvature $h_\infty$, 
%$\Si_\infty\in\PMC^{h_\infty}(\La, I)$ 
such that $\Si_k\to \Si_\infty$ (possibly with integer multiplicity) in the varifold sense, and hence also in the Hausdorff distance by monotonicity formula.

\item There exists a finite set of points $\mathcal Y\subset M$ with $\#\mathcal Y\leq I$, such that the convergence of $\Si_k\to \Si_\infty$ is locally smooth and graphical on $\Si_\infty\setminus \mathcal Y$.  

\item If $h_\infty\in \mathcal{S}(g)$, then the multiplicity of $\Si_\infty$ is 1, and $\Si_\infty\in\PMC^{h_\infty}(\La, I)$.

\item Assuming $\Si_k\neq \Si_\infty$ eventually and $h_k=h_\infty=h\in\mathcal S(g)$ for all $k$ such that every $\Si\in \PMC^{h}$ is properly embedded with no self-touching, then $\mathcal Y=\emptyset$, and the nullity of $\Si_\infty$ with respect to $\de^2\Ah$ is $\geq 1$.

\item If $h_\infty\equiv 0$, then the classical Morse index of $\Si_\infty$ satisfies $\ind(\Si_\infty)\leq I$ (without counting multiplicity).
\end{enumerate}
\end{theorem}
\begin{remark}
One main goal of this result is to use PMC hypersurfaces with prescribing functions in $\mathcal{S}(g)$ to approximate PMC's with prescribing functions lying in $C^\infty(M)\backslash \mathcal{S}(g)$. Therefore, it is natural to not assume $h_\infty \in \mathcal S(g)$.  Indeed for some $h\in C^\infty(M)$, a PMC $\Si$ associated with $h$ may have touching set containing a relative open subset $W\subset \Si$, where $h$ vanishes. For such hypersurfaces, $\Ah$ is defined by viewing $\Si$ as an Alexandrov immersed hypersurface, and so does the weak index. 
%In special cases such as (\rom{3})(\rom{5}), we can obtain bounds for $\ind(\Si_\infty)$. 
\end{remark}
\begin{proof}
The proof follows essentially the same way as \cite[Theorem 2.3]{Sharp17} once we use Theorem \ref{T:Curvature estimates for weakly stable PMC} to replace \cite[Theorem 2.1]{Sharp17}; we will provide necessary modifications.

\vspace{1em}
{\bf Part 1:} We first have the following variant of \cite[Lemma 3.1]{Sharp17}. Given any collection of $I+1$ pairwise disjoint open sets $\{U_i\}_{i=1}^{I+1}$, we have that $\Si_k$ (we drop the sub-index $k$ in this paragraph) is weakly stable in $U_i$ for some $1\leq i\leq I+1$. Indeed, suppose this were false, then $\Si=\partial\Om$ is at least $1$-unstable in each $U_i$, hence there exist $c_i\in(0, 1)$ and $\{F^i_t\}_{t\in[-1, 1]}\subset \Diff(U_i)$ with $F^i_{-t}=(F^i_t)^{-1}$, such that $-\frac{1}{c_i}\leq \frac{d^2}{dt^2} \Ah(F^i_t(\Om))\leq -c_i$. Now for $v=(v_1, \cdots, v_{I+1})\in \overline{B}^{I+1}$, let $F_v(x)=F_{v_{I+1}}\circ \cdots \circ F_{v_1}(x)$. Since $\{U_i\}$ are pairwise disjoint, it is easy to see that $c_0=\min\{c_i\}$ and $\{F_v\}$ give an $(I+1)$-unstable pair for $\Si$, and hence is a contradiction.

This fact together with Theorem \ref{T:Curvature estimates for weakly stable PMC} imply that (up to a subsequence) $\Si_k$ converges locally smoothly and graphically to an almost embedded hypersurface $\Si_\infty$ of prescribed mean curvature $h_\infty$ (possibly with integer multiplicity) away from at most $I$ points, which we denote by $\mathcal Y$. Since as varifolds $\Si_k$ have uniformly bounded first variation, by Allard's compactness theorem \cite{Allard72}, $\Si_k$ also converges as varifolds to an integral varifold represented by $\Si_\infty$.

Now we prove that $\Si_\infty$ extends smoothly as an almost embedded hypersurface across the singular points $\mathcal Y$, i.e. $\mathcal Y$ are removable. By the argument in \cite[Claim 2, page 326]{Sharp17}, for each $y_i\in\mathcal Y$, there exists some $r_i>0$ such that $\Si_\infty$ is weakly stable in $B_{r_i}(y_i)\backslash\{y_i\}$ in the following sense. Denote $\Om_\infty$ as the weak limit of $\Om_k$ as Caccioppoli sets where $\Si_k=\partial\Om_k$. The associated functional for $\Si_\infty$ is $\A^{h_\infty}(\Si_\infty)=\Area(\Si_\infty)-\int_{\Om_\infty} h_\infty d\mH^{n+1}$. Note that the touching set of $\Si_\infty$ may contain an open subset $W\subset \Si_\infty$ and hence $\partial\Om_\infty=\Si_\infty\backslash \{\text{touching set of }\Si_\infty\}$ may only be a proper subset of $\Si_\infty$. Nevertheless, we say $\Si_\infty$ is weakly stable, if for any $X\in \X(B_{r_i}(y_i)\backslash\{y_i\})$ with the associated flow $\{\phi^X_t: t\in[-\ep, \ep]\}$, $\frac{d^2}{dt^2}\big\vert_{t=0} \A^{h_\infty}(\phi^X_t(\Si_\infty))\geq 0$. Note that if this were not true for some $X\in \X(B_{r_i}(y_i)\backslash\{y_i\})$,  as $\A^{h_k}(\phi^X_t(\Si_k))$ converges to $\A^{h_\infty}(\phi^X_t(\Si_\infty))$ smoothly as functions of $t$, then $\Si_k$ is not weakly stable in $B_{r_i}(y_i)\backslash\{y_i\}$ for $k$ sufficiently large. Following \cite[Claim 2, page 326]{Sharp17}, we can deduce the required stability property for $\Si_\infty$. %this weak stability is defined by viewing $\Si_\infty$ as a pair $(\Si_\infty, \Om_\infty)$ where $\Om_\infty$ is the weak limit of $\Om_k$ as Caccioppoli sets, and $\A^{h_\infty}(\Si_\infty)=\Area(\Si_\infty)-\int_{\Om_\infty} h_\infty d\mH^{n+1}$. Weak stability can be defined the same way by using $F_v$ to deform the pari$(\Si_\infty, \Om_\infty)$. Actually $\A^{h_k}_{\Om_k}$ converges smoothly to $\A^{h_\infty}(\Si_\infty)$ as functions on any $B^k$, so stability is preserved.  
Since $\Si_\infty$ has bounded first variation, then by a classical removable singularity result, Theorem \ref{T:removable singularity}, we get the smooth extension. Up to here, we have finished proving (\rom{1}) and (\rom{2}).

\vspace{1em}
{\bf Part 2:} If $h_\infty\in \mathcal S(g)$, \cite[Theorem 3.20]{Zhou-Zhu18} implies that $\Si_\infty$ has multiplicity 1, and is a boundary of some open set $\Om_\infty$; (note that when $h_\infty\in\mathcal{S}(g)$, only case (2) of \cite[Theorem 3.20]{Zhou-Zhu18} will happen). %Assuming $\mathcal Y\neq \emptyset$, pick $y\in \mathcal Y$. If $y$ is not a touching point of $\Si_\infty$, then $\Si_k$ converges to $\Si_\infty$ smoothly around $y$ by the Allard regularity theorem. If $y$ is a touching point, then Allard regularity theorem applies to each embedded piece of $\Si_\infty$ around $y$, and also gives smooth graphical convergence of $\Si_k$ to $\Si_\infty$ near $y$. This is a contradiction, and hence $\mathcal Y=\emptyset$. 
In fact, fix a point $p\in\Si_\infty$ where $\Si_\infty$ is properly embedded. If the limit $\Si_\infty$ has multiplicity $\geq 2$, then for $i$ sufficiently large and inside a neighborhood of $p$, $\Si_i$ consists of several sheets with normal pointing to the same side of $\Si_\infty$, but this can not happen when $\Si_i$ bounds a region $\Om_i$. We refer to the proof of \cite[Theorem 3.20]{Zhou-Zhu18} for more details.

If $\ind(\Si_\infty)> I$, then there exist $c_0\in(0, 1)$ and $\{F_v: v\in \overline{B}^{I+1}\}\subset \Diff(M)$ such that $-\frac{1}{c_0}\Id \leq D^2 \mathcal A^{h_\infty}(F_v(\Om_\infty))\leq -c_0\Id$ for all $v\in \overline{B}^{I+1}$. Since $\Si_k=\partial\Om_k$ converges to $\Si_\infty$ smoothly away from finitely many points, we know that $\Om_k$ converges to $\Om_\infty$ in the $\mF$-topology as Caccioppoli sets, then the sequence $v\to \mathcal A^{h_k}(F_v(\Om_k))$ converges to $v\to \mathcal A^{h_\infty}(F_v(\Om_\infty))$ smoothly as functions on $\overline{B}^{I+1}$.  Therefore, for $k$ large enough, $-\frac{2}{c_0}\Id \leq D^2 \mathcal A^{h_k}(F_v(\Om_k))\leq -\frac{c_0}{2}\Id$, so $\Si_k$ is $(I+1)$-unstable, which is a contradiction. This finishes the proof of (\rom{3}).

\vspace{1em}
{\bf Part 3:} Assuming $\Si_k\neq \Si_\infty$ eventually and $h_k=h_\infty=h\in \mathcal S(g)$ such that every element in $\PMC^h$ is properly embedded, we know $\mathcal Y=\emptyset$ by multiplicity 1 convergence and the Allard regularity theorem \cite{Allard72}. Next we will produce a Jacobi field for the second variation $\de^2\Ah$ along $\Si_\infty$; this implies the nullity is $\geq 1$. 

By (\ref{E: 2nd variation for Ah}), the Jacobi operator associated with $\de^2\Ah$ along a PMC $\Si\in \PMC^h$ is
\[ L^h_{\Si} \varphi = -\lap_{\Si} \varphi - \big(Ric^M(\nu, \nu)+|A^\Si|^2 + \partial_\nu h\big)\varphi. \]
The smooth graphical convergence of $\Si_k\to \Si$ implies that for $k$ sufficiently large, $\Si_k$ can be written as a graph $u_k$ in the normal bundle of $\Si_\infty$, and $u_k\to 0$ uniformly in smooth topology. Subtracting the mean curvature operators between $\Si_k$ and $\Si_\infty$, we get:
\[h(x, u_k)-h(x, 0)=H_{\Si_k}-H_{\Si_\infty}=L_{\Si_\infty}u_k + o(u_k),\]
where $L_{\Si_\infty}u= -\lap u - \big(Ric^M(\nu, \nu)+|A^\Si|^2\big)u$ is the Jacobi operator for second variation of area, and the second equation follows from \cite{Simon87} and \cite[page 331]{Sharp17}; (note that though the calculation in \cite[page 331]{Sharp17} is done assuming $h\equiv 0$, it does not depend on $h$). The left hand side equals to $\partial_\nu h(x, t(x)u_k) \cdot u_k$ by the mean value theorem. Let $\tilde{u}_k=u_k/\| u_k\|_{L^2(\Si_\infty)}$ be the renormalizations, then standard elliptic estimates imply that $\tilde{u}_k$ converges smoothly to a nontrivial $\varphi\in C^\infty(\Si_\infty)$ such that $\partial_\nu h \cdot \varphi=L_{\Si_\infty} \varphi$. This is the same as $L^h_{\Si_\infty} \varphi =0$, so we finish proving (\rom{4}).

\vspace{1em}
{\bf Part 4:} Assuming $h_\infty\equiv 0$, then $\Si_\infty$ is an embedded minimal hypersurface. Assume without loss of generality that $\Si_\infty$ is connected with multiplicity $m\in\N$. Suppose the Morse index $\ind(\Si_\infty)\geq I+1$, then by similar argument as in (\rom{3}), we can deduce a contradiction. In particular, by \cite[Proposition 4.3]{Marques-Neves16}, there exist $c_0\in (0, 1)$ and $\{F_v: v\in \overline{B}^{I+1}\}\subset \Diff(M)$ such that $-\frac{1}{c_0}\Id \leq D^2 \Area(F_v(\Si_\infty))\leq -c_0\Id$ for all $v\in \overline{B}^{I+1}$. Since $\Si_k$ converges to $m\cdot \Si_\infty$ as varifolds, and since $h_k\to 0$ uniformly, we know that $\mathcal A^{h_k}(F_v(\Om_k))$ converges to $m\cdot \Area(F_v(\Si_\infty))$ smoothly as functions on $\overline{B}^{I+1}$.  Therefore, for $k$ large enough, $\Om_k$ is $(I+1)$-unstable, which is a contradiction. So we finish proving (\rom{5}).
\end{proof}
%Next, we show that the Morse index of $\Si_\infty$ is bounded from above by $I$ (in the sense of Definition \ref{D:weak index}). Denote $\Si_\infty=\partial\Om_\infty$ for some open set $\Om_\infty$ in the Alexandrov sense. Assume by contradiction that $\ind(\Si_\infty)\geq I+1$, then there exists $c_0\in (0, 1)$ and $\{F_v: v\in \overline{B}^{I+1}\}\subset \Diff(M)$ such that $-\frac{1}{c_0}\Id \leq D^2 \Ah(F_v(\Om_\infty))\leq -c_0\Id$ for all $v\in \overline{B}^{I+1}$.  The locally smooth and graphical convergence implies that $\Om_k$ converges to 

There is also a theorem analogous to the above one in the setting of changing ambient metrics on $M$; see \cite[Theorem A.6]{Sharp17} for a similar result for minimal hypersurfaces. The proof proceeds the same way when one realizes that the constant $C_1$ in Theorem \ref{T:Curvature estimates for weakly stable PMC} depends only on the $\|g\|_{C^4}$ when $g$ is allowed to change.

\begin{theorem}
\label{T:compactness with changing metrics}
Let $M^{n+1}$ be a closed manifold of dimension $3\leq (n+1)\leq 7$, and $\{g_k\}_{k\in\N}$ be a sequence of metrics on $M$ that converges smoothly to some limit metric $g$. Let $\{h_k\}_{k\in \N}$ be a sequence of smooth functions with $h_k\in \mathcal S(g_k)$ that converges smoothly to some limit $h_\infty \in C^\infty(M)$. Let $\{\Si_k\}_{k\in\N}$ be a sequence of hypersurfaces with $\Si_k\in \PMC^{h_k}(\La, I)$ for some fixed $\La>0$ and $I\in\N$. Then there exists a smooth, closed, almost embedded hypersurface $\Si_\infty$ with prescribing mean curvature $h_\infty$, such that all properties (\rom{1})(\rom{2})(\rom{3}) in the above theorem are satisfied.
\end{theorem}

%%%%%%%%%%%%%%%%%%%%%%%%%%%%%%%%%%
% Section 3   	 		  Index bound	                      	   %
%%%%%%%%%%%%%%%%%%%%%%%%%%%%%%%%%%

\section{Morse index upper bound}
\label{S:Morse index upper bound}

In this part, we will establish Morse index upper bound for min-max PMC hypersurfaces obtained in Theorem \ref{T:main min-max theorem}. We will follow closely the strategy of Marques-Neves \cite[Theorem 1.2]{Marques-Neves16}, where they proved Morse index upper bound for min-max minimal hypersurfaces. Recall that the Morse index of an almost embedded PMC hypersurface $\Si$ is given in Definition \ref{D:weak index}.

\begin{theorem}
\label{T:main index bound thm}
Let $(M^{n+1}, g)$ be a closed Riemannian manifold of dimension $3\leq (n+1)\leq 7$, and $h\in\mathcal{S}(g)$ which satisfies $\int_M h\geq 0$.  Given a $k$-dimensional cubical complex $X$ and a subcomplex $Z\subset X$, let $\Phi_0 
: X\to \C(M)$ be a map continuous in the $\mF$-topology, and $\Pi$ be the associated $(X, Z)$-homotopy class of $\Phi_0$. Suppose
\begin{equation}
\label{E:nontrivial assumption2}
\bL^h(\Pi)>\max_{x\in Z}\Ah(\Phi_0(x)).
\end{equation}
%Let $\{\Phi_i\}_{i\in\N}\in\Pi$ be a min-max sequence for $\Pi$. Then there exists a nontrivial, smooth, closed, almost embedded hypersurface $\Si^n\subset M$ with $|\Si|\in \bC(\{\Phi_i\}_{i\in \N})$, such that 
Then there exists a nontrivial, smooth, closed, almost embedded hypersurface $\Si^n\subset M$, such that 
\begin{itemize}
\item $\Si$ is the boundary of some $\Om\in\C(M)$ where its mean curvature with respect to the unit outer normal of $\Om$ is $h$, i.e.
\[ H_{\Si} = h\vert_\Si, \]
\item $\Ah(\Om)=\bL^h(\Pi)$,
\item $\ind(\Si)\leq k$.
\end{itemize}
\end{theorem}

\subsection{Preliminary lemmas}

Let $h\in \mathcal{S}(g)$. Assume that $\Si_0=\partial\Om_0\in \PMC^h$ is $k$-unstable in an $\ep$-neighborhood, $\ep>0$. Let $\{F_v\}_{v\in \overline{B}^k}$ be the associated smooth family given in Definition \ref{D:k-unstable}.

The first lemma is a counterpart of \cite[Lemma 4.4]{Marques-Neves16}.
\begin{lemma}
\label{L:index bound pre lemma1}
There exists $\bar{\eta}=\bar{\eta}(\ep, \Si_0, \{F_v\})>0$, such that if $\Om\in\C(M)$ with $\mF(\Om, \Om_0)\geq \ep$ satisfies
\[ \Ah(F_v(\Om))\leq \Ah(\Om)+\bar{\eta} \]
for some $v\in \overline{B}^k$, then $\mF(F_v(\Om), \Om_0)\geq 2\bar{\eta}$.
\end{lemma}
\begin{proof}
Assume by contradiction that there exist $\Om_i$, $\mF(\Om_i, \Om_0)\geq \ep$ satisfying
\[ \Ah(F_{v_i}(\Om_i))\leq \Ah(\Om_i)+\frac{1}{i} \]
for some $v_i\in \overline{B}^k$, but $\mF(F_{v_i}(\Om_i), \Om_0)\leq \frac{2}{i}$.

Denote $v=\lim v_i$, and pass to the limit as $i\to\infty$, then $\Om_i\to F_{-v}(\Om_0)$ in $\mF$-metric, and $\Ah(\Om_0)\leq \Ah(F_{-v}(\Om_0))$, which implies that $v=0$; hence $\Om_i\to \Om_0$ in the $\mF$-metric, which is a contradiction.
\end{proof}

For each $\Om\in\overline{\bB}^{\mF}_{2\ep}(\Om_0)$, %let $\Ah_\Om(v)$ be the associated function in Definition \ref{D:k-unstable}
consider the one-parameter flow $\{\phi^\Om(\cdot, t): t\geq 0\}\subset \Diff(\overline{B}^k)$ generated by the vector field
\[ u\to -(1-|u|^2)\nabla \Ah_\Om(u),\quad u\in \overline{B}^k. \]
When $u\in \overline{B}^k$ is fixed, the function $t\to \Ah_\Om(\phi^\Om(u, t))$ is non-increasing.

The following lemma is a variant of \cite[Lemma 4.5]{Marques-Neves16}, and the proof is recorded in Appendix \ref{A:a calculus lemma}.
\begin{lemma}
\label{L:index bound pre lemma2}
For any $\de<1/4$ there exists $T=T(\de, \ep, \Om_0, \{F_v\}, c_0)\geq 0$ such that for any $\Om\in\overline{\bB}^{\mF}_{2\ep}(\Om_0)$ and $v\in \overline{B}^k$ with $|v-m(\Om)|\geq \de$, we have
\[ \Ah_{\Om}(\phi^\Om(v, T))<\Ah_\Om(0)-\frac{c_0}{10}\, \text{ and }\, |\phi^\Om(v, T)|>\frac{c_0}{4}. \]
\end{lemma}

\subsection{Deformation theorem}

Taking a min-max sequence $\{\Phi_i\}_{i\in\N}$, we will prove a deformation theorem as an adaption of \cite[Theorem 5.1]{Marques-Neves16}  to our setting. Recall that $\PMC^h$ denotes the class of smooth, closed, almost embedded hypersurface $\Si\subset M$ represented as boundary $\Si=\partial\Om$, and of prescribed mean curvature $h$.

Fix a $\si>0$ such that $\bL^h-\sup_{x\in Z}\Ah(\Phi_0(x))>2\si$. Denote 
\[ X_{i, \si}=\{x\in X, \text{ such that } \Ah(\Phi_i(x))\geq \bL^h-\si \}. \]
Note that when $i$ is sufficiently large, $ X_{i, \si}\subset X\backslash Z$.

Now we present the deformation theorem, and the proof follows closely that of \cite[Theorem 5.1]{Marques-Neves16}. Given two subsets $A, B\subset \C(M)$, we denote
\[ \mF(A, B) := \inf\{ \mF(\Om_A, \Om_B): \Om_A\in A, \Om_B\in B \}. \]

\begin{theorem}
\label{T:deformation theorem}
Suppose that
\begin{enumerate}[label=(\alph*)]
\item $\Si=\partial\Om \in \PMC^h$ is $(k+1)$-unstable;
\item $K\subset \C(M)$ is a subset, so that $\mF(\{\Om\}, K)> 0$ and $\mF(\Phi_i(X_{i, \si}), K)>0$ for all $i\geq i_0$;
\item $\Ah(\Om)=\bL^h$.
\end{enumerate}
Then there exist $\bar{\ep}>0$, $j_0\in\N$, and another sequence $\{\Psi_i\}_{i\in\N}$, $\Psi_i: X\to (\C(M), \mF)$, so that
\begin{enumerate}[label=(\roman*)]
\item $\Psi_i$ is homotopic to $\Phi_i$ in the $\mF$-topology for all $i\in \N$ and $\Psi_i\vert_Z=\Phi_i\vert_Z$ for $i\geq j_0$;
\item $\bL^h(\{\Psi_i\})\leq \bL^h$;
\item %$\Psi_i(X_{i, \si})\cap (\overline{\bB}^{\mF}_{\bar\ep}(\Om) \cup \overline{\bB}^{\mF}_{d/2}(K))=\emptyset$ for all $i\geq j_0$.
$\mF(\Psi_i(X_{i, \si}), \overline{\bB}^{\mF}_{\bar\ep}(\Om)\cup K)>0$ for all $i\geq j_0$.
%$|\partial \Psi_i(X_{i, \ep})| \cap (\overline{B}^\mF_{\bar{\ep}}(\Si) \cup K)=\emptyset$.
\end{enumerate}
\end{theorem}
\begin{proof}
Denote $d=\mF(\{\Om\}, K)>0$.

By (a), $\Si$ is $(k+1)$-unstable in some $\ep$-neighborhood. Let $\{F_v\}_{v\in \overline{B}^{k+1}}$, $c_0$ be the associated family and constant as in Definition \ref{D:k-unstable}. By possibly changing $\ep$, $\{F_v\}$, $c_0$, we can assume that 
\begin{equation}
\label{E:mF distance after deformation}
\inf\{ \mF( F_v(\Om'), K), v\in \overline{B}^{k+1} \} > \frac{d}{2},\, \text{ for all } \Om'\in \overline{\bB}^{\mF}_{2\ep}(\Om).
\end{equation}

Let $X(k_i)$ be a sufficiently fine subdivision of $X$ so that $\mF(\Phi_i(x), \Phi_i(y))<\de_i$ for any $x, y$ belonging to the same cell in $X(k_i)$ with $\de_i=\min\{ 2^{-(i+k+2)}, \ep/4\}$. We can also assume that
\[ |m(\Phi_i(x))-m(\Phi_i(y))|<\de_i \]
for any $x, y$ with $\mF(\Phi_i(x), \Om)\leq 2\ep$, $\mF(\Phi_i(y), \Om)\leq 2\ep$, and belonging to the same cell in $X(k_i)$.  

For $\eta>0$, let $U_{i, \eta}$ be the union of all cells $\si\in X(k_i)$ so that $\mF(\Phi_i(x), \Om)<\eta$ for all $x\in \si$. Then $U_{i, \eta}$ is a subcomplex of $X(k_i)$. If a cell $\beta\notin U_{i, \eta}$, then $\mF(\Phi_i(x'), \Om)\geq \eta$ for some $x'\in\beta$. Therefore, $\mF(\Phi_i(x), \Om)\geq \eta-\de_i$ for all $x\in \beta$. By (c) (after possibly shrinking $\ep$), we can assume
\[ U_{i, 2\ep}\subset X_{i, \si}. \]

For each $i\in \N$ and $x\in U_{i, 2\ep}$, we simply denote $\Ah_{i, x}= \Ah_{\Phi_i(x)}$, $m_i(x)=m(\Phi_i(x))$ and $\phi_{i, x}=\phi^{\Phi_i(x)}$. The function $m_i: U_{i, 2\ep}\to \overline{B}^{k+1}$ is continuous, and the two families $\{\Ah_{i, x}\}_{x\in U_{i, 2\ep}}$, $\{\phi_{i, x}\}_{x\in U_{i, 2\ep}}$ are continuous in $x$. Following \cite[5.1]{Marques-Neves16} we can define a continuous map
\[ \hat{H}_i: U_{i, 2\ep}\times [0, 1]\to B^{k+1}_{1/2^i}(0),\, \text{ so that $\hat{H}_i(x, 0)=0$ for all $x\in U_{i, 2\ep}$} \]
and
\begin{equation}
\label{E:end point values of hatH}
\inf_{x\in U_{i, 2\ep}}|\hat{H}_i(x, 1)-m_i(x)|\geq \eta_i>0,\, \text{ for some } \eta_i>0.
\end{equation} %one can do this because the dimension of $U_{i, 2\ep}$ is $\leq k$.
The construction here is the same so we omit details. The crucial ingredient is the fact that $U_{i, 2\ep}$ has dimension less than or equal to $k$ while the image set $\overline{B}^{k+1}$ has dimension $k+1$.

Let $c:[0, \infty)\to [0, 1]$ be a cutoff function which is non-increasing, equals to $1$ in a neighborhood of $[0, 3\ep/2]$, and $0$ in a neighborhood of $[7\ep/4, +\infty)$. For $y\notin U_{i, 2\ep}$, $\mF(\Phi_i(y), \Om)\geq 2\ep-\de_i\geq 7\ep/4$. Hence
\[ c(\mF(\Phi_i(y), \Om))=0,\, \text{ for all } y\notin U_{i, 2\ep}. \]

Consider the map $H_i: X\times [0, 1]\to B^{k+1}_{2^{-i}}(0)$ defined as
\[ H_i(x, t)=\hat{H}_i(x, c(\mF(\Phi_i(x), \Om))t),\, \text{ if } x\in U_{i, 2\ep} \]
and
\[ H_i(x, t)=0,\, \text{ if } x\in X\backslash U_{i, 2\ep}. \]
Then $H_i$ is continuous.

With $\eta_i$ as given in (\ref{E:end point values of hatH}), let $T_i=T(\eta_i, \ep, \Om, \{F_v\}, c_0)\geq 0$ be given by Lemma \ref{L:index bound pre lemma2}. Now we set $D_i: X\to \overline{B}^{k+1}$ such that
\[ D_i(x)=\phi_{i, x}(H_i(x, 1), c(\mF(\Phi_i(x), \Om))T_i),\, \text{ if } x\in U_{i, 2\ep} \]
and
\[ D_i(x)=0, \, \text{ if } x\in X\backslash U_{i, 2\ep}. \]
Then $D_i$ is continuous. 

Define
\[ \Psi_i: X\to \C(M), \quad \Psi_i(x)=F_{D_i(x)}(\Phi_i(x)). \]
In particular,
\[ \Psi_i(x)=\Phi_i(x), \, \text{ if }  x\in X\backslash U_{i, 2\ep}.\]
Hence $\Psi_i\vert_Z= \Phi_i\vert_Z$ for $i$ sufficiently large. 

Note that the map $D_i$ is homotopic to the zero map in $\overline{B}^{k+1}$, so $\Psi_i$ is homotopic to $\Phi_i$ in the $\mF$-topology for all $i\in \N$. Up to here, we proved (\rom{1}).

\vspace{0.5em}
\noindent{\bf Claim 1:} $\bL^h(\{\Psi\}_{i\in\N})\leq \bL^h$.
\vspace{0.5em}

By the non-increasing property of $t \to \Ah_{i, x}(\phi_{i, x}(u, t))$, we have that for all $x\in X$,
\[ \Ah(\Psi_i(x))\leq \Ah( F_{H_i(x, 1)}(\Phi_i(x)) ). \]
Using the fact that $H_i(x, 1)\in B^{k+1}_{1/2^i}(0)$ for all $x\in X$ and that $\|F_v-\Id\|_{C^2}\to 0$ uniformly as $v\to 0$, we have that
\begin{equation}
\label{E:Ah change after first deformation}
\lim_{i\to\infty} \sup_{x\in X} \big\vert \Ah(\Phi_i(x)) - \Ah( F_{H_i(x, 1)}(\Phi_i(x)) ) \big\vert =0,
\end{equation}
and this finishes proving Claim 1.

\vspace{0.5em}
\noindent{\bf Claim 2:} There exists $\bar{\ep}>0$, such that for all sufficiently large $i$, $\mF(\Psi_i(X), \Om)> \bar{\ep}$.
\vspace{0.5em}

There are three cases. If $x\in X\backslash U_{i, 2\ep}$, then $\Psi_i(x)=\Phi_i(x)$ and so $\mF(\Psi_i(x), \Om)\geq \frac{7\ep}{4}$.

If $x\in U_{i, 2\ep}\backslash U_{i, 5\ep/4}$, then $\mF(\Phi_i(x), \Om)\geq \ep$. The non-increasing property of $t \to \Ah_{i, x}(\phi_{i, x}(u, t))$ implies
\[ \Ah(\Psi_i(x))=\Ah( F_{D_i(x)}(\Phi_i(x)) ) \leq \Ah( F_{H_i(x, 1)}(\Phi_i(x)) ). \]
From (\ref{E:Ah change after first deformation}), we have that for $i$ large enough,
\[ \Ah( F_{H_i(x, 1)}(\Phi_i(x)) ) \leq \Ah(\Phi_i(x))+\bar{\eta}, \, \text{ for all } x\in X, \]
where $\bar{\eta}=\bar{\eta}(\ep, \Om, \{F_v\})>0$ is given by Lemma \ref{L:index bound pre lemma1}. Combining the two inequalities with Lemma \ref{L:index bound pre lemma1} applied to $\Phi_i(x)$, $v=D_i(x)$, we get $\mF(\Psi_i(x), \Om)\geq 2\bar{\eta}$.

Finally when $x\in U_{i, 5\ep/4}$, $c( \mF(\Phi_i(x), \Om))=1$. Hence by Lemma \ref{L:index bound pre lemma2} (with $\de=\eta_i$, $\Om=\Phi_i(x)$, $v=H_i(x, 1)$) we have
\[ \Ah(\Psi_i(x))=\Ah_{i, x}( \phi_{i, x}(H_i(x, 1), T_i) )< \Ah_{i, x}(0)-\frac{c_0}{10} =\Ah(\Phi_i(x))-\frac{c_0}{10}.  \]
Note that there exists $\bar{\ga}=\bar{\ga}(\Om, c_0)$ so that
\[ \Ah(\Om')\leq \Ah(\Om)-\frac{c_0}{20} \Longrightarrow \mF(\Om', \Om)\geq 2\bar{\ga}. \]
By assumption (c), we can choose $i$ sufficiently large so that
\[ \sup_{x\in X} \Ah(\Phi_i(x))\leq \Ah(\Om)+\frac{c_0}{20}. \]
So
\[ \Ah(\Psi_i(x))\leq \Ah(\Om)-\frac{c_0}{20}. \]
This implies that $\mF(\Psi_i(x), \Om)\geq 2\bar{\ga}$, and hence ends the proof of Claim 2.

\vspace{0.5em}
\noindent{\bf Claim 3:} For all $i$, $\mF(\Psi_i(X_{i, \si}), K)>0$.
\vspace{0.5em}

If $x\in X_{i, \si}\backslash U_{i, 2\ep}$, then $\Psi_i(x)=\Phi_i(x)$ and so $\mF(\Psi_i(X_{i, \si}\backslash U_{i, 2\ep}), K)>0$. If $x\in U_{i, 2\ep}$, then $\mF(\Phi_i(x), \Om)\leq 2\ep$, and by (\ref{E:mF distance after deformation}) we have $\mF(\Psi_i(x), K)\geq \frac{d}{2}$. So we finish proving Claim 3, and hence the theorem.
\end{proof}

\subsection{Proof of Morse index upper bound}

Let $M^{n+1}$ be a closed manifold of dimension $3\leq (n+1)\leq 7$. A pair $(g, h)$ consisting of a Riemannian metric $g$ and a smooth function $h\in C^\infty(M)$ is called a {\em good pair}, if
\begin{itemize} 
\item $h\in \mathcal S(g)$, i.e. $h$ is Morse and the zero set $\{h=0\}$ is a smooth embedded hypersurface in $M$ with mean curvature $H$ vanishing to at most finite order, and
\item $g$ is {\em bumpy} for $\PMC^h$, i.e. every $\Si\in \PMC^h$ is properly embedded (no self-touching), and is nondegenerate (nullity equal to zero). 
\end{itemize}

Denote $\mathcal S_0$ as the class of smooth functions $h\in C^\infty(M)$ such that $h$ is Morse and the zero set $\{h=0\}$ is a smooth embedded hypersurface. $\mathcal S_0$ is open and dense in $C^\infty(M)$, and is independent of the choice of a metric; (see \cite[Proposition 3.8]{Zhou-Zhu18}).

\begin{lemma}
\label{L:good pairs are generic}
Given $h\in \mathcal S_0$, the set of Riemannian metrics $g$ on $M$ with $(g, h)$ as a good pair is generic in the Baire sense.
\end{lemma} 
\begin{proof}
By the proof of \cite[Proposition 3.8]{Zhou-Zhu18}, we know that the set of metrics $g$ under which $\{h=0\}$ has mean curvature vanishing to at most finite order is an open and sense subset. In particular, openness follows as small smooth perturbations of $g$ will bound the order of vanishing of $H_{\{h=0\}}$. To show denseness, note that it is proved in \cite[Proposition 3.8]{Zhou-Zhu18} for any $h\in \mathcal S_0$ and any metric $g$, one can first perturb $g$ slightly so that $\{h=0\}$ is not a minimal hypersurface, and then there exists a flow $\{F_t: t\in(-\ep, \ep)\}\subset \Diff(M)$ supported near $\{h=0\}$, such that the zero set of $h\circ (F_t)^{-1}$ has mean curvature vanishing to at most finite order for $t>0$. That is to say the zero set $\{h=0\}$ satisfies the requirement for the pull-back metrics $F_t^* g$.

In a series of celebrated papers \cite{White91, White17, White18}, White proved that for a fixed $h\in\mathcal S_0$, the set of metrics under which 
all closed, simple immersed PMC's are non-degenerate and self-transverse is generic in the Baire sense.  %more reasoning
%all immersed closed PMC's are nondegenerate and all simple immersed closed PMC's are self-transverse is generic in the Baire sense. 
In fact, White proved in \cite[Section 7]{White91} that the set of metrics under which all closed, simple immersed CMC hypersurfaces are non-degenerate is generic, and the proof is the same in a smooth neighborhood of an arbitrary pair $(g, h)$ when $h\in \mathcal S(g)$, hence the result follows as the set of $g$ where $h\in\mathcal S(g)$ is open and dense.  In \cite[Theorem 33]{White18}, White further proved self-transverse property for a generic set of metrics. 
Our almost embedded hypersurfaces are simple immersed. So for such generic metrics, almost embedded PMC's are properly embedded.

To finish the proof, we take the intersection of the two generic sets of metrics, which is still generic in the Baire sense.
\end{proof}  

The following theorem is a counterpart of \cite[Theorem 6.1]{Marques-Neves16}, and the proof follows closely. We remark that by Theorem \ref{T:compactness with bounded index}(\rom{4}), if $(g, h)$ is a good pair, then there are only finitely many elements in $\PMC^h(\La, I)$.

\begin{theorem}
\label{T:generic index bound thm}
Assume that $(g, h)$ is a good pair and let $\{\Phi_i\}_{i\in\N}$ be a min-max sequence of $\Pi$ such that $\bL^h(\{\Phi_i\}_{i\in\N})=\bL^h(\Pi)=\bL^h$ and (\ref{E:nontrivial assumption2}) is satisfied.

There exists a smooth, closed, properly embedded hypersurface $\Si=\partial\Om \in \bC(\{\Phi_i\}_{i\in\N})$ such that $\Si\in \PMC^h$ with
\[ \bL^h(\Pi)=\Ah(\Om), \text{ and } \ind(\Si)\leq k. \]
\end{theorem}

\begin{proof}
By the finiteness remark above, it suffices to show that, for every $r>0$, there is a $\tilde{\Si}=\partial\tilde{\Om} \in \PMC^h$ such that $\mF( [\tilde{\Si}], \bC(\{\Phi_i\}_{i\in\N}) )<r$,
\[ \bL^h(\Pi)= \Ah(\tilde{\Om}),\, \text{ and } \ind(\tilde{\Si})\leq k. \] %show more arguments

Denote by $\mathcal W$ the set of all $\tilde{\Si}=\partial \tilde{\Om}\in \PMC^h$ with $\Ah(\tilde{\Om})=\bL^h$ and by $\mathcal W(r)$ the set
\[ \{ \Si\in \mathcal W: \mF( [\Si], \bC(\{\Phi_i\}_{i\in\N}))\geq r \}. \]
\begin{lemma}
There exist $i_0\in\N$ and $\bar{\ep}_0>0$ such that $\mF(\Phi_i(X), \mathcal W(r) )> \bar\ep_0$ for all $i\geq i_0$.
\end{lemma} %Being $\mF$-distance $r$-away from $\bC$ is stronger somehow than current-$\mF$-distance away.
\begin{proof}
Suppose by contradiction for some subsequence $\{j\}\subset \{i\}$, $x_j\in X$, $\tilde{\Si}_j=\partial\tilde\Om_j\in \mathcal W(r)$ so that 
\[ \lim_{j\to\infty} \mF(\Phi_j(x_j), \tilde{\Om}_j) =0. \]
 Since $\Ah(\tilde\Om_j)\equiv \bL^h$, we have $\lim_{j\to\infty} \Ah(\Phi_j(x_j))=\bL^h$. %We claim that $\lim_{j\to\infty}\Ah(\Phi_j(x_j))=\bL^h$. Indeed, $\{\tilde{\Om}_j\}$ may not have a subsequence limit in the $\mF$-topology, but we can introduce a pair $(V_\infty, \Om_\infty)$ such that $V_\infty=\lim_{j\to\infty}|\partial\tilde{\Om}_i|$ as varifolds, and $\Om_\infty=\lim_{j\to\infty}\tilde{\Om}_j$ as Caccioppoli sets. We then have $\bL^h=\|V_\infty\|(M)- \int_{\Om_\infty}h d\mH^{n+1}$. As $|\partial\Phi_j(x_j)|\to V_\infty$ and $\Phi_j(x_j)\to \Om_\infty$ we have $\Ah(\Phi_j(x_j))\to \bL^h$.
 Hence a subsequence $|\partial\Phi_j(x_j)|$ will converge as varifolds to some $V\in \bC(\{\Phi_i\}_{i\in\N})$, which is a contradiction to $\mF(|\partial \tilde{\Om}_i|, \bC(\{\Phi_i\}_{i\in\N}))\geq r$.
\end{proof}

Denote $\mathcal W^{k+1}$ as the collection of elements in $\mathcal W$ with index greater than or equal to $(k+1)$. As $(g, h)$ is a good pair, this set is countable by the remark above the theorem, 
%(by Theorem \ref{T:compactness with bounded index}(\rom{4}) and the definition of good pair),
and we can write $\mathcal W^{k+1}\backslash \overline{\bB}^{\mF}_{\bar\ep_0}(\mathcal W(r))=\{\Si_1, \Si_2, \cdots\}$, where $\Si_i=\partial\Om_i$. %note that we have to choose $2\bar\ep_0$: on one hand, this implies $\mF(\Si_i, \overline{\bB}^\mF_{\bar\ep_0}(\mathcal W(r)))\geq \bar\ep_0$, on the other hand, $\bC(\{\Phi_i\}_{i\in\N})$ does not intersect $\overline{\bB}^{\mF}_{2\bar\ep_0}(\mathcal W(r))$.
%Need to take $\bar\ep_0$, as at the end we want to say that $\Psi_l(X)$ are away from both $\mathcal W^{k+1} \backslash \overline{\bB}^{\mF}_{\bar\ep_0}(\mathcal W(r)$ and also away from $\mathcal W^{k+1}\ cap \overline{\bB}^{\mF}_{\bar\ep_0}(\mathcal W(r)$.
%Note that $\mathcal W(r)$ may not be compact in $\mF$-topology, so $\overline{\bB}^{\mF}_{\bar\ep_0}(\mathcal W(r))$ is understood as $\C(M)\setminus \{\Om: \mF(\Om, \mathcal W(r)) >\bar\ep_0\}$.
Note that by possibly perturbing $\bar{\ep}_0$, we can make sure $\mathcal W^{k+1}\cap \partial \overline{\bB}^{\mF}_{\bar\ep_0}(\mathcal W(r))=\emptyset$.

Using Theorem \ref{T:deformation theorem} (we can take $X_{i, \si}$ to be $X$) with $K=\overline{\bB}^\mF_{\bar\ep_0}(\mathcal W(r))$ and $\Si=\Si_1$, we find $\bar\ep_1>0$, $i_1\in\N$, and $\{\Phi^1_i\}_{i\in\N}$ so that %note that if $\Om_1\notin K$, then $\mF(\Om_1, K)>0$, and this can be proven by taking an approximation sequence.
\begin{itemize}
\item $\Phi^1_i$ is homotopic to $\Phi_i$ in the $\mF$-topology for all $i\in\N$ and $\Phi^1_i\vert_Z=\Phi_i\vert_Z$ for $i\geq i_1$;
\item $\bL^h(\{\Phi^1_i\}_{i\in\N})\leq \bL^h$;
\item %$\Phi^1_i(X)\cap (\overline{\bB}^\mF_{\bar\ep_1}(\Om_1) \cup \overline{\bB}^\mF_{3\bar\ep_0/4}(\mathcal W(r)))=\emptyset$ for $i\geq i_1$;
$\mF(\Phi^1_i(X), \overline{\bB}^\mF_{\bar\ep_1}(\Om_1) \cup\overline{\bB}^\mF_{\bar\ep_0}(\mathcal W(r)) ) >0$ for $i\geq i_1$.
\item no $\Om_j$ belongs to $\partial \overline{\bB}^\mF_{\bar\ep_1}(\Om_1)$.
%satisfies $\mF(\Om_j, \Om_1)=\bar\ep_1$.
\end{itemize}

We consider $\Si_2$ now. If $\Om_2\notin \overline{\bB}^\mF_{\bar\ep_1}(\Om_1)$, we apply Theorem \ref{T:deformation theorem} with $K=\overline{\bB}^\mF_{\bar\ep_1}(\Om_1) \cup \overline{\bB}^\mF_{\bar\ep_0}(\mathcal W(r))$, $\Si=\Si_2$, and find $\bar\ep_2>0$, $i_2\in\N$, and $\{\Phi^2_i\}_{i\in\N}$ so that
\begin{itemize}
\item $\Phi^2_i$ is homotopic to $\Phi_i$ in the $\mF$-topology for all $i\in\N$ and $\Phi^2_i\vert_Z=\Phi_i\vert_Z$ for $i\geq i_2$;
\item $\bL^h(\{\Phi^2_i\}_{i\in\N})\leq \bL^h$;
\item $\mF(\Phi^2_i(X), \overline{\bB}^\mF_{\bar\ep_1}(\Om_1) \cup \overline{\bB}^\mF_{\bar\ep_2}(\Om_2) \cup \overline{\bB}^\mF_{\bar\ep_0}(\mathcal W(r)) )>0$ for $i\geq i_2$;
%$\mF(\Phi^2_i(X), \{\Om_1, \Om_2\}\cup \mathcal W(r)) \geq \bar\ep_2$ for $i\geq i_2$.
\item no $\Om_j$ belongs to $\partial \overline{\bB}^\mF_{\bar\ep_1}(\Om_1) \cup \partial \overline{\bB}^\mF_{\bar\ep_2}(\Om_2)$.%satisfies $\mF(\Om_j, \Om_1)=\bar\ep_1$ nor $\mF(\Om_j, \Om_2)=\bar\ep_2$.
\end{itemize}
If $\mF(\Om_2, \Om_1)<\bar\ep_1$, we skip it and repeat the construction with $\Si_3$. 

By induction there are two possibilities. We can find for all $l\in \N$ a sequence $\{\Phi^l_i\}_{i\in\N}$, $\bar\ep_l>0$, $i_l\in\N$, and $\Si_{j_l}\in \mathcal W^{k+1}\backslash \overline{\bB}^{\mF}_{\bar\ep_0}(\mathcal W(r))$ for some subsequences $\{j_l\}\subset\N$ so that
\begin{itemize}
\item $\Phi^l_i$ is homotopic to $\Phi_i$ in the $\mF$-topology for all $i\in\N$ and $\Phi^l_i\vert_Z=\Phi_i\vert_Z$ for $i\geq i_l$;
\item $\bL^h(\{\Phi^l_i\}_{i\in\N})\leq \bL^h$;
\item $\mF(\Phi^l_i(X), \cup_{q=1}^l\overline{\bB}^\mF_{\bar\ep_q}(\Om_{j_q}) \cup \overline{\bB}^\mF_{\bar\ep_0}(\mathcal W(r)) ) >0$ for $i\geq i_l$;
\item $\{\Om_1, \cdots, \Om_l\}\subset \cup_{q=1}^l\overline{\bB}^\mF_{\bar\ep_q}(\Om_{j_q})$;
\item no $\Om_j$ belongs to $\partial \overline{\bB}^\mF_{\bar\ep_q}(\Om_{j_q})$ for all $q=1, \cdots, l$.
\end{itemize}
Or the process ends in finitely many steps. That means we can find some $m\in\N$, a sequence $\{\Phi^m_i\}_{i\in \N}$, $\bar\ep_1, \dots, \bar\ep_m>0$, $i_m\in\N$ and $\Si_{j_1}, \cdots, \Si_{j_m}\in \mathcal W^{k+1}\backslash \overline{\bB}^{\mF}_{\bar\ep_0}(\mathcal W(r))$ so that
\begin{itemize}
\item $\Phi^m_i$ is homotopic to $\Phi_i$ in the $\mF$-topology for all $i\in\N$ and $\Phi^m_i\vert_Z=\Phi_i\vert_Z$ for $i\geq i_m$;
\item $\bL^h(\{\Phi^m_i\}_{i\in\N})\leq \bL^h$;
\item $\mF(\Phi^m_i(X),  \cup_{q=1}^m\overline{\bB}^\mF_{\bar\ep_q}(\Om_{j_q}) \cup \overline{\bB}^\mF_{\bar\ep_0}(\mathcal W(r)) ) >0$ for $i\geq i_m$.
\item $\{\Om_j: j\geq 1\}\subset \cup_{q=1}^m \overline\bB^{\mF}_{\bar\ep_q}(\Om_{j_q})$.
\end{itemize}

In the first case we choose an increasing sequence $p_l\geq i_l$ so that
\[ \sup_{x\in X}\Ah(\Phi^l_{p_l})\leq \bL^h+\frac{1}{l}, \]
and set $\Psi_l=\Phi^l_{p_l}$. In the second case we set $p_l=l$ and $\Psi_l=\Phi^m_l$. The sequence $\{\Psi_l\}_{l\in\N}$ satisfies that
\begin{enumerate}[label=(\roman*)]
\item $\Psi_l$ is homotopic to $\Phi_{p_l}$ in the $\mF$-topology, and $\Psi_l\vert_Z=\Phi_{p_l}\vert_Z$ for all $l$;
%$\lim_{l\to\infty} \sup_{x\in Z}\mF(\Psi_l(x), \Phi_0(x))=0$;
\item $\bL^h(\{\Psi_l\}_{l\in\N})\leq \bL^h$;
\item given any subsequence $\{l_j\}\subset \{l\}$, $x_j\in X$, if $\lim_{j\to\infty} \Ah(\Psi_{l_j}(x_j))=\bL^h$, then $\{\Psi_{l_j}(x_j)\}_{j\in\N}$ does not converge in $\mF$-topology to any element in $\mathcal W^{k+1}\cup \mathcal W(r)$.
%$\bC(\{\Psi_l\}_{l\in\N})\cap (\mathcal W^{k+1}\cup \mathcal W(r))=\emptyset$.
\end{enumerate}

The Min-max Theorem \ref{T:main min-max theorem} applied to $\{\Psi_l\}_{i\in\N}$ implies that $\mathcal W\backslash (\mathcal W^{k+1}\cup \mathcal W(r))$ is not empty and this proves the theorem.
\end{proof}

\vspace{1em}
Now we can use the previous theorem and the Compactness Theorem \ref{T:compactness with changing metrics} to prove Theorem \ref{T:main index bound thm}.

\begin{proof}[Proof of Theorem \ref{T:main index bound thm}]
Given $(g, h)$ as in the theorem, then $h\in \mathcal S(g)\subset \mathcal S_0$. By Lemma \ref{L:good pairs are generic} there exists a sequence of metrics $\{g_j\}_{j\in\N}$ converging smoothly to $g$ such that $(g_j, h)$ is a good pair for all $j\in \N$. If $\bL^h_j=\bL^h_j(\Pi, g_j)$ is the $h$-width of $\Pi$ with respect to $g_j$, then the sequence $\{\bL^h_j\}_{j\in\N}$ tends to the $h$-width $\bL^h(\Pi, g)$ with respect to $g$, and for $j$ large enough (\ref{E:nontrivial assumption2}) is satisfied with $g_j$ in place of $g$. For each $j$ large enough, the previous theorem gives a properly embedded closed hypersurface $\Si_j=\partial\Om_j \in \PMC^{h}$ with $\A^{h_j}(\Om_j)=\bL^h_j$ and $\ind(\Si_j)\leq k$ (with respect to $g_j$). Let $\Si_\infty=\partial\Om_\infty$ be the limit of $\{\Si_j\}_{j\in\N}$ given in Theorem \ref{T:compactness with changing metrics}, then the locally smooth convergence implies that $\Ah(\Om_\infty)= \bL^h(\Pi, g)$ and $\ind(\Si_\infty)\leq k$.
\end{proof}

%%%%%%%%%%%%%%%%%%%%%%%%%%%%%%%%%%
% Section 4     Multiplicity one for sweepouts of boundaries 	     %
%%%%%%%%%%%%%%%%%%%%%%%%%%%%%%%%%%

\section{Min-max hypersurfaces associated with sweepouts of boundaries have multiplicity one in a bumpy metric}
\label{S:first multiplicity one result}

We present our first multiplicity one result. In particular, we will prove that the min-max minimal hypersurfaces associated with sweepouts of boundaries of Caccioppoli sets are two-sided and have multiplicity one in a bumpy metric.  We will approximate the area functional by the weighted $\A^{\ep h}$-functionals for some prescribing function $h$ when $\ep\to 0$. We know by Section \ref{S:Multi-parameter min-max for PMC} that the min-max PMC hypersurfaces are two-sided with multiplicity one, and we will prove that the limit minimal hypersurfaces (when $\ep\to 0$) are also two-sided and have multiplicity one by choosing the right prescribing function $h$.

Recall that a Riemannian metric $g$ is said to be {\em bumpy} if every smooth closed immersed minimal hypersurface is non-degenerate. White proved that the set of bumpy metrics is generic in the Baire sense \cite{White91, White17}. 
 
\begin{theorem}[Multiplicity one theorem for sweepouts of boundaries]
\label{T:multiplicity 1 for sweepouts of boundaries}
Let $(M^{n+1}, g)$ be a closed Riemannian manifold of dimension $3\leq (n+1)\leq 7$. Let $X$ be a $k$-dimensional cubical complex and $Z\subset X$ be a subcomplex, and $\Phi_0 : X\to \C(M)$ be a map continuous in the $\mF$-topology. Let $\Pi$ be the associated $(X, Z)$-homotopy class of $\Phi_0$. Assume that
\begin{equation}
\label{E:nontrivial assumption3}
\bL(\Pi)>\max_{x\in Z}\M(\partial\Phi_0(x)),
\end{equation}
where we let $h\equiv 0$ in Section \ref{SS:min-max construction in continuous setting}.

If $g$ is a bumpy metric, then there exists a disjoint collection of smooth, connected, closed, embedded, two-sided, minimal hypersurfaces $\Si=\cup_{i=1}^N\Si_i$, such that
\[ \bL(\Pi) =\sum_{i=1}^N \Area(\Si_i),\ \text{ and } \ind(\Si)=\sum_{i=1}^N\ind(\Si_i)\leq k. \]
In particular, each component of $\Si$ is two-sided and has exactly multiplicity one.
\end{theorem}

\begin{proof}
Pick a $h\in\mathcal S (g)$ with $\int_M h\geq 0$ (to be fixed at the end), and $\ep>0$ small enough so that 
\[ \bL(\Pi)-\max_{x\in Z}\M(\partial \Phi_0(x))>2 \ep \sup_{M} |h| \cdot \vol(M). \]

Note that we have for each $\Om\in\C(M)$
\begin{equation}
\label{E:comparison between M and Ah}
\M(\partial\Om)-\ep \sup_M |h| \cdot \vol(M) \leq \A^{\ep h}(\Om) \leq \M(\partial\Om) + \ep \sup_M |h| \cdot \vol(M).
\end{equation}

The above two inequalities imply that if we consider the $\A^{\ep h}$-functional in place of the mass $\M$-functional for the $(X, Z)$-homotopy class $\Pi$, we have
\[ \bL^{\ep h}(\Pi)> \max_{x\in Z} \A^{\ep h}(\Phi_0(x)). \]
Note that when $h\in \mathcal S(g)$, $\ep h$ also belongs to $\mathcal S(g)$. Therefore Theorem \ref{T:main index bound thm} applies to $\Pi$ and produces a nontrivial, smooth, closed, almost embedded hypersurface $\Si_\ep$, such that
\begin{itemize}
\item $\Si_\ep$ is the boundary for some $\Om_\ep\in\C(M)$ where its mean curvature with respect to the unit outer normal $\nu$ (of $\Om_\ep$) is $\ep\cdot h$, i.e.
\[ H_{\Si_\ep} = \ep \cdot h\vert_{\Si_\ep}; \]
\item $\A^{\ep h}(\Om_\ep)=\bL^{\ep h}(\Pi)$;
\item $\ind(\Si_\ep)\leq k$.
\end{itemize}

We denote $\bL=\bL(\Pi)$ and $\bL^\ep=\bL^{\ep h}(\Pi)$. In the following, we proceed the proof by parts.

\vspace{1em}
\noindent{\bf Part 1}: $\bL^{\ep} \to \bL$ when $\ep\to 0$.
\vspace{0.5em}

\textit{Proof}: From (\ref{E:comparison between M and Ah}), it is easy to see
\[ \bL -\ep\sup_M |h|\vol(M)\leq \bL^\ep \leq \bL +\ep\sup_M |h|\vol(M). \]

\begin{comment}
\textit{Proof of Claim}: Take a min-max sequence $\{\Phi_i\}_{i\in \N}$ for $\bL$, i.e. $\limsup_{i\to\infty}\sup_{x\in X}\M(\partial\Phi_i(x))=\bL$. Then for every $i\in\N$
\[ \sup_{x\in X}\A^{\ep h}(\Phi_i(x))\leq \sup_{x\in X}\M(\partial\Phi_i(x))+ \ep \sup_{M} |h| \vol(M). \]
So it follows that 
\[ \bL^\ep \leq \bL +\ep\sup_M |h|\vol(M). \]
On the other hand, take a min-max sequence $\{\Phi'_i\}_{i\in\N}$ for $\bL^\ep$, i.e. $\limsup_{i\to\infty}\sup_{x\in X}\A^{\ep h}(\Phi'_i(x))=\bL^\ep$. Then for every $i\in\N$
\[ \sup_{x\in X}\M(\partial\Phi'_i(x)) \leq \sup_{x\in X}\A^{\ep h}(\Phi'_i(x)) + \ep \sup_{M} |h| \vol(M). \]
So it follows that
\[ \bL\leq \bL^\ep+\ep\sup_M |h| \vol(M). \]
So this finished the proof of claim. \qed
\end{comment}

\vspace{1em}
\noindent{\bf Part 2}: By Theorem \ref{T:compactness with bounded index}, there exists a subsequence $\{\ep_k\}\to 0$, such that $\Si_k=\Si_{\ep_k}$ converges to some smooth, closed, embedded, minimal hypersurface $\Si_\infty$ (with integer multiplicity) in the sense of Theorem \ref{T:compactness with bounded index}(\rom{1})(\rom{2}). We denote $\mathcal Y$ as the set of points where the convergence fails to be smooth. In particular, by (\ref{E:comparison between M and Ah}) and Part 1 and Theorem \ref{T:compactness with bounded index}(\rom{5}), we have
\[ \M(\Si_\infty) = \bL, \, \text{ and } \ind(\Si_\infty)\leq k. \]
That is to say that $\Si_\infty$ is a min-max minimal hypersurface associated with $\Pi$.  

\vspace{1em}
Without loss of generality, we assume from Part 3 to Part 8 that $\Si_\infty$ has only one connected component.  If $\Si_\infty$ is 2-sided with the multiplicity equal to one, then we are done; otherwise we may assume that either the multiplicity $m>1$ or $\Si_\infty$ is 1-sided.

\vspace{1em}
\noindent{\bf Part 3}:  We first assume that $\Si_\infty$ is 2-sided. We will implicitly use exponential normal coordinates of $\Si_\infty$ with respect to one fixed unit normal of $\Si_\infty$. By the local, smooth graphical convergence $\Si_k \to \Si_\infty$ away from $\mathcal Y$, we know that there exists an exhaustion by compact domains $\{U_k \subset \Si_\infty\backslash \mathcal Y\}$ and some small $\de>0$, so that for $k$ large enough, $\Si_k \cap (U_k \times (-\de, \de))$ can be written as a set of $m$-normal graphs $\{u_k^1, \cdots, u_k^m: u_k^i \in C^\infty(U_k)\}$ over $U_k$, and such that
\[ u_k^1\leq u_k^2\leq \cdots \leq u_k^m, \text{ and } u_k^i \to 0, \text{ in smooth topology as } k\to\infty. \]
Since $\Si_k$ is the boundary of some set $\Om_k$, by the Constancy Theorem (applied to $\Om_k$ in $U_k \times (-\de, \de)$), we know that the unit outer normal $\nu_k$ of $\Om_k$ will alternate orientations along these graphs. In particular, if $\nu_k$ restricted to the graph of $u_k^i$ points upward (or downward), then $\nu_k$ restricted to the graph of $u_k^{i+1}$ will point downward (or upward).

\vspace{1em}
\noindent{\bf Part 4}: We first deal with an easier case: {\bf $m$ is an odd number}. Hence $m\geq 3$. In this case $\nu_k$ restricted to the bottom ($u_k^1$) and top ($u_k^m$) sheets point to the same side of $\Si_\infty$, and without loss of generality we may assume that $\nu_k$ points upward therein. That means:
\[ H\vert_{\Gr(u_k^m)}(x)= \ep_k h(x, u_k^m(x)), \text{ and } H\vert_{\Gr(u_k^1)}(x)= \ep_k h(x, u_k^1(x)),\, \text{ for } x\in U_k. \]
Here and in the following the sign convention is made so that $H\vert_{\Gr(u)}$ is defined with respect to the upward pointing normal of $\Gr(u)$, and hence the linearized operator is positively definite. 

Note that since $\ep h\in \mathcal S(g)$, by the Strong Maximum Principle \cite[Lemma 3.12]{Zhou-Zhu18} (applied to two sheets of the same orientation), we know 
\[ u_k^m(x)-u_k^1(x)>0,\, \text{ for all } x\in U_k. \]

Now by subtracting the above two equations, and using the fact $H\vert_{\Gr(u_k^m)}-H\vert_{\Gr(u_k^1)}= L_{\Si_\infty} (u_k^m-u_k^1) + o(u_k^m-u_k^1)$ (see \cite[page 331]{Sharp17} and part 3 in the proof of Theorem \ref{T:compactness with bounded index}), %reference or an appendix
we have
\begin{equation}
\label{E:height between two PMCs of the same orientation}
L_{\Si_\infty} (u_k^m-u_k^1) + o(u_k^m-u_k^1)= \ep_k\cdot \partial_\nu h(x,  v_k(x)) \cdot (u_k^m(x)-u_k^1(x)), 
\end{equation}
where $v_k(x) = t(x) u^m_k(x)+(1-t(x))u_k^1(x)$ for some $t(x)\in [0, 1]$. 

Now it is a standard argument to produce a nontrivial positive Jacobi field on $\Si_\infty \backslash \mathcal Y$. Let us present the details for completeness. Write $h_k= u_k^m-u_k^1$, and pick a fixed point $p\in U_1$. Let $\tilde h_k=h_k /h_k(p)$, then $\tilde h_k(p)=1$. By standard Harnack and elliptic estimates, $\tilde h_k$ will converge locally smoothly to a positive function $\varphi$ on any fixed $U\subset U_k$, and by a diagonalization process, we can extend $\varphi$ to $\Si_\infty \backslash \mathcal Y$, and such that
\[ L_{\Si_\infty} \varphi =0,\, \text{ outside } \mathcal Y. \] 
%The existence of a positive Jacobi field implies that $\Si_\infty \backslash \mathcal Y$ is stable, and hence $\Si_\infty$ is stable by a simple extension argument.

\vspace{1em}
\noindent{\bf Part 5}: Next we use White's local foliation argument \cite{White87} to prove that $\varphi$ extends smoothly across $\mathcal Y$, and this will contradict the bumpy assumption of $g$. 

Fix $y\in \mathcal Y$. We use the exponential normal coordinates $(x, z)\in \Si_\infty \times [-\de, \de]$. Let $\ep>0$ be as given in Proposition \ref{P:existence of local PMC foliations}. Fix a small radius $0<\eta<\ep$, and choose $k$ large enough such that $\|u_k^1\|_{2, \al}, \|u_k^m\|_{2, \al}\ll \ep\eta$ near $\partial B^n_{\eta}(y)$ so that some extensions of them to the whole $B^n_{\eta}(y)$ have $C^{2, \al}$-norms bounded by $\ep\eta$. Let $v_{k, t}^1, v_{k, t}^m : B^n_\eta(y) \to \R$, $t\in [-\eta, \eta]$, be the PMC local foliations associated with $\ep_k h$,
\[ H_{\Gr(v_{k, t}^i)}(x) = \ep_k h(x, v_{k, t}^i(x)),\, i=1, m, \, x\in B^n_\eta(y), \]
and 
\[ v_{k, t}^i(x)=u_k^i(x)+t,\, i=1, m, \, x\in \partial B^n_\eta(y).  \]
%Note that $v_{k, 0}^m-v_{k, 0}^1$ converges to 0 as $\ep_k\to 0$.

By the Hausdorff convergence of $\Si_k\to\Si_\infty$ and the Strong Maximum Principle \cite[Lemma 3.12]{Zhou-Zhu18} (applied to $\Gr(u_k^1)$ and $\{\Gr(v_{k, t}^1)\}$, $\Gr(u_k^m)$ and $\{\Gr(v_{k, t}^m)\}$), we have
%\[ u_k^m(x)\leq v_{k, 0}^m(x), \, \text{ and } u_k^1(x) \geq v_{k, 0}^1(x) \, \text{ when } x\in U_k\cap B^n_\eta(y),\, \text{ so }\]
\[ u_k^m(x)-u_k^1(x)\leq v_{k, 0}^m(x)-v_{k, 0}^1(x),\, \text{ when } x\in U_k\cap B^n_\eta(y). \]

By subtracting the mean curvature equations for $\Gr(v_{k, 0}^i)$, $i=1, m$, we get an equation similar to (\ref{E:height between two PMCs of the same orientation}),
\[ L_{\Si_\infty} (v_{k, 0}^m-v_{k, 0}^1) + o(v_{k, 0}^m-v_{k, 0}^1)= \ep_k\cdot \partial_\nu h(x,  v_k(x)) \cdot (v_{k, 0}^m(x)-v_{k, 0}^1(x)). \]
Note that the two graphs $\Gr(v_{k, 0}^i), i=1, m$ must be disjoint by the Strong Maximum Principle. By elliptic estimates via the weak maximum principle \cite[Theorem 3.7]{Gilbarg-Trudinger01},  
%Harnack estimates \cite[Section 8.6]{Gilbarg-Trudinger01}, 
we have for $\eta$ small enough and $k$ sufficiently large and a uniform $C>0$ so that,  
\[ \max_{B^n_\eta} (v_{k, 0}^m-v_{k, 0}^1) \leq C \max_{\partial B^n_\eta} (v_{k, 0}^m-v_{k, 0}^1). \]
This implies
\[ \max_{U_k \cap B^n_\eta} (u_k^m(x)-u_k^1(x)) \leq C \max_{\partial B^n_\eta} (u_k^m(x)-u_k^1(x)). \]
Hence $\max_{U_k \cap B^n_\eta} \tilde h_k \leq C \max_{\partial B^n_\eta} \tilde h_k $, so $\varphi$ is uniformly bounded and hence extends smoothly across $y$.

\vspace{1em}
\noindent{\bf Part 6}: We now take care the more interesting case: {\bf $m$ is an even number}. Hence $m\geq 2$.  In this case $\nu_k$ restricted to the bottom ($u_k^1$) and top ($u_k^m$) sheets point to different side of $\Si_\infty$, and without loss of generality we may assume that $\nu_k$ points downward on top sheet, and upward on bottom sheet. That means:
\[ H\vert_{\Gr(u_k^m)}(x)= -\ep_k h(x, u_k^m(x)), \text{ and } H\vert_{\Gr(u_k^1)}(x)= \ep_k h(x, u_k^1(x)),\, \text{ for } x\in U_k. \]

Note that 
\[ u_k^m(x)-u^1_k(x)\geq 0, \text{  for all } x\in U_k, \]
but it may take zeros in a co-dimension 1 subset by \cite[Proposition 3.17]{Zhou-Zhu18}.

Again by subtracting the above two equations, and using the fact $H\vert_{\Gr(u_k^m)}-H\vert_{\Gr(u_k^1)}= L_{\Si_\infty} (u_k^m-u_k^1) + o(u_k^m-u_k^1)$, %reference or an appendix
we have
\begin{equation}
\label{E:height between two PMCs of the different orientation}
L_{\Si_\infty} (u_k^m-u_k^1) + o(u_k^m-u_k^1)= -\ep_k\cdot (h(x, u_k^1(x) + h(x, u_k^m(x)).
\end{equation}

Fix a point $p\in U_1$, and we discuss the renormalization in two cases. Again write $h_k= u_k^m-u_k^1$. 

\vspace{0.5em}
{\bf Case 1:} $\limsup_{k\to\infty} \frac{h_k(p)}{\ep_k}=+\infty$. Consider renormalizations $\tilde h_k(x)= h_k(x)/h_k(p)$. Then by the same reasoning as Part 4, $\tilde h_k$ converges locally smoothly to a nontrivial function $\varphi \geq 0$ on $\Si_\infty\backslash \mathcal Y$, and such that
\[ L_{\Si_\infty}\varphi =0,\, \text{  outside } \mathcal Y. \]

\vspace{0.5em}
{\bf Case 2:} $\limsup_{k\to\infty} \frac{h_k(p)}{\ep_k}<+\infty$. Consider renormalizations $\tilde h_k(x)= h_k(x)/\ep_k$. Then again by the same reasoning, $\tilde h_k$ converges locally smoothly to a nonnegative $\varphi\geq 0$ on $\Si_\infty\backslash \mathcal Y$, and such that
\[ L_{\Si_\infty}\varphi = -2 h\vert_{\Si_\infty},\, \text{  outside } \mathcal Y. \]

\vspace{1em}
\noindent{\bf Part 7}: We will follow a slightly different local foliation argument to prove removable singularity for $\varphi$. We inherit all notations in Part 5. Without loss of generality, we may assume $\sup_M|h|=1$.  Let $v_{k, t}^1, v_{k, t}^m : B^n_\eta \to \R$, $t\in [-\eta ,\eta]$, be the CMC local foliations associated with $-\ep_k$ and $\ep_k$ respectively,
\[ H_{\Gr(v_{k, t}^m)}(x) = \ep_k, \text{ and } H_{\Gr(v_{k, t}^1)}(x) = -\ep_k, \, x\in B^n_\eta(y), \]
and 
\[ v_{k, t}^i(x)=u_k^i(x)+t,\, i=1, m, \, x\in \partial B^n_\eta(y).  \]
%Note that $v_{k, 0}^m-v_{k, 0}^1$ converges to 0 in smooth topology as $\ep_k\to 0$.

By the same reasoning as Part 5 using the Strong Maximum Principle for varifolds by White \cite{White10}, we get
\[ \max_{U_k\cap B^n_\eta} (u_k^m(x)-u_k^1(x)) \leq \max_{ B^n_\eta} (v_{k, 0}^m(x)-v_{k, 0}^1(x)). \]
Note that slightly different with Part 5, we have
\[ L_{\Si_\infty} (v_{k, 0}^m-v_{k, 0}^1) + o(v_{k, 0}^m-v_{k, 0}^1)= 2\ep_k. \]
By \cite[Theorem 3.7]{Gilbarg-Trudinger01}, 
%Harnack estimates \cite[Section 8.6]{Gilbarg-Trudinger01}, 
we have for $\eta$ small enough, $k$ large enough and for some uniform $C>0$
\[ \max_{U_k\cap B^n_\eta} (u_k^m(x)-u_k^1(x)) \leq C \big(\max_{\partial B^n_\eta} (u_k^m(x)-u_k^1(x)) + \ep_k \big). \]
Then for both Case 1 and Case 2, this implies that $\varphi$ is uniformly bounded and hence extends smoothly across $\mathcal Y$.

Note that if we flip the orientations of the top and bottom sheets, then in Case 2 the limit of renormalizations of heights will converge to a solution of $L_{\Si_\infty}\varphi = 2 h\vert_{\Si_\infty}$, where $\varphi \geq 0$.  Note that in the previous case, we can just flip the sign of $\varphi$, and obtain
\[ L_{\Si_\infty}\varphi = 2 h\vert_{\Si_\infty},\, \text{ where } \varphi\leq 0. \]

\vspace{1em}
\noindent{\bf Part 8}: Now we briefly record the case when $\Si_\infty$ is only one-sided. Then the convergence of $\Si_k$ must have multiplicity at least 2; otherwise the convergence will be smooth by the Allard regularity theorem \cite{Allard72}, and hence all $\Si_k$ will be 1-sided for $k$ sufficiently large, which is a contradiction. Denote $\pi: \tilde \Si_\infty \to \Si_\infty$ as the 2-sided double cover of $\Si_\infty$, and $\tau: \tilde \Si_\infty \to \tilde \Si_\infty$ the deck transformation map. By the same argument for the 2-sided case applied to the double cover $\tilde \Si_\infty$, %more details
we can either construct a non-trivial Jacobi field $\varphi$ on $\tilde \Si_\infty$ with $\varphi\circ \tau =\varphi$ and
\[ L_{\tilde \Si_\infty}\varphi=0; \]
or a smooth function $\varphi$ on $\tilde \Si_\infty$ with $\varphi\circ \tau =\varphi$, such that $\varphi$ does not change sign, and
\[ L_{\tilde \Si_\infty}\varphi = 2 h\vert_{\Si_\infty}\circ \pi. \]
By \cite{White17}, the first case cannot happen in a bumpy metric.

\vspace{1em}
Summarizing the discussion, we proved that if $g$ is bumpy, then each connected 2-sided component $\Si_o$ of $\Si_\infty$ with multiplicity bigger than one must carry a smooth solution $\varphi$ to the equation
\begin{equation}
\label{E:Jacobi type solution 2-sided}
L_{\Si_o}\varphi = 2 h\vert_{\Si_o};
\end{equation} 
and the double cover $\tilde \Si_u$ of each 1-sided component $\Si_u$ of $\Si_\infty$ must carry a smooth solution $\varphi$
\begin{equation}
\label{E:Jacobi type solution 1-sided}
L_{\tilde\Si_u}\varphi = 2 h\vert_{\Si_u}\circ \pi,\, \text{ and } \varphi\circ \tau=\varphi.
\end{equation} 
Moreover, in both cases $\varphi$ does not change sign.

\vspace{1em}
\noindent{\bf Part 9}: We will show that for a nicely chosen $h\in \mathcal S(g)$, the (unique) solutions to (\ref{E:Jacobi type solution 2-sided}) and (\ref{E:Jacobi type solution 1-sided}) must change sign. Thus there is no 1-sided component, and the multiplicity for 2-sided component must be one.

\begin{lemma}[Key Lemma]
\label{L:key lemma}
Assume that $g$ is bumpy. Given $\bL>0$ and $k\in \N$, there exists $h\in \mathcal S(g)$, such that if $\Si$ is a smooth, connected, closed, embedded minimal hypersurface with 
\[ \Area(\Si)\leq \bL, \,  \text{ and } \ind(\Si)\leq k,\] 
then the solution of (\ref{E:Jacobi type solution 2-sided}) (when $\Si$ is 2-sided) or (\ref{E:Jacobi type solution 1-sided}) (when $\Si$ is 1-sided) must change sign.
\end{lemma}
\begin{proof}
As $g$ is bumpy, by the compactness analysis of Sharp \cite{Sharp17}, there are only finitely many such $\Si$ with $\Area(\Si)\leq \bL$ and $\ind(\Si)\leq k$, and we can denote them as $\{\Si_1, \cdots, \Si_L\}$. If $\Si_i$ is 1-sided, we use $\pi_i: \tilde \Si_i\to \Si_i$ to denote the 2-sided double cover, and $\tau_i: \tilde\Si_i\to \tilde\Si_i$ to denote the deck transformation map.

On each $\Si_i$, we can choose two disjoint open subsets $U_i^+$ and $U_i^- \subset \Si_i$, so that the collection of subsets $\{ U_i^\pm\}_{i=1,\cdots, L}$ are pairwise disjoint. Moreover, by possibly changing $U_i^\pm$, we can make sure that the pre-image $\pi_i^{-1}(U_i^+)$, $\pi_i^{-1}(U_i^-)$ are diffeomorphic to two disjoint copies of $U_i^+$, $U_i^-$ respectively. In that case, we will denote the two copies as $\tilde U_{i, 1}^+, \tilde U_{i, 2}^+$, and $\tilde U_{i, 1}^-, \tilde U_{i, 2}^-$. That is
\[ \pi_i^{-1}(U_i^+)= \tilde U_{i, 1}^+ \cup \tilde U_{i, 2}^+,\, \text{ and }  \pi_i^{-1}(U_i^-)= \tilde U_{i, 1}^- \cup \tilde U_{i, 2}^-. \]

For each $i\in\{1, \cdots, L\}$ such that $\Si_i$ is 2-sided, we can choose an arbitrary pair of nontrivial smooth functions $f_i^+ \in C^\infty_c(U_i^+)$, $f^- \in C^\infty_c(U_i^-)$ such that 
\[ f_i^+\geq 0,\, \text{ and } f_i^+(p_i^+)>0\, \text{ at some } p_i^+\in U_i^+,\]
and
\[ f_i^-\leq 0,\, \text{ and } f_i^-(p_i^-)<0\, \text{ at some } p_i^-\in U_i^1.\]
Let $h_i^+\in C^\infty_c(U_i^+)$ and $h_i^-\in C^\infty_c(U_i^-)$ be defined by:
\[ h_i^+=L_{\Si_i} f_i^+,\quad h_i^-=L_{\Si_i} f_i^-. \]

If $\Si_i$ is 1-sided, we choose $\tilde f^\pm_{i, 1} \in C^\infty_c(\tilde U_{i, 1}^\pm), \tilde f_{i, 2}^\pm \in C^\infty_c(\tilde U_{i, 2}^\pm)$ in the same way, and we can make sure they are the same under deck transformation: $\tilde f^\pm_{i, 1}\circ \tau=\tilde f_{i, 2}^\pm$. In particular,
\[ \tilde f_{i, 1}^+\geq 0, \, \text{ and } \tilde f_{i, 1}^+>0 \text{ somewhere in } \tilde U_{i, 1}^+, \]
and
\[ \tilde f_{i, 1}^-\leq 0, \, \text{ and } \tilde f_{i, 1}^-<0 \text{ somewhere in } \tilde U_{i, 1}^-. \]
Then we define $h_{i, 1}^\pm, h_{i, 2}^\pm$ in the same manner, so obviously $h_{i, 1}^\pm\circ \tau = h_{i, 2}^\pm$, and they pass to two functions 
\[ h_i^+ \in C^\infty_c(U_i^+),\, \text{ and } h_i^-\in C^\infty_c(U_i^-).   \]

We can extend each $h_i^\pm$ to a function defined on $\Si_i$ by letting it be zero outside $U_i^\pm$. Using the fact that the set of smooth functions $\mathcal S(g)$ is open and dense in $C^\infty(M)$, we can choose a $h\in \mathcal S(g)$ so that
\[ h \text{ is as close to } h_i^\pm \text{ as we want in any $C^{k, \al}$-norm when restricted to }\Si_i. \]
We may need to flip the sign of $h$ to make $\int_M h \geq 0$, but the following argument proceeds the same way. 
Since all $\{\Si_i: i=1, \cdots, L\}$ are non-degenerate (the Jacobi operator is an isomorphism), we know that if 
\[ \text{$L_{\Si_i}\varphi =2 h\vert_{\Si_i}$ when $\Si_i$ is 2-sided, or} \]
\[ \text{$L_{\tilde \Si_i} \varphi = 2 h\vert_{\Si_i}\circ \pi_i$ when $\Si$ is 1-sided}, \]
then
\[ \text{$\varphi$ is as close to $f_i^\pm$ or $\tilde f_{i, j}^\pm$ ($j=1, 2$) as we want in $C^{k+2, \al}$-norm when restricted to $\Si_i$ or $\tilde \Si_i$}. \]
Then $\varphi$ must change sign, and this is what we want to prove.
\end{proof}

Note that by Part 2, all connected components of a min-max minimal hypersurface must satisfy the area and index bound in Lemma \ref{L:key lemma}. So we finish the proof of the theorem.
\end{proof}
\begin{remark}
Indeed, we can obtain more information. Since $\Si_\infty$ has multiplicity one, the Allard regularity theorem \cite{Allard72} implies that the convergence $\Si_k \to \Si_\infty$ is smooth everywhere, and hence $\Si_k$ is properly embedded for $k$ large.  
\end{remark}

\begin{remark}
Without assuming that $g$ is bumpy, our proof says that if the multiplicity of a 2-sided component is greater than 2, or if the multiplicity for a 1-sided component is greater than 1, then there exists a nontrivial, nonnegative Jacobi field. Let us point out the necessary details for 2-sided case, and the 1-sided case follows the same way.  Indeed, we only need to focus on the case when the multiplicity $m$ is even and $m\geq 4$; and moreover, we can focus on Case 2 in Part 6. Using notations in Part 6 and 7,  we consider the height difference between the two pairs $(u_k^1, u_k^{m-1})$ and $(u_k^2, u_k^m)$,
\[ h_k^a= u_k^{m-1}-u_k^1, \quad h_k^b=u_k^m-u_k^2. \]
Then both $h_k^a, h_k^b > 0$ and satisfy equations of type (\ref{E:height between two PMCs of the same orientation}) since the graphs of the two pairs have outer normals pointing to the same side. Consider the renormalizations: $\tilde h_k^a = h_k^a/ \ep_k$ and $\tilde h_k^b =h_k^b/\ep_k$. Then
\[ \tilde h_k^a, \tilde h_k^b \leq \tilde h_k,\, \text{ and }\, \tilde h_k^a+\tilde h_k^b \geq \tilde h_k. \]
Note that the limit of $\tilde h_k$ can not be identically zero, as then $h\vert_{\Si_\infty}\equiv 0$, violating the assumption $h\in \mathcal S(g)$. Then the above two inequalities and standard elliptic estimates imply that at least one limit of the two sequences $\{ \tilde h_k^a \}_{k\in \N}$ and $\{ \tilde h_k^b \}_{k\in \N}$ must be a smooth, nontrivial, nonnegative Jacobi field.
\end{remark}

%\vspace{1em}
Part of the proof of the theorem can be summarized as the following multiplicity one convergence result, which we believe has its independent interests.

\begin{theorem}[Multiplicity one convergence]
\label{T:Multiplicity one convergence}
Let $(M^{n+1}, g)$ be a closed manifold of dimension $3\leq (n+1)\leq 7$ with a bumpy metric $g$. Given $\bL>0$, $I\in\N$, then there exists a smooth function $h: M\to\R$, $h\in\mathcal S(g)$, such that: 

Let $\{\Si_k\}_{k\in\N}$ be a sequence of smooth, closed, almost embedded hypersurfaces, and $\{\ep_k\}_{k\in \N}\to 0$, such that
\begin{itemize}
\item $\Si_k$ is the boundary of some open set $\Om_k$, and the mean curvature of $\Si_k$ with respect to the outer normal of $\Om_k$ is prescribed by $\ep_k h$;
\item $\Area(\Si_k)\leq \bL$, and $\ind(\Si_k)\leq I$.
\end{itemize} 
Then up to a subsequence $\{\Si_k\}_{k\in\N}$ converges smoothly to a smooth, closed, embedded, two-sided, minimal hypersurface $\Si_\infty$ with multiplicity one.
\end{theorem}

%%%%%%%%%%%%%%%%%%%%%%%%%%%%%%%%%%
% Section 5     Application to volume spectrum 	     %
%%%%%%%%%%%%%%%%%%%%%%%%%%%%%%%%%%

\section{Application to volume spectrum}
\label{S:Application to volume spectrum}

In this part, we will show how to apply the result in Section \ref{S:first multiplicity one result} to study volume spectrum introduced by Gromov, Guth, and Marques-Neves. In particular, we will prove that in a bumpy metric, the volume spectrum can be realized by the area of min-max minimal hypersurfaces produced by Theorem \ref{T:multiplicity 1 for sweepouts of boundaries}. To do this, we will carefully pick a sequence of sweepouts of mod 2 cycles, and open the parameter space so as to produce sweepouts of boundaries of Caccioppoli sets, whose relative homotopy classes satisfy (\ref{E:nontrivial assumption3}). As the space of Caccioppli sets forms a double cover of the space of mod 2 cycles, the parameter-space-opening process is achieved by lifting to the double cover.

We first recall the definition of volume spectrum following \cite[Section 4]{Marques-Neves17}. Let $(M^{n+1}, g)$ be a closed Riemannian manifold. Let $X$ be a cubical subcomplex of $I^m=[0, 1]^m$ for some $m\in \N$. Given $k\in \N$, a continuous map $\Phi: X\to \Z_n(M, \mZ_2)$ is a {\em $k$-sweepout} if 
\[ \Phi^*(\bar \la^k)\neq 0\in H^k(X, \mZ_2), \]
where $\bar\la\in H^1(\Z_n(M, \mZ_2), \mZ_2)=\mZ_2$ is the generator.  $\Phi$ is said to be {\em admissible} if it has no concentration of mass. Denote by $\mathcal P_k$ as the set of all admissible $k$-sweepouts. Then

\begin{definition}
\label{D:k-width}
The {\em $k$-width} of $(M, g)$ is
\[ \om_k(M, g) = \inf _{\Phi\in \mathcal P_k} \sup\{ \M(\Phi(x)) : x\in \dmn (\Phi) \},  \]
where $\dmn(\Phi)$ is the domain of $\Phi$.
\end{definition}

It was proved in \cite[Theorem 5.1 and 8.1]{Marques-Neves17} that there exists some constant $C=C(M, g)$, such that
\[ \frac{1}{C} k^{\frac{1}{n+1}}\leq \om_k(M, g) \leq C k^{\frac{1}{n+1}}. \]

Assume from now on that the dimension satisfies $3\leq (n+1)\leq 7$. It was later observed by Marques-Neves in \cite{Marques-Neves16} that one can restrict to a subclass of $\mathcal P_k$ in the definition of $\om_k(M, g)$. In particular, let $\tilde{\mathcal P}_k$ denote those elements $\Phi\in \mathcal P_k$ which is continuous under the $\mF$-topology, and whose domain $X=\dmn(\Phi)$ has dimension $k$ (and is identical to its $k$-skeleton). Then
\[ \om_k(M, g) = \inf _{\Phi\in \tilde{\mathcal P}_k} \sup\{ \M(\Phi(x)) : x\in \dmn (\Phi) \}. \]
They also proved in \cite{Marques-Neves16} that for each $k\in\N$ there exists a disjoint collection of smooth, connected, closed, embedded minimal hypersurfaces $\{\Si^k_i: i=1, \cdots, l_k\}$ with integer multiplicities $\{m^k_i: i=1, \cdots, l_k\}\subset \N$, such that
\[ \om_k(M, g)= \sum_{i=1}^{l_k} m^k_i \cdot \Area(\Si^k_i), \quad \text{ and }\quad \sum_{i=1}^{l_k} \ind(\Si^k_i) \leq k. \]

Now we are going to state and prove our main theorem.
\begin{theorem}[Theorem A]
\label{T:theorem A}
If $g$ is a bumpy metric and $3\leq (n+1)\leq 7$, then for each $k\in\N$, there exists a disjoint collection of smooth, connected, closed, embedded, two-sided minimal hypersurfaces $\{\Si^k_i: i=1, \cdots, l_k\}$, such that
\[ \om_k(M, g)= \sum_{i=1}^{l_k} \Area(\Si^k_i), \quad \text{and }\quad \sum_{i=1}^{l_k} \ind(\Si^k_i) \leq k.\]
%$\{\Si^k_i: k\in \N, i=1, \cdots, l_k\}$ are all two-sided and have multiplicity one, i.e. $m^k_i\equiv 1$. 
That is to say, the min-max minimal hypersurfaces are all two-sided and have multiplicity one.
\end{theorem}

\begin{proof}

If $g$ is bumpy, then there are only finitely many closed, embedded, minimal hypersurfaces with $\Area \leq \La$ and $\ind \leq I$ for given $\La>0, I\in\N$ by Sharp's result \cite{Sharp17}. Using the Morse index upper bound estimates for min-max theory by Marques-Neves \cite{Marques-Neves16}, we have
\begin{lemma}
Suppose $g$ is bumpy, then for each $k\in\N$, there exists a $k$-dimensional cubical complex $X_k$ %(which is identical to its $k$-skeleton), 
and a map $\Phi_{0, k}: X_k\to \Z_n(M, \mF, \mZ_2)$ continuous in the $\mF$-topology with $\Phi_{0, k}\in \tilde{\mathcal P}_k$, such that
\[ \bL(\Pi_k)= \om_k(M, g), \]
where $\Pi_k=\Pi(\Phi_{0, k})$ is the class of all maps $\Phi: X_k\to \Z_n(M, \mF, \mZ_2)$ continuous in the $\mF$-topology that are homotopic to $\Phi_{0, k}$ in flat topology. 
\end{lemma}
\begin{proof}
From definition we know that 
\[ \om_k(M, g)=\inf\{ \bL(\Pi(\Phi)), \,  \Phi \in\tilde{\mathcal P}_k \}. \]
By area and index upper bounds and the finiteness result, the infimum is achieved.
\end{proof}

Now we fix $k\in \N$ and omit the sub-index $k$ in the following. Take $\Pi=[\Phi_0: X\to \Z_n(M, \mF, \mZ_2)]$ with $\bL(\Pi)=\om_k$. The following result is an outcome of the proof of \cite[Theorem 6.1]{Marques-Neves16}.
\begin{lemma}
\label{L:good min-max sequence}
Suppose $g$ is bumpy. Then there exists a pull-tight (see \cite[3.7]{Marques-Neves16}) min-max sequence $\{\Phi_i\}_{i\in\N}$ of $\Pi$ such that if $\Si\in \bC(\{\Phi_i\}_{i\in\N})$ has support a smooth, closed, embedded minimal hypersurface, then
\[ \|\Si\|(M)=\om_k(M, g),\, \text{ and } \ind(\text{support of } \Si)\leq k. \] 
\end{lemma}

We proceed the proof by the following four steps.

\vspace{1em}
\noindent{\bf Step 1}: \textit{In this and the next step, we show how to find another min-max sequence, still denoted as $\{\Phi_i\}_{i\in\N}$, such that for $i$ sufficiently large, either $|\Phi_i (x)|$ is close to a regular min-max minimal hypersurface, or the mass $\M(\Phi_i(x))$ is strictly less than $\om_k(M, g)$.}
\vspace{0.5em}

We recall the following observation by \cite[Claim 6.2]{Marques-Neves17}. Let $\mathcal S$ be the set of all stationary integral varifolds with $\Area\leq \om_k$ whose support is a smooth closed embedded minimal hypersurface with $\ind(support)\leq k$.  Consider the set $\mathcal T$ of all mod 2 flat cycles $T\in \Z_n(M, \mZ_2)$ with $\M(T)\leq \om_k$ and such that either $T=0$ or the support of $T$ is a smooth closed embedded minimal hypersurface with $\ind\leq k$.  By the bumpy assumption, both sets $\mathcal S$ and $\mathcal T$ are finite. Moreover,
\begin{lemma}[Claim 6.2 in \cite{Marques-Neves17}]
\label{L:varifold close implies flat close}
For every $\bar\ep>0$, there exists $\ep>0$ such that
\[ T\in\Z_n(M, \mZ_2) \text{ with } \mF(|T|, \mathcal S)\leq 2\ep \Longrightarrow \F(T, \mathcal T)<\bar\ep. \]
\end{lemma}

We also need another observation by \cite[Corollary 3.6]{Marques-Neves17}. Denote $S^1$ by the unit circle.
\begin{lemma}[Corollary 3.6 in \cite{Marques-Neves17}]
\label{L:threshold for flat topology}
If $\bar\ep$ is sufficiently small, depending on $\mathcal T$, then every map $\Phi: S^1\to \Z_n(M, \mZ_2)$ with  
\[ \Phi(S^1)\subset B^\F_{\bar\ep}(\mathcal T)=\{ T\in \Z_n(M, \mZ_2): \F(T, \mathcal T)<\bar\ep \} \]
is homotopically trivial.
\end{lemma}

let $\{\Phi_i\}_{i\in\N}$ be chosen as in Lemma \ref{L:good min-max sequence}. We choose $\bar\ep$ as Lemma \ref{L:threshold for flat topology} and $\ep$ by Lemma \ref{L:varifold close implies flat close}. Take a sequence $\{k_i\}_{i\in\N}\to \infty$, such that %maybe say something about the modulus of continuity of $\Phi_i$ in each cell $X(k_i)$. 
\[ \sup\{\mF(\Phi_i(x), \Phi_i(y)): \al\in X(k_i), x, y\in\al\}\leq \ep/2. \]
Consider $Z_i$ to be the cubical subcomplex of $X(k_i)$ consisting of all cells $\al\in X(k_i)$ so that
\[ \mF(|\Phi_i(x)|, \mathcal S)\geq \ep,\, \text{ for every vertex $x$ in $\al$.} \]
Hence $\mF(|\Phi_i(x)|, \mathcal S)\geq \ep/2$ for all $x\in Z_i$. 

Consider this sub-coordinating sequence $\{ \Phi_i\vert_{Z_i} \}_{i\in\N}$. $\bL(\{\Phi_i\vert_{Z_i}\})$ and $\bC(\{\Phi_i\vert_{Z_i}\})$ are defined in the same way as in Section \ref{SS:min-max construction in continuous setting} with $\Ah$ replaced by $\M$.
\begin{lemma}
\label{L:a dichotomy}
We have the following dichotomy:
\begin{itemize}
\item no element $V\in \bC(\{ \Phi_i\vert_{Z_i} \}_{i\in\N})$ is $\mZ_2$-almost minimizing in small annuli (see \cite[2.10]{Marques-Neves17}),
\item or
\begin{equation}
\label{E:good sub-width upper bound1}
\bL(\{ \Phi_i\vert_{Z_i} \}_{i\in\N})<\bL(\Pi)=\om_k. 
\end{equation}
\end{itemize}
\end{lemma}
\begin{proof}
Suppose that (\ref{E:good sub-width upper bound1}) does not hold, then $\bL(\{ \Phi_i\vert_{Z_i} \}_{i\in\N})=\bL(\Pi)$. As $\{\Phi_i\}_{i\in\N}$ is pull-tight, we know that every $V\in \bC(\{ \Phi_i\vert_{Z_i} \}_{i\in\N})$ is stationary. If $V$ is also $\mZ_2$-almost minimizing in small annuli, then $V$ is regular by the regularity of Pitts \cite[Theorem 7.11]{Pitts81} and Schoen-Simon \cite[Theorem 4]{Schoen-Simon81}; (see also \cite[Theorem 2.11]{Marques-Neves17} for the adaption to $\mZ_2$-coefficients). By Lemma \ref{L:good min-max sequence}, $V\in \mathcal S$, which is a contradiction.
\end{proof}

Let $Y_i=\overline{X\backslash Z_i}$. It then follows that 
\begin{equation}
\label{E:Y_i are trivial}
\mF(|\Phi_i(x)|, \mathcal S)\leq \frac{3}{2}\ep,\, \text{ for all } x\in Y_i.
\end{equation}
We also denote $B_i=Y_i\cap Z_i$. In fact, $B_i$ is the topological boundary of $Y_i$ and $Z_i$. For later purpose, we also consider the set
\[ \bB_i= \text{ the union of all cells $\al\in Z_i$ such that } \al\cap B_i\neq \emptyset. \]
$\bB_i$ can be thought of the ``thickening" of $B_i$ inside $Z_i$. %and $\bB_i$ is homeomorphic to $B_i\times [0, 1]$.

Let $\la=\Phi_i^* (\bar\la)\in H^1(X, \mZ_2)$. Consider the inclusion maps $i_1: Y_i\to X$ and $i_2: Z_i\to X$. It then follows from (\ref{E:Y_i are trivial}), Lemma \ref{L:varifold close implies flat close} and Lemma \ref{L:threshold for flat topology} that
\[ i_1^*(\la)=0 \in H^1(Y_i, \mZ_2). \]
Then by \cite[Claim 6.3]{Marques-Neves17}, $(\Phi_i)\vert_{Z_i}$ is a $(k-1)$-sweepout, i.e. 
\[ i_2^*(\la^{k-1})\neq 0 \in H^{k-1}(Z_i, \mZ_2). \]
%Note that $\bB_i$ can be contracted to $B_i$. Therefore if we 
Now we let $Y'_i=Y_i\cup \bB_i$ and $Z_i'=\overline{Z_i\backslash \bB_i}$, and $i'_i: Y'_i \to X$ and $i'_2: Z'_i \to X$ be the inclusion maps. Note that (\ref{E:Y_i are trivial}) is satisfied with $Y_i, \frac{3}{2}\ep$ replaced by $Y_i', 2\ep$ respectively, so by similar reasoning we have
\[ (i'_1)^*(\la)=0 \in H^1(Y'_i, \mZ_2),\, \text{ and } (i'_2)^*(\la^{k-1})\neq 0 \in H^{k-1}(Z_i', \mZ_2). \]

\vspace{1em}
\noindent{\bf Step 2}: \textit{The strategy is to follow the idea in the proof of Theorem \ref{T:main min-max theorem} and apply \cite[Theorem 2.13]{Marques-Neves17} (see also Theorem \ref{T:combinatorial deformation}) to deform $\{\Phi_i\}_{i\in\N}$ so as to decrease $\bL(\{ (\Phi_i)\vert_{Z_i} \}_{i\in\N})$ and make (\ref{E:good sub-width upper bound1}) be satisfied.}
\vspace{0.5em}

If (\ref{E:good sub-width upper bound1}) holds true, then we are done for this step. So let us assume that
\begin{equation}
\label{E:bad sub-width upper bound}
\bL(\{ \Phi_i\vert_{Z_i} \}_{i\in\N}) = \bL(\Pi)=\om_k. 
\end{equation}

By Lemma \ref{L:a dichotomy} and our assumption (\ref{E:bad sub-width upper bound}), we know that no element $V\in \bC(\{ \Phi_i\vert_{Z_i}\}_{i\in\N})$ is $\mZ_2$-almost minimizing in small annuli.

\vspace{1em}
Since $\Phi_i: X\to \Z_n(M, \mF, \mZ_2)$ has no concentration of mass as it is continuous in $\mF$-topology, we can apply \cite[Theorem 3.9]{Marques-Neves17} (the counterpart of Theorem \ref{T:continuous to discrete} for maps to $\Z_n(M, \mZ_2)$) to produce a sequence of maps
\[ \phi_i^j: X(k_i+k_i^j)_0 \to \Z_n(M, \mZ_2), \]
with $k_i^j\in\N$ and %$k_i\leq k_i^j$ and 
$k_i^j<k_i^{j+1}$ for all $j\in \N$ and a sequence of positive $\{\de_i^j\}_{j\in\N} \to 0$, such that
\begin{enumerate}[label=(\roman*)]
\item the fineness $\mf(\phi_i^j)\leq \de_i^j$;
\item \[ \sup\{ \F(\phi_i^j(x)-\Phi_i(x)): x\in X(k_i+k_i^j)_0 \}\leq \de_i^j; \]
\item for some sequence $l_i^j\to\infty$ with $l_i^j < k_i^j$
\[ \M(\phi_i^j(x)) \leq \sup\{ \M(\Phi_i(y)): x, y\in \al, \text{ for some } \al\in X(k_i+l_i^j) \}+\de_i^j. \]
\end{enumerate}

As $\Phi_i$ is continuous in $\mF$-topology, we get from property (\rom{3}) that for all $x\in X(k_i+k_i^j)_0$,
\[ \M(\phi_i^j(x))\leq \M(\Phi_i(x))+\eta_i^j \]
with $\eta_i^j\to 0$ as $j\to\infty$. Applying \cite[Lemma 4.1]{Marques-Neves14} with $\mathcal S=\Phi_i(X)$, we get by (\rom{2}) that
\begin{itemize}
\item[(\rom{4})] \[ \sup\{ \mF(\phi_i^j(x), \Phi_i(x)): x\in X(k_i+k_i^j)_0 \} \to 0,\, \text{ as } j\to\infty. \]
\end{itemize}

We can choose $j(i)\to\infty$ as $i\to\infty$ (then $k_i^{j(i)}\to\infty$) such that $\varphi_i=\phi_i^{j(i)}: X(k_i+k_i^{j(i)})_0\to \Z_n(M, \mZ_2)$ satisfies:
\begin{itemize}
\item $\sup\{ \mF(\varphi_i(x), \Phi_i(x)): x\in X(k_i+k_i^{j(i)})_0 \}\leq a_i$ with $a_i\to 0$ as $i\to\infty$;
\item $\sup\{ \mF(\Phi_i(x), \Phi_i(y)): x, y\in\al, \al\in X(k_i+k_i^{j(i)}) \}\leq a_i$;
\item the fineness $\mf(\varphi_i) \to 0$ as $i\to\infty$;
\item the Almgren extension $\Phi_i^{j(i)}: X\to \Z_n(M, \M, \mZ_2)$ (see \cite[3.10]{Marques-Neves17} for definition, and it is continuous in the $\M$-topology) is homotopic to $\Phi_i$ in the flat topology (by \cite[Corollary 3.12]{Marques-Neves17}), and $\sup\{ \mF(\Phi_i^{j(i)}(x), \Phi_i(x)): x\in X \} \to 0$ as $i\to\infty$ (by \cite[3.10]{Marques-Neves17}).
\end{itemize} 
If we let $S=\{\varphi_i\}_{i\in\N}$ be a discrete sweepout, then we have $\bL(S)=\bL(\{\Phi_i\}_{i\in\N})$ and $\bC(S)=\bC(\{\Phi_i\}_{i\in\N})$.
Moreover, consider the restrictions of $\varphi_i$ to $Z_i(k_i^{j(i)})_0$:
\[ S_Z=\{ \varphi_i: Z_i(k_i^{j(i)})_0 \to \Z_n(M, \mZ_2) \}. \]
Similarly we have 
\[ \bL(S_Z)= \bL(\{ \Phi_i\vert_{Z_i} \}_{i\in\N})=\bL(\Pi),\, \text{ and } \bC(S_Z)= \bC(\{ \Phi_i\vert_{Z_i} \}_{i\in\N}). \]

\vspace{1em}
As no $V\in \bC(S_Z)$ is $\mZ_2$-almost minimizing in small annuli, by \cite[Theorem 2.13]{Marques-Neves17} (which is a reformulation of Almgren-Pitts combinatorial argument \cite[Theorem 4.10]{Pitts81}), we can find a sequence $\tilde S_Z=\{\tilde \varphi_i\}$ of maps:
\[ \tilde\varphi_i: Z_i(k_i^{j(i)}+l_i)_0 \to \Z_n(M, \mZ_2), \]
and a sequence of homotopies 
\[\psi_i: I(l_i)_0\times Z_i(k_i^{j(i)}+l_i)_0 \to \Z_n(M, \mZ_2), \]
such that
\begin{itemize}
\item $\psi_i([0], x)= \varphi_i \circ \n(k_i^{j(i)}+l_i, k_i^{j(i)})(x)$ and $\psi_i([1], x) = \tilde\varphi_i (x)$;
\item the fineness of $\psi_i$ tends to zero as $i\to \infty$;
\item \[ \limsup_{i\to\infty} \sup\{ \M(\psi_i(t, x)): (t, x)\in I(l_i)_0\times Z_i(k_i^{j(i)}+l_i)_0\}=\bL(S_Z); \]
(note that this property was not explicitly listed in \cite[Theorem 2.13]{Marques-Neves17}, but it follows from the construction in \cite[Theorem 4.10]{Pitts81}).)
\item $\bL(\tilde S_Z)<\bL(S_Z)$.
\end{itemize}

Now we construct a new sequence $S^*=\{\varphi^*_i\}_{i\in\N}$ with
\[ \varphi^*_i: X(k_i+k_i^{j(i)}+l_i)_0 \to \Z_n(M, \mZ_2), \]
defined as
\begin{itemize}
\item $\varphi^*_i(x)= \varphi_i\circ \n(k_i^{j(i)}+l_i, k_i^{j(i)})(x)$,  when $x\in Y_i(k_i^{j(i)}+l_i)_0$;
\item $\varphi^*_i(x)=\psi_i(t(x), x)$, where $x\in\bB_i(l_i)_0$ and $t(x)=\min\{ 3^{-l_i} \cdot \bd(x, \bB_i\cap Y_i), 1\}\in I(l_i)_0$; (here %$\bB_i$ is chosen in $X(k_i+k_i^{j(x)})$, and 
$\bd$ is the distance function restricted to $\bB_i(l_i)_0$; see Appendix \ref{A:cubical complex structures});
%$\varphi^*_i(x)= \psi_i(t, (b, t))$, when $x=(b, t)\in \bB_i(l_i)_0=B_i(l_i)_0\times I(l_i)_0$; 
%when $x=(b, t)\in \bB_i(k_i^{j(i)}-k_i+l_i)_0=B_i(k_i^{j(i)}-k_i+l_i)_0\times I(l_i)_0$;
%double check the subdivision, maybe change the use of $k_i^{j(i)}$
\item $\varphi^*_i(x)= \tilde\varphi_i(x)$, when $x\in Z'_i(k_i^{j(i)}+l_i)_0$; (note that $t(x)\geq 1$ when $x\in Z'_i\cap \bB_i$).
\end{itemize}

By the construction, we see that 
\begin{itemize}
\item $\varphi^*_i$ is homotopic to $\varphi_i$ with fineness tending to zero as $i\to\infty$;
\item $\bL(S^*) = \bL(\Pi)$;
\item  $\limsup_{i\to\infty} \sup \{ \M(\varphi^*_i(x)): x\in Z'_i(k_i^{j(i)}+l_i)_0\} \leq \bL(\tilde S_Z)< \bL(\Pi)$.
\end{itemize}

\vspace{1em}
Consider the Almgren's extension of $\varphi^*_i$:
\[ \Phi^*_i: X\to \Z_n(M, \M, \mZ_2) .\]
Then 
\begin{enumerate}[label=(\alph*)]
\item $\Phi^*_i$ is homotopic to $\Phi_i^{j(i)}$ and hence to $\Phi_i$ in the flat topology by \cite[3.11]{Marques-Neves17}; and by \cite[3.10]{Marques-Neves17}
\item $\sup\{\mF(\Phi^*_i(x), \Phi_i(x)): x\in Y_i\}\to 0$;
\item $\bL(\{\Phi^*_i\})=\bL(S^*)=\bL(\Pi)$;
\item 
\[ \limsup_{i\to\infty} \sup\{ \M(\Phi^*_i (x)): x\in Z_i' \} \leq \bL(\tilde S_Z)<\bL(\Pi). \]
\end{enumerate}

By summarizing what we have done (and abusing the notation $Y_i=Y'_i$ and $Z_i=Z'_i$), we produced another min-max sequence $\{ \Phi^*_i \}_{i\in\N}\subset \Pi$ such that
\begin{enumerate}%[label=(\alph*)]
\item $X$ can be decomposed to $Y_i$ and $Z_i$ with $Z_i=\overline{X\backslash Y_i}$, and for $i$ large enough, 
\[ i_1^*(\la)=0 \in H^1(Y_i, \mZ_2),\, \text{ and } i_2^*(\la^{k-1})\neq 0 \in H^{k-1}(Z_i, \mZ_2). \]
\item $\bL(\{\Phi^*_i\})=\bL(\{\Phi_i\})=\bL(\Pi)$;
\item 
\[ \limsup_{i\to\infty} \sup\{ \M(\Phi^*_i (x)): x\in Z_i \}<\bL(\Pi). \]
\end{enumerate}
Note that both $Y_i$ and $Z_i$ are nonempty for $i$ large enough by (1)(3).

\vspace{1em}
\noindent{\bf Step 3}: \textit{Now we want to produce sweepouts in $\C(M)$ by lifting to the double cover $\partial: \C(M)\to \Z_n(M, \mZ_2)$ so as to produce sweepouts satisfying the assumption of Theorem \ref{T:multiplicity 1 for sweepouts of boundaries}.}
\vspace{0.5em}

We abuse notation and still write $\{\Phi_i^*\}$ as $\{\Phi_i\}$. Since $(\Phi_i)^*(\bar\la)\neq 0 \in H^1(X, \mZ_2)=\mZ_2$, there exist a double cover $\pi: \tilde X \to X$ with deck transformation map $\tau: \tilde X\to\tilde X$, and the lifting maps:
\[ \tilde \Phi_i: \tilde X \to (\C(M), \mF), \]
satisfying $\partial \tilde\Phi_i=\Phi_i\circ \pi$. 
Indeed, the cohomological condition implies that the induced maps $(\Phi_i)_*: \pi_1(X)\to \pi_1(\Z_n(M, \mZ_2))=\mZ_2$ are surjective; see \cite[Definition 4.1 (i)]{Marques-Neves17}. So the kernel of $(\Phi_i)_*$ is a subgroup of $\pi_1(X)$ with index 2.  Then the existence of such liftings follows from \cite[Proposition 1.36 and Proposition 1.33]{Hatcher02}. %check algebraic topology book! Yes, done.

Note that $i_1^*\la =0 \in H^1(Y_i, \mZ_2)$, 
so the pre-image of $Y_i$ is disconnected, and is a disjoint union of two copies of $Y_i$:
\[ \tilde Y_i=\pi^{-1}(Y_i) = Y_i^+\cup Y_i^-, \]
where both $Y_i^+$ and $Y_i^-$ are homeomorphic to $Y_i$. In fact, the cohomological condition implies that every closed curve $\ga: S^1\to Y_i$ lies in the kernel of $(\Phi_i)_*$, so the lift $\tilde \ga$ of $\ga$ to $\tilde X$ is still a closed curve. This means that $\tilde Y_i$ is disconnected as we want. 

Denote $\tilde Z_i$, $\tilde B_i$ and $\tilde\bB_i$ the pre-images of $Z_i, B_i, \bB_i$ under $\pi$ respectively. Then $\tilde \bB_i =\bB_i^+ \cup \bB_i^-$ is also a disjoint union of two copies of $\bB_i$. %check algebraic topology book!

\begin{lemma}
\label{L:lifting maps are good}
For $i$ large enough, if $\tilde\Pi_i$ is the $(\tilde X, \tilde Z_i)$-homotopy class associated with $\tilde \Phi_i$, then we have
\[ \bL(\tilde\Pi_i) \geq \bL(\Pi)> \max_{x\in \tilde Z_i} \M(\partial \tilde\Phi_i(x)). \]
\end{lemma}
\begin{proof}
Fix $i$ large, so that
\[ \sup_{x\in Z_i} \M(\Phi_i(x))<\bL(\Pi), \]
and we will omit the sub-index in the following proof. 

If the conclusion were not true, then we can find a sequence of maps $\{ \tilde\Psi_j: \tilde X \to (\C(M), \mF) \}_{j\in\N}\subset \tilde\Pi$,  such that
\[ \limsup_{j\to\infty} \sup\{ \M(\partial \tilde\Psi_j(x)): x\in X \}<\bL(\Pi),  \]
and homotopy maps $\{ H_j: [0, 1]\times \tilde X \to \C(M) \}$ which are continuous in the flat topology, $H_j(0, \cdot)=\tilde\Psi_j$, $H_j(1, \cdot)=\tilde\Phi$, and
\[ \limsup_{j\to\infty} \sup\{ \mF(H_j(t, x), \tilde\Phi(x)): t\in [0, 1], x\in \tilde Z \}=0. \]

We construct a new sequence of maps $\{\tilde\Psi^*_j\}_{j\in \N}$ defined as
\begin{itemize}
\item $\tilde\Psi^*_j(x)= \tilde\Psi_j(x)$, if $x\in Y^+$, and $\tilde\Psi^*_j(x)=\tilde\Psi_j \circ \tau(x)$, if $x\in Y^-$;
\item $\tilde\Psi^*_j(x)= H_j(t(x), x)$, where $t(x)=\min\{ \dist(x, \bB^+\cap Y^+), 1\}$ if $x\in \bB^+$, and $\tilde\Psi^*_j(x)= H_j \circ \tau(x)$, if $x\in \bB^-$; (here $\dist$ is the distance function by viewing $\bB$ as a cube complex in some $I(m, l)$);
%$\tilde\Psi^*_j(x)= H_j(t, (b, t))$, if $x=(b, t)\in \bB^+$, and $\tilde\Psi^*_j(x)= H_j \circ \tau(x)$, if $x\in \bB^-$;
\item $\tilde\Psi^*_j(x)= \tilde\Phi(x)$, if $x\in \tilde Z'$; (note that $t(x) \geq 1$ for $x\in \tilde Z'\cap (\bB^+\cap \bB^-)$).
\end{itemize}
%Note that we can make the new maps continuous essentially using the structure of $\tilde Y$ and $\tilde \bB$. 
Note that though $\tilde\Psi^*_j$ themselves may not be continuous as maps to $\C(M)$, $\Psi_j^*$ can be passed to quotient as continuous maps from $X$ to $\Z_n(M, \mZ_2)$. This is essentially where we used the structures of $\tilde Y$ and $\tilde \bB$, that is, $(Y^+, Y^-)$ and $(\bB^+, \bB^-)$ are pairwise disjoint.

Denote the quotient maps of $\{\tilde\Psi^*_j\}_{j\in\N}$ by $\{ \Psi^*_j=\partial\circ \tilde\Psi_j^*: X\to \Z_n(M, \mZ_2) \}_{j\in \N}$. We have
\begin{itemize}
\item $\Psi^*_j$ is homotopic to $\Phi$ in the flat topology;
\item $\limsup_{j\to\infty} \sup\{ \M(\Psi_j^*(x)): x\in X \}<\bL(\Pi)=\om_k(M, g)$ (by the three above inequalities).
\end{itemize}
This will lead to a contradiction with the definition of $k$-width once we prove that $\Psi^*_j$ is an admissible $k$-sweepout when $j$ is sufficiently large. Indeed, the only thing left is to show that $\Psi^*_j$ has no concentration of mass. This follows from the third inequality above. 
%We only need to check this for values of $\Psi^*_j$ assumed on $\bB$. In fact, using (\ref{E:flat norm control of flat homotopy})(\ref{E:mass control of flat homotopy}) and \cite[Lemma 4.1]{Marques-Neves14}, we know that the varifolds assumed by $|\Psi^*_j|$ on $\bB$ converge to $|\Phi|(\bB)$ in the $\mF$-topology as $j\to\infty$, so the no mass concentration follows for $j$ large enough. 
So we finish the proof.
\end{proof}

\vspace{1em}
\noindent{\bf Step 4}: \textit{We are ready to finish the proof of Theorem \ref{T:theorem A}.}
\vspace{0.5em}

For $i$ large enough as in Lemma \ref{L:lifting maps are good}, Theorem \ref{T:multiplicity 1 for sweepouts of boundaries} applied to $\tilde\Pi_i$ gives a disjoint collection of smooth, connected, closed, embedded, 2-sided, minimal hypersurfaces $\Si_i=\cup_{j=1}^{N_i} \Si_{i, j}$, such that
\[ \bL(\tilde\Pi_i) = \sum_{j=1}^{N_i} \Area(\Si_{i, j}),\, \text{ and } \ind(\Si_i)\leq k. \]
Note also that $\bL(\tilde\Pi_i) \leq \bL(\Phi_i)\to \bL(\Pi)=\om_k$. Counting the fact that there are only finitely many smooth, closed, embedded minimal hypersurfaces with $\Area\leq \om_k+1$ and $\ind\leq k$, for $i$ sufficiently large we have
\[ \bL(\tilde\Pi_i)=\bL(\tilde\Pi_{i+1})=\cdots = \om_k. \]
Hence we finish the proof of Theorem \ref{T:theorem A}.
\end{proof}

\begin{remark}
By the course of the above proof, in a bumpy metric, the min-max minimal hypersurfaces associated with any homotopically nontrivial sweepouts of mod-2 cycles are always two-sided and have multiplicity one.  In fact, if $\Phi: X\to \Z_n(M, \mZ_2)$ is homotopically nontrivial, then the induced map $\Phi_*: \pi_1(X)\to \pi_1(\Z_n(M, \mZ_2))=\mZ_2$ must be surjective. Otherwise by \cite[Proposition 1.33]{Hatcher02} $\Phi$ can be lifted to a map $\tilde\Phi: X\to \C(M)$ which is then homotopically trivial as $\C(M)$ is contractible.  With these topological information, the above proof works the same way and implies the two-sidedness and multiplicity one for min-max minimal hypersurfaces associated with $\Pi(\Phi)$.
\end{remark}

%%%%%%%%%%%%%%%%%%%%%%%%%%%%%%%%%%
% Appendix        		     %
%%%%%%%%%%%%%%%%%%%%%%%%%%%%%%%%%%
\appendix

\section{Cubical complex structures}
\label{A:cubical complex structures}

Here we recall several cubical complex structures in \cite[2.1]{Marques-Neves17}.

For each $k\in\N$, $I(1, k)$ denotes the cubical complex on the unit interval $I=[0, 1]$ whose 1-cells and 0-cells (which are also called vertices) are, respectively,
\[ [0, 3^{-k}], [3^{-k}, 2\cdot 3^{-k}], \cdots, [1-3^{-k}, 1]\, \text{ and }\, [0], [3^{-k}], \cdots, [1-3^{-k}], [1]. \]
We then denote by $I(m, k)$ the cell complex on $I^m$:
\[ I(m, k)=I(1, k)\otimes\cdots\otimes I(1, k)\quad \text{$m$ times}. \]
Then $\al=\al_1\otimes\cdots\otimes\al_m$ is a $q$-cell of $I(m, k)$ if and only if $\al_i$ is a cell of $I(1, k)$ for each $i$, and $\sum_{i=1}^m\dim(\al_i)=q$. We often identify a $q$-cell $\al$ with its support $\al_1\times\cdots\times \al_m\subset I^m$. The distance function $\bd$ on $I(m, k)_0$ is defined as $\bd(x, y)=3^k \sum_{i=1}^k|x_i-y_i|$, $x, y \in I(m, k)_0$, \cite[4.1(1)(e)]{Pitts81}.

Let $X\subset I^m$ be a cubical subcomplex. The cubical complex $X(k)$ is the union of all cells of $I(m, k)$ whose support is contained in some cell of $X$.  We use the notation $X(k)_q$ to denote the set of all $q$-cells in $X(k)$, and particularly $X(k)_0$ to denote the set of vertices in $X(k)$. Two vertices $x, y\in X(k)_0$ are adjacent if they belong to a common cell in $X(k)_1$. 

Let $Y\subset I(m, k)$ be a cubical subcomplex. Similarly, the cubical complex $Y(l)$ is the union of all cells of $I(m, k+l)$ whose support is contained in some cell of $Y$. $Y(k)_q$ is defined in the same way.

Given $k, l\in\N$, we define $\n(k, l): X(k)_0\to X(l)_0$ so that $\n(i, j)(x)$ is the element in $X(l)_0$ that is closest to $x$; (see \cite[page 141]{Pitts81}).

\section{Removing singularity for weakly stable PMC}
\label{A:removable singularity}

We record the following standard removable singularity result. 

\begin{theorem}
\label{T:removable singularity}
Let $(M^{n+1}, g)$ be a closed Riemannian manifold of dimension $3\leq (n+1)\leq 7$. Given $h\in C^\infty(M)$ and $\Si\subset B_\ep(p)\backslash\{p\}$ an almost embedded hypersurface with $\partial\Si\cap B_\ep(p)\backslash\{p\}=\emptyset$, assume that $\Si$ has prescribing mean curvature $h$, and $\Si$ is weakly stable in $B_\ep(p)\backslash\{p\}$ as in Theorem \ref{T:compactness with bounded index}, Part 1 of proof. %(as an Alexandrov immersion) 
If $\Si$ represents a varifold of bounded first variation in $B_\ep(p)$, then $\Si$ extends smoothly across $p$ as an almost embedded hypersurface in $B_\ep (p)$.
\end{theorem}
\begin{proof}
Given any sequence of positive $\la_i\to 0$, consider the blowups $\{\bmu_{p, \la_i}(\Si)\subset \bmu_{p, \la_i}(M)\}$, where $\bmu_{p, \la_i}(x)=\frac{x-p}{\la_i}$. Since $\Si$ has bounded first variation, $\bmu_{p, \la_i}(\Si)$ converges (up to a subsequence) to a stationary integral rectifiable cone $C$ in $\R^{n+1}=T_p M$.  By weakly stability and Theorem \ref{T:Curvature estimates for weakly stable PMC} (which works well for the notion of weak stability of $\Si_\infty$), %which works well for Alexandrov immersion
the convergence is locally smooth and graphical away from the origin, so $C$ is an integer multiple of some embedded minimal hypercone; moreover, $C$ is weakly stable, and hence is stable as an embedded minimal hypersurface away from $0$. 
%strongly stable by Proposition \ref{P:equivalence of k-unstable}. 
Therefore $C$ is an integer multiple of some $n$-plane $P$ by Simons's classification \cite{Simons}, i.e. $C=m \cdot P$ where $m=\Theta^n(\Si, p)$. Note that a priori $C$ may not be unique.

By the locally smooth and graphical convergence, there exists $\si_0>0$ small enough, such that for any $0<\si\leq \si_0$, $\Si$ has an $m$-sheeted, ordered (in the sense of \cite[Definition 3.2]{Zhou-Zhu18}), graphical decomposition in the annulus $A_{\si/2, \si}(p)=B_\si(p)\backslash\overline{B}_{\si/2}(p)$:
\[ \Si\cap A_{\si/2, \si}(p)=\cup_{i=1}^m \Si_i(\si). \]
Here each $\Si_i(\si)$ is a graph over $A_{\si/2, \si}(p)\cap P$ for some $n$-plane $P\subset T_p M$. 

We can continue each $\Si_i(\si)$ all the way to $B_{\si_0}(p)\backslash \{p\}$, and we denote the continuation by $\Si_i$. Each $\Si_i$ can be extended as a varifold across $p$ with uniformly bounded first variation (since $\Si_i\subset \Si$ satisfies the area decay estimates, $\area(\Si_i\cap B_\si(p))\leq C \si^n$). We claim that the density satisfies $\Theta^n(\Si_i, p)=1$ for each $i$. In fact, $\Theta^n(\Si_i, p)\geq 1$ as any blowups of $\Si_i$ converges to an $n$-plane, but $m=\Theta^n(\Si, p)=\sum_{i=1}^m \Theta^n(\Si_i, p)$. Now applying the Allard regularity theorem \cite{Allard72} to each $\Si_i$, we get that $\Si_i$ extends as a $C^{1, \alpha}$ hypersurface across $p$. Higher regularity of $\Si_i$ follows from the prescribing mean curvature equation and elliptic regularity.  
\end{proof}

\section{Proof of Lemma \ref{L:index bound pre lemma2}}
\label{A:a calculus lemma}
\cite[Lemma 4.5]{Marques-Neves16} is purely a result in finite dimensional multi-variable calculus. Let us translate the problem as follows: let $\bB$ be some compact topological space with $0\in \bB$, and $\{f^\om\in C^\infty(\overline{B}^k): \om\in\bB\}$ be a family of smooth functions defined on $\overline{B}^k$, such that $\om \to f^\om$ is a continuous map in the smooth topology on $C^\infty(\overline{B}^k)$. Moreover we assume
\begin{itemize}
\item $f^\om$ has a unique maximum $m(\om)\in B^k_{c_0/\sqrt{10}}$, and $m(0)=0$;
\item $-\frac{1}{c_0}\Id\leq D^2 f^\om(u)\leq -c_0\Id$, for all $u\in \overline{B}^k$ and for some $c_0\in (0, 1)$.
\end{itemize}
So for each $\om\in \bB$, we have
\begin{equation}
\label{E:quadratic estimates for functions}
f^\om(m(\om))-\frac{1}{2c_0}|u-m(\om)|^2\leq  f^\om(u)\leq f^\om(m(\om))-\frac{c_0}{2}|u-m(\om)|^2
\end{equation}
for all $u\in\overline{B}^k$.

For each $f^\om$, consider the one-parameter flow $\{\phi^\om(\cdot, t): t\geq 0\}\subset \Diff(\overline{B}^k)$ generated by the vector field
\[ u\to -(1-|u|^2)\nabla f^\om (u),\quad u\in\overline{B}^k. \]
For fixed $u\in\overline{B}^k$, the function $t\to f^\om(\phi^\om(u, t))$ is non-increasing.

The prototype of \cite[Lemma 4.5]{Marques-Neves16} is the following lemma, and the proof is essentially the same as therein so we omit it.
\begin{lemma}
\label{L:a calculus lemma}
For any $\de<\frac{1}{4}$, there exists $T=T(\de, \bB, \{f^\om\}, c_0)\geq 0$ such that for any $\om \in \bB$ and $v\in \overline{B}^k$ with $|v-m(\om)|\geq \de$, we have
\[ f^\om(\phi^\om(v, T))<f^\om(0)-\frac{c_0}{10}\, \text{ and }\, |\phi^\om(v, T)|>\frac{c_0}{4}. \]
\end{lemma}

\vspace{1em}
Now we are ready to prove Lemma \ref{L:index bound pre lemma2}. Note that the ball $\overline{\bB}^\mF_{2\ep}(\Om_0)$ is not compact under the $\mF$-topology, so to apply Lemma \ref{L:a calculus lemma}, we need to introduce a compactification of $\overline{\bB}^\mF_{2\ep}(\Om_0)$.  
\begin{proof}[Proof of Lemma \ref{L:index bound pre lemma2}]
Given a $\mF$-Cauchy sequence $\{\Om_i\}\subset \overline{\bB}^\mF_{2\ep}(\Om_0)$, we denote $(V_\infty, \Om_\infty)\in \V_n(M)\times\C(M)$ as the limit such that $V_\infty=\lim_{i\to\infty}|\partial\Om_i|$ as varifolds and $\Om_\infty=\lim_{i\to\infty}\Om_i$ as Caccioppoli sets. If we define 
\[ \Ah_{\infty}(v)=\|(F_v)_\# V_\infty\|(M)-\int_{F_v(\Om_\infty)} h\, d\mH^{n+1},\, \text{ for each } v\in\overline{B}^k, \]
Then $\Ah_{\Om_i}$ converges smoothly to $\Ah_\infty$ as functions in $C^\infty(\overline{B}^k)$.

Now take $\bB$ as the union of $\overline{\bB}^\mF_{2\ep}(\Om_0)$ with the limits of the form $(V_\infty, \Om_\infty)$, $f^\Om=\Ah_\Om$ and $f^{(V_\infty, \Om_\infty)}=\Ah_\infty$, then Lemma \ref{L:index bound pre lemma2} follows from Lemma \ref{L:a calculus lemma}.
\end{proof}

\section{Existence of local PMC foliations}
\label{A:existence of local PMC foliations}

We recall the following classical result of White \cite[Appendix and Remark 2]{White87}. Note that the $\Ah$-functional can be locally expressed as the integration of an elliptic integrand.

\begin{proposition}
\label{P:existence of local PMC foliations}
Given a Riemannian metric $g$ in a neighborhood $U$ of $0\in\R^{n+1}$, %and a smooth function $h: U \to \R$, 
there exists an $\ep>0$, such that if $h: U \to \R$ is a smooth function with $\|h\|_{4, \al}<\ep$, $r<\ep$, and if
\[ w: B^n_r \subset \R^n \to \R\,  \text{ satisfies }  \|w\|_{2, \al}<\ep r, \]
then for each $t\in [-r, r]$, there exists a $C^{2, \al}$-function $v_t: B^n_r\to\R$ whose graph $G_t$ satisfies:
\[  H_{G_t} = h\vert_{G_t}, \]
(where $H_{G_t}$ is evaluated with respect to the upward pointing normal of $G_t$), and 
\[ v_t(x)= w(x)+t,\, \text{ if } x\in\partial B^n_r. \]
Furthermore, $v_t$ depends on $r, t, h, w$ in $C^1$ and the graphs $\{G_t: t\in [-r, r]\}$ forms a foliation.
\end{proposition}

%Add reference here
\bibliography{refs}
\bibliographystyle{plain}

\end{document}